\pdfoutput=1
\documentclass[a4paper,11pt]{amsart}
\usepackage[hmarginratio={1:1},vmarginratio={1:1},lmargin=80.0pt,tmargin=90.0pt]{geometry}
\usepackage{calligra}
\usepackage[T1]{fontenc}
\usepackage[ugly]{nicefrac}
\usepackage{mathtools}
\mathtoolsset{mathic}
\usepackage[final]{microtype}
\usepackage{multicol}
\usepackage{tabularx,tabu}
\usepackage{latexsym,exscale,enumitem,amsfonts,amssymb,mathtools}
\usepackage{amsmath,amsthm,amsfonts,amssymb,amscd,textcomp,bbm}
\usepackage{stmaryrd}
\SetSymbolFont{stmry}{bold}{U}{stmry}{m}{n}
\usepackage[normalem]{ulem}
\usepackage{thmtools}
\usepackage{etex}
\usepackage{fancybox}
\usepackage[table]{xcolor}
\usepackage{mathrsfs}
\DeclareMathAlphabet{\mathpzc}{OT1}{pzc}{m}{it}
\DeclareMathAlphabet{\mathscrbf}{OMS}{mdugm}{b}{n}
\usepackage[norelsize]{algorithm2e}

\definecolor{mycolor}{rgb}{0.9,0,0}
\definecolor{colormy}{rgb}{0.9,0,0}
\definecolor{nonecolor}{rgb}{0.9,0,0.4}
\definecolor{nonecolor2}{rgb}{0,0.3,0}
\definecolor{nonecolor3}{rgb}{0.8,0.8,0}
\definecolor{noneblue}{rgb}{0.1,0.1,0.7}
\definecolor{cupgray}{gray}{0.5}
\definecolor{paragray}{gray}{0.3}
\definecolor{fade}{gray}{0.5}
\definecolor{myblue}{rgb}{0,0,0.9}
\definecolor{mygreen}{rgb}{0,0.9,0}
\definecolor{myred}{rgb}{0.9,0,0}
\definecolor{myyellow}{rgb}{0.9,0.9,0}
\definecolor{flowcolor}{rgb}{0.9,0.3,0}
\definecolor{orchid}{RGB}{143,40,194}
\definecolor{lava}{RGB}{207,16,32}
\definecolor{mydarkblue}{RGB}{10,10,170}

\newcommand{\phantombox}{\,\tikz[baseline=-.05,scale=0.25]{\draw[nonecolor,fill=nonecolor, opacity=.4] (0,0) rectangle (1,1);}\,}

\usepackage{caption} 
\usepackage[all]{xy}
\SelectTips{cm}{}

\usepackage{tikz}
\tikzset{anchorbase/.style={baseline={([yshift=-0.5ex]current bounding box.center)}}}
\usetikzlibrary{decorations.markings}
\usetikzlibrary{decorations.pathreplacing}
\usetikzlibrary{arrows,shapes,positioning}
\tikzstyle directed=[postaction={decorate,decoration={markings,
    mark=at position #1 with {\arrow{>}}}}]
\tikzstyle rdirected=[postaction={decorate,decoration={markings,
    mark=at position #1 with {\arrow{<}}}}]

\newcommand{\C}{\mathbb{C}}
\newcommand{\Z}{\mathbb{Z}}
\newcommand{\R}{\mathbb{R}}

\newcommand{\K}{\mathbb{K}}

\newcommand{\typeA}{\mathrm{A}}
\newcommand{\typeD}{\mathrm{D}}

\newcommand{\placeholder}{{}_{-}}
\newcommand{\placeholderout}{{}_{-}}

\newcommand{\webcatf}{\boldsymbol{\mathrm{W}}}
\newcommand{\word}[1]{\vec{#1}}
\newcommand{\emptyword}{\emptyset}
\newcommand{\emptyweb}{\varnothing}
\newcommand{\invo}[1]{{-}{#1}}
\newcommand{\closure}[1]{\overline{#1}}

\newcommand{\clocal}[1]{#1^{\varepsilon}}

\newcommand{\Kmod}{\K\text{-}\boldsymbol{\mathrm{VS}}}
\newcommand{\TQFT}{\mathcal{T}}

\newcommand{\Hom}{\mathrm{Hom}}
\newcommand{\End}{\mathrm{End}}
\newcommand{\twoHom}{2\mathrm{Hom}}
\newcommand{\twoEnd}{2\mathrm{End}}

\newcommand{\biMod}[1]{#1\text{-}\boldsymbol{\mathrm{biMod}}}

\newcommand{\CUP}{\mathrm{CUP}}
\newcommand{\webalg}{\mathfrak{W}}
\newcommand{\M}{\boldsymbol{\mathrm{W}}}
\newcommand{\webMod}{\webalg\text{-}\boldsymbol{\mathrm{biMod}}}

\newcommand{\gwebalg}{g\mathfrak{W}}

\newcommand{\Cupbasis}[1]{\mathbb{B}(#1)}
\newcommand{\Cupbbasis}[2]{{}_{#1}\mathbb{B}{}_{#2}}
\newcommand{\Cupcbasis}[3]{{}_{#1}\mathbb{B}(#3){}_{#2}}

\newcommand{\cwebalg}{c\mathfrak{W}}
\newcommand{\opoint}{\mathtt{B}}
\newcommand{\mydot}{\bullet}

\newcommand{\dpath}[2]{{#1}\hspace*{-.02cm}\rightarrow\hspace*{-.02cm}{#2}}

\newcommand{\pguy}{{\color{nonecolor}\mathtt{p}}}
\newcommand{\npguy}{\mathtt{n}{\color{nonecolor}\mathtt{p}}}

\newcommand{\psaddle}{{\color{nonecolor}\mathtt{s}}}

\newcommand{\distt}{\pguy\mathrm{edge}}

\newcommand{\apc}{\npguy\mathrm{circ}}
\newcommand{\dist}{\npguy\mathrm{loop}}
\newcommand{\sdist}{\npguy\mathrm{sad}}

\newcommand{\stype}{\psaddle\mathrm{type}}

\newcommand{\upguys}{\npguy\mathrm{esci}}

\newcommand{\Hcase}{\mathtt{H}}
\newcommand{\Ccase}{\mathtt{C}}
\newcommand{\caseC}{\text{\reflectbox{$\mathtt{C}$}}}

\newcommand{\slt}{\mathfrak{sl}_2}
\newcommand{\glt}{\mathfrak{gl}_2}

\newcommand{\g}{\mathfrak{g}}
\newcommand{\bV}{\raisebox{0.03cm}{\mbox{\footnotesize$\textstyle{\bigwedge}$}}}

\newcommand{\foamslt}{\boldsymbol{\mathfrak{F}}^{\slt}}
\newcommand{\foamglt}{\boldsymbol{\mathfrak{F}}^{\glt}}

\newcommand{\arcalgintroa}{\mathfrak{A}^{\typeA}}

\newcommand{\arcalgintrod}{\mathfrak{A}^{\typeD}}
\newcommand{\arcalgintrodd}{\overline{\mathfrak{A}}^{\typeD}}

\newcommand{\funnyOa}{\mathcal{O}_0^{\typeA_{p} \times \typeA_{q}}(\mathfrak{gl}_{m}(\mathbb{C}))}
\newcommand{\funnyOd}{\mathcal{O}_0^{\typeA_{n-1}}(\mathfrak{so}_{2n}(\mathbb{C}))}

\newcommand{\isocomb}{\mathtt{comb}}
\newcommand{\isotop}{\mathtt{top}}
\newcommand{\isotops}{\overline{\mathtt{top}}}
\newcommand{\isosign}{\mathtt{sign}}

\newcommand{\isotopg}{g\mathtt{top}}

\newcommand{\coeff}{\mathrm{coeff}}

\newcommand{\Downn}{\vee}
\newcommand{\Upp}{\wedge}

\newcommand{\typeDbulletout}{\!\!
\raisebox{0.11cm}{
\xy
(0,0)*{
\includegraphics[scale=1]{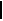}
};
\endxy
}
}

\newcommand{\X}{\mathtt{Bl}^{\circ}}
\newcommand{\bX}{\mathtt{Bl}^{\circ}}

\newcommand{\bx}{\mathtt{bl}^{\circ}}

\newcommand{\arcalg}{\mathfrak{A}}
\newcommand{\arcalgs}{\overline{\mathfrak{A}}}

\newcommand{\garcalg}{\mathfrak{C}}

\newcommand{\typeDmark}{\mathtt{u}}
\newcommand{\foammark}{{\color{nonecolor}\mathtt{m}}}

\newcommand{\length}{\typeDmark\mathrm{len}}
\newcommand{\unmarked}{\typeDmark\mathrm{type}}

\newcommand{\lengthf}{\foammark\mathrm{len}}
\newcommand{\unmarkedf}{\foammark\mathrm{type}}

\newcommand{\rpoint}{\mathtt{B}}

\newcommand{\pathd}[2]{{#1}\hspace*{-.02cm}\rightarrow\hspace*{-.02cm}{#2}}

\newcommand{\brank}{K}

\newcommand{\prefoam}{\boldsymbol{\mathrm{pF}}}
\newcommand{\foamf}{\boldsymbol{\mathfrak{F}}}
\newcommand{\emptyfoam}{f(\varnothing)}

\newcommand{\webhom}[1]{\TQFT(#1)}

\newcommand{\foamdot}{\bullet}

\newcommand{\foamg}{g\boldsymbol{\mathfrak{F}}}

\newcommand{\bl}{{\color{black}\mathtt{o}}}
\newcommand{\re}{{\color{nonecolor}\mathtt{p}}}

\theoremstyle{definition}
\newtheorem{theoremm}{Theorem}[section]

\declaretheorem[style=definition,name=Theorem,qed=$\square$,numberlike=theoremm]{theorem}
\declaretheorem[style=definition,name=Corollary,qed=$\blacksquare$,numberlike=theoremm]{corollary}
\declaretheorem[style=definition,name=Lemma,qed=$\square$,numberlike=theoremm]{lemma}
\declaretheorem[style=definition,name=Proposition,qed=$\square$,numberlike=theoremm]{proposition}

\declaretheorem[style=definition,name=Example,qed=$\blacktriangle$,numberlike=theorem]{example}
\declaretheorem[style=definition,name=Definition,qed=$\blacktriangle$,numberlike=theorem]{definition}
\declaretheorem[style=definition,name=Remark,qed=$\blacktriangle$,numberlike=theorem]{remark}
\declaretheorem[style=definition,name=Convention,qed=$\blacktriangle$,numberlike=theorem]{convention}
\declaretheorem[style=definition,name=A remark about diagrams and colors,qed=$\blacktriangle$,numberlike=theorem]{diagrams}
\declaretheorem[style=definition,name=Definition,numberlike=theorem]{definitionn}

\declaretheorem[style=definition,name=Lemma,numberlike=theoremm]{lemman}
\declaretheorem[style=definition,name=Proposition,numberlike=theoremm]{propositionn}
\declaretheorem[style=definition,name=Theorem,numberlike=theoremm]{theoremn}

\newtheorem{theoremmain}{Theorem}

\newtheorem{question}{Question}

\declaretheorem[style=definition,name=Beware,qed=$\bullet$,numberlike=theorem]{beware}

\newcommand{\closeqed}{\hfill\ensuremath{\blacktriangleright}}
\newcommand{\openqed}{\hfill\ensuremath{\vartriangleright}}
\newcommand{\qedmake}{\hfill\ensuremath{\square}}
\newcommand{\makeqed}{\hfill\ensuremath{\blacksquare}}
\newcommand{\makeqedtri}{\hfill\ensuremath{\blacktriangle}}

\numberwithin{equation}{section}
\usepackage[hypertexnames=false]{hyperref}
\usepackage{cleveref,bookmark}
\hypersetup{
    pdftoolbar=true,        
    pdfmenubar=true,        
    pdffitwindow=false,     
    pdfstartview={FitH},    
    pdftitle={Singular TQFTs, foams and type \texorpdfstring{$\typeD$}{D} arc algebras},    
    pdfauthor={Michael Ehrig, Daniel Tubbenhauer and Arik Wilbert},     
    pdfsubject={},   
    pdfcreator={Michael Ehrig, Daniel Tubbenhauer and Arik Wilbert},   
    pdfproducer={Michael Ehrig, Daniel Tubbenhauer and Arik Wilbert}, 
    pdfkeywords={}, 
    pdfnewwindow=true,      
    colorlinks=true,       
    linkcolor=mydarkblue,          
    citecolor=teal,        
    filecolor=magenta,      
    urlcolor=orchid,          
    linkbordercolor=lava,
    citebordercolor=teal,
    urlbordercolor=orchid,  
    linktocpage=true
}
\let\fullref\autoref
\def\makeautorefname#1#2{\expandafter\def\csname#1autorefname\endcsname{#2}}
\makeautorefname{equation}{Equation}%
\makeautorefname{footnote}{footnote}%
\makeautorefname{item}{item}%
\makeautorefname{figure}{Figure}%
\makeautorefname{table}{Table}%
\makeautorefname{part}{Part}%
\makeautorefname{appendix}{Appendix}%
\makeautorefname{chapter}{Chapter}%
\makeautorefname{section}{Section}%
\makeautorefname{subsection}{Section}%
\makeautorefname{subsubsection}{Section}%
\makeautorefname{paragraph}{Paragraph}%
\makeautorefname{subparagraph}{Paragraph}%
\makeautorefname{theorem}{Theorem}%
\makeautorefname{theo}{Theorem}%
\makeautorefname{thm}{Theorem}%
\makeautorefname{addendum}{Addendum}%
\makeautorefname{addend}{Addendum}%
\makeautorefname{add}{Addendum}%
\makeautorefname{maintheorem}{Main theorem}%
\makeautorefname{mainthm}{Main theorem}%
\makeautorefname{corollary}{Corollary}%
\makeautorefname{corol}{Corollary}%
\makeautorefname{coro}{Corollary}%
\makeautorefname{cor}{Corollary}%
\makeautorefname{lemma}{Lemma}%
\makeautorefname{lemm}{Lemma}%
\makeautorefname{lem}{Lemma}%
\makeautorefname{sublemma}{Sublemma}%
\makeautorefname{sublem}{Sublemma}%
\makeautorefname{subl}{Sublemma}%
\makeautorefname{proposition}{Proposition}%
\makeautorefname{proposit}{Proposition}%
\makeautorefname{propos}{Proposition}%
\makeautorefname{propo}{Proposition}%
\makeautorefname{prop}{Proposition}%
\makeautorefname{property}{Property}
\makeautorefname{proper}{Property}
\makeautorefname{scholium}{Scholium}%
\makeautorefname{step}{Step}%
\makeautorefname{conjecture}{Conjecture}%
\makeautorefname{conject}{Conjecture}%
\makeautorefname{conj}{Conjecture}%
\makeautorefname{question}{Question}
\makeautorefname{questn}{Question}
\makeautorefname{quest}{Question}
\makeautorefname{ques}{Question}
\makeautorefname{qn}{Question}
\makeautorefname{definition}{Definition}%
\makeautorefname{defin}{Definition}%
\makeautorefname{defi}{Definition}%
\makeautorefname{def}{Definition}%
\makeautorefname{dfn}{Definition}%
\makeautorefname{notation}{Notation}
\makeautorefname{nota}{Notation}
\makeautorefname{notn}{Notation}
\makeautorefname{remark}{Remark}%
\makeautorefname{rema}{Remark}%
\makeautorefname{rem}{Remark}%
\makeautorefname{rmk}{Remark}%
\makeautorefname{rk}{Remark}%
\makeautorefname{remarks}{Remarks}%
\makeautorefname{rems}{Remarks}%
\makeautorefname{rmks}{Remarks}%
\makeautorefname{rks}{Remarks}%
\makeautorefname{example}{Example}%
\makeautorefname{examp}{Example}%
\makeautorefname{exmp}{Example}%
\makeautorefname{exam}{Example}%
\makeautorefname{exa}{Example}%
\makeautorefname{algorithm}{Algorithm}%
\makeautorefname{algo}{Algorithm}%
\makeautorefname{alg}{Algorithm}%
\makeautorefname{axiom}{Axiom}%
\makeautorefname{axi}{Axiom}%
\makeautorefname{ax}{Axiom}%
\makeautorefname{case}{Case}%
\makeautorefname{claim}{Claim}%
\makeautorefname{clm}{Claim}%
\makeautorefname{assumption}{Assumption}%
\makeautorefname{assumpt}{Assumption}%
\makeautorefname{conclusion}{Conclusion}%
\makeautorefname{concl}{Conclusion}%
\makeautorefname{conc}{Conclusion}%
\makeautorefname{condition}{Condition}%
\makeautorefname{condit}{Condition}%
\makeautorefname{cond}{Condition}%
\makeautorefname{construction}{Construction}%
\makeautorefname{construct}{Construction}%
\makeautorefname{const}{Construction}%
\makeautorefname{cons}{Construction}%
\makeautorefname{criterion}{Criterion}%
\makeautorefname{criter}{Criterion}%
\makeautorefname{crit}{Criterion}%
\makeautorefname{exercise}{Exercise}%
\makeautorefname{exer}{Exercise}%
\makeautorefname{exe}{Exercise}%
\makeautorefname{problem}{Problem}%
\makeautorefname{problm}{Problem}%
\makeautorefname{probm}{Problem}%
\makeautorefname{prob}{Problem}%
\makeautorefname{solution}{Solution}%
\makeautorefname{soln}{Solution}%
\makeautorefname{sol}{Solution}%
\makeautorefname{summary}{Summary}%
\makeautorefname{summ}{Summary}%
\makeautorefname{sum}{Summary}%
\makeautorefname{operation}{Operation}%
\makeautorefname{oper}{Operation}%
\makeautorefname{observation}{Observation}%
\makeautorefname{observn}{Observation}%
\makeautorefname{obser}{Observation}%
\makeautorefname{obs}{Observation}%
\makeautorefname{ob}{Observation}%
\makeautorefname{convention}{Convention}%
\makeautorefname{convent}{Convention}%
\makeautorefname{conv}{Convention}%
\makeautorefname{cvn}{Convention}%
\makeautorefname{warning}{Warning}%
\makeautorefname{warn}{Warning}%
\makeautorefname{note}{Note}%
\makeautorefname{fact}{Fact}%
\makeautorefname{hope}{Hope}%

\begin{document}
\vbadness=10001
\hbadness=10001
\title[Singular TQFTs, foams and type $\typeD$ arc algebras]{Singular TQFTs, foams and type $\typeD$ arc algebras}
\author[M. Ehrig, D. Tubbenhauer and A. Wilbert]{Michael Ehrig, Daniel Tubbenhauer and Arik Wilbert}

\address{M.E.: Beijing Institute of Technology, School of Mathematics and Statistics, Liangxiang Campus of Beijing Institute of Technology, Fangshan District, 100488 Beijing, China, \href{https://michaelehrig.github.io/index.html}{https://michaelehrig.github.io/index.html}}
\email{michae.ehrig@outlook.com}

\address{D.T.: Institut f\"ur Mathematik, Universit\"at Z\"urich, Winterthurerstrasse 190, Campus Irchel, Office Y27J32, CH-8057 Z\"urich, Switzerland, \href{www.dtubbenhauer.com}{www.dtubbenhauer.com}}
\email{daniel.tubbenhauer@math.uzh.ch}

\address{A.W.: Department of Mathematics, University of Georgia, Athens, GA 30602, 526 Boyd GSRC, USA, \href{http://www.arik-wilbert.de}{http://www.arik-wilbert.de}}
\email{arik.wilbert@uga.edu}

\begin{abstract}
We combinatorially describe the 
$2$-category 
of singular cobordisms, called 
(rank one) foams, 
which governs the functorial version of Khovanov homology.
As an application we topologically realize
the type $\mathrm{D}$ arc algebra 
using this singular cobordism construction.
\end{abstract}

\maketitle
\vspace*{-.3cm}
\tableofcontents
\vspace*{-.3cm}
\renewcommand{\theequation}{\thesection-\arabic{equation}}
\section{Introduction}\label{sec:intro}
%%%%%%%%%%%%%%%%%%%%%%%%%%%%%%%%%%%%%%%%%%%%%%%%%%
\subsubsection*{Motivation}\label{subsec:intropart-1} 

In this paper we study the 
\textit{web algebra} $\webalg$ attached to $\glt$.
The algebra $\webalg$ naturally appears in the 
setup of singular TQFTs in the sense 
that a $2$-subcategory $\biMod{\webalg}$ of its bimodule $2$-category
is equivalent to the $2$-category $\foamf$ 
of certain singular surfaces 
{\`a} la Blanchet \cite{Bl1},
called foams,
and $\webalg$
algebraically controls the functorial version 
of Khovanov's link homology. 
The $2$-category $\foamf$ can be interpreted as a sign modified version 
of Bar-Natan's \cite{BN1} original cobordism 
(a.k.a. $\mathfrak{sl}_2$-foam) $2$-category 
attached to Khovanov's link and tangle invariant. 
The signs are crucial for 
making Khovanov's link homology functorial \cite{Bl1}, but very delicate 
to compute in practice.

Moreover, $\webalg$ contains the 
(type $\typeA$) 
\textit{arc algebra} $\arcalgintroa$, introduced 
by Khovanov \cite{Kh1}, as 
a subalgebra. The algebra $\arcalgintroa$ is related to 
the principal block of parabolic BGG category $\mathcal{O}$ of 
type $\typeA_m$ with parabolic of 
type $\typeA_{p}\times\typeA_{q}$ for $p+q=m$ \cite{BS3}, 
and can be constructed using convolution algebras 
and $2$-block Springer fibers of type $\typeA$ \cite{SW1}, giving us a connection to Lie theory 
and the geometry of Springer fibers. 

However, a \textit{combinatorial model} of $\webalg$ is missing, i.e. an 
algebra with an explicit basis and combinatorial multiplication rule on this basis, 
which is isomorphic to $\webalg$. 
This is thus far only known for certain subalgebras of $\webalg$, 
including $\arcalgintroa$, cf. \cite{BS1}, \cite{EST1}, \cite{EST2}.

\begin{question}\label{question-first}
Can one construct a combinatorial model of $\webalg$?\makeqedtri
\end{question}

Further, in joint work with Stroppel \cite{ES1}, the first author
has defined a type $\typeD$ generalization $\arcalgintrod$ of Khovanov's 
arc algebra, which we call the \textit{type} $\typeD$ \textit{arc algebra}. 
The algebra $\arcalgintrod$ is akin to Khovanov's algebra $\arcalgintroa$, and shares 
many of its features. For example, the algebra $\arcalg$ 
relates to the principal block of the parabolic BGG category $\mathcal{O}$ 
of type $\typeD_n$ with parabolic of type $\typeA_{n-1}$ \cite{ES1}. 
Secondly, $\arcalgintrod$ can be constructed using 
$2$-block Springer fibers of type $\typeD$, see \cite{ES2}, \cite{Wi1}.
However, $\arcalgintrod$ was defined in combinatorial terms, using 
so-called arc diagrams, and
no \textit{topological model}, i.e. an isomorphic algebra defined via 
a TQFT construction, is known so far.

\begin{question}\label{question-second}
Can one construct a topological model of $\arcalgintrod$?\makeqedtri
\end{question}

The purpose of our paper is to answer both questions affirmatively 
at the same time.

\subsubsection*{The main results in a few words}\label{subsec:intropart0}

In the first part of this paper we define an algebra $\cwebalg$
in terms of the combinatorics of so-called dotted webs.
Our first main result is then that $\cwebalg$ is a combinatorial model of $\webalg$ 
and $\foamf$, 
providing an answer to \fullref{question-first}:

\begin{theoremmain}\label{theorem-a}
There is an isomorphism of graded algebras
\abovedisplayskip0.35em
\belowdisplayskip0.5em
\begin{gather*}
\isocomb\colon\cwebalg
\overset{\cong}{\longrightarrow}
\webalg.
\end{gather*}
(Consequently, we obtain a combinatorial description 
of the foam $2$-category $\foamf$.)\qedmake
\end{theoremmain}

The answer to \fullref{question-second} is then given by:

\begin{theoremmain}\label{theorem-b}
There is an embedding of graded algebras
\abovedisplayskip0.35em
\belowdisplayskip0.5em
\begin{gather*}
\isotop\colon\arcalgintrod\xhookrightarrow{\phantom{f\circ}}\webalg.
\end{gather*}
Moreover, $\arcalgintrod$ is an idempotent 
truncation of $\webalg$ giving an embedding between the associated 
bimodule $2$-categories.\qedmake
\end{theoremmain}

Thus, the representation theory of 
the web algebra $\webalg$  
relates to different versions of category $\mathcal{O}$ 
and the geometry of Springer fibers. (A summary of the various connections is given in \fullref{fig:story}.)
However, an 
interpretation in terms of link invariants of 
this is still open.

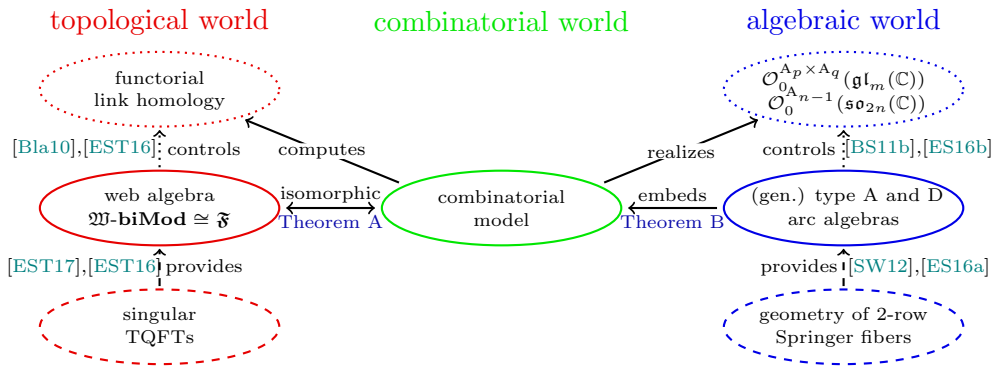
\begin{figure}[ht]
\[
\scalebox{.915}
{$
\begin{tikzpicture}[anchorbase, scale=.45]
	\draw [thick, myred] (-10,0) ellipse (3.5cm and 1.1cm);
	\node at (-10,.35) {\tiny $\text{web algebra}$};
	\node at (-10,-.35) {\tiny $\biMod{\webalg}\cong\foamf$};
	\draw [thick, mygreen] (0,0) ellipse (3.5cm and 1.1cm);
	\node at (0,.35) {\tiny $\text{combinatorial}$};
	\node at (0,-.35) {\tiny $\text{model}$};
	\draw [thick, myblue] (10,0) ellipse (3.5cm and 1.1cm);
	\node at (10,.35) {\tiny $\text{ (gen.) type }\typeA\text{ and }\typeD$};
	\node at (10,-.35) {\tiny $\text{arc algebras}$};
	\draw [thick, dotted, myred] (-10,3.5) ellipse (3.5cm and 1.1cm);
	\node at (-10,3.85) {\tiny $\text{functorial}$};
	\node at (-10,3.15) {\tiny $\text{link homology}$};
	\draw [thick, dotted, myblue] (10,3.5) ellipse (3.5cm and 1.1cm);
	\node at (10,3.85) {\tiny $\funnyOa$};
	\node at (10.1,3.15) {\tiny $\funnyOd$};
	\draw [thick, dashed, myred] (-10,-3.5) ellipse (3.5cm and 1.1cm);
	\node at (-10,-3.15) {\tiny $\text{singular}$};
	\node at (-10,-3.85) {\tiny $\text{TQFTs}$};
	\draw [thick, dashed, myblue] (10,-3.5) ellipse (3.5cm and 1.1cm);
	\node at (10,-3.15) {\tiny $\text{geometry of }2\text{-row}$};
	\node at (10,-3.85) {\tiny $\text{Springer fibers}$};
	\draw [thick, dotted, ->] (-10,1.2) to (-10,2.3);
	\node at (-8.7,1.75) {\tiny $\text{controls}$};
	\node at (-12.4,1.75) {\tiny $\text{\cite{Bl1},\cite{EST2}}$};
	\draw [thick, dotted, ->] (10,1.2) to (10,2.3);
	\node at (8.7,1.75) {\tiny $\text{controls}$};
	\node at (12.4,1.75) {\tiny $\text{\cite{BS3},\cite{ES1}}$};
	\draw [thick, dashed, ->] (-10,-2.3) to (-10,-1.2);
	\node at (-8.7,-1.75) {\tiny $\text{provides}$};
	\node at (-12.4,-1.75) {\tiny $\text{\cite{EST1},\cite{EST2}}$};
	\draw [thick, dashed, ->] (10,-2.3) to (10,-1.2);
	\node at (8.7,-1.75) {\tiny $\text{provides}$};
	\node at (12.4,-1.75) {\tiny $\text{\cite{SW1},\cite{ES2}}$};
	\draw [thick, <->] (-6.3,0) to (-3.7,0);
	\node at (-5,.4) {\tiny $\text{isomorphic}$};
	\node at (-5,-.4) {\tiny $\text{\fullref{theorem-a}}$};
	\draw [thick, ->] (6.3,0) to (3.7,0);
	\node at (5,.4) {\tiny $\text{embeds}$};
	\node at (5,-.4) {\tiny $\text{\fullref{theorem-b}}$};
	\draw [thick, ->] (-3,.75) to (-7.5,2.6);
	\draw [thick, white, fill=white] (-4.5,1.475) rectangle (-6,1.875);
	\node at (-5.25,1.675) {\tiny $\text{computes}$};
	\draw [thick, ->] (3,.75) to (7.5,2.6);
	\draw [thick, white, fill=white] (4.5,1.475) rectangle (6,1.875);
	\node at (5.25,1.675) {\tiny $\text{realizes}$};
	\node at (-10,5.5) {{\color{myred}$\text{topological world}$}};
	\node at (0,5.5) {{\color{mygreen}$\text{combinatorial world}$}};
	\node at (10,5.5) {{\color{myblue}$\text{algebraic world}$}};
\end{tikzpicture}
$}
\]
\caption{Our story in a nutshell.
}\label{fig:story}
\end{figure}

\subsubsection*{Upshot}\label{subsec:intropart01}

As a consequence of 
\fullref{theorem-a} we obtain a hands-on 
way to work with $\webalg$ and therefore 
with the functorial versions of Khovanov's link homology 
constructed from it.

As a consequence of 
\fullref{theorem-b} we immediately 
get that $\arcalgintrod$ is associative. 
This is not at all clear from the purely combinatorial definition 
and was proven in a quite involved way in \cite{ES1}.
In fact, as an intermediate step on our way to prove \fullref{theorem-b}, 
we define a sign adjusted version $\arcalgintrodd$ of $\arcalgintrod$, 
and an isomorphism of graded algebras
\abovedisplayskip0.35em
\belowdisplayskip0.5em
\[
\isosign\colon\arcalgintrod\overset{\cong}{\longrightarrow}\arcalgintrodd.
\] 
The definition of $\arcalgintrodd$
does not require any knowledge of singular 
TQFTs or foams, and has a simpler sign placement 
than the original type $\typeD$ arc algebra $\arcalgintrod$. 
But the construction of $\arcalgintrodd$ comes directly from topology
which eliminates the 
non-locality problem of the original definition in \cite{ES1}, cf. 
\fullref{remark:order-important}.
Then the isomorphism $\isotop$ in \fullref{theorem-b}
is given by
assembling the pieces, i.e. the following 
diagram defines it:
\abovedisplayskip0.35em
\belowdisplayskip0.5em
\[
\xymatrix{
\arcalgintrod\ar[r]^{\isosign}\ar@<-3pt>@/_/[rrr]_{\isotop}
&
\arcalgintrodd\ar[r]^{\isotops\phantom{c}}
&
\cwebalg\ar[r]^{\isocomb}
&
\webalg.
}
\]

\subsubsection*{The papers content in a nutshell}\label{subsec:intropart1}

In \fullref{sec:foams}
we explain webs, foams 
and singular topological quantum 
field theories (\textit{singular TQFTs}).
These are pieced together into the 
web algebra $\webalg$ and its representation theory 
in \fullref{sec:web-algebra}.
In \fullref{sec:comb-model} we give a combinatorial model 
by using what we call dotted webs.
In \fullref{sec:typeD} we recall the 
notions underlying the arc algebras, which is given 
by putting an algebraic multiplication structure on arc diagrams, 
and we also define the sign adjusted version and the embedding $\isotops$.
The summary of how these are connected is sketched in \fullref{fig:three-worlds}.

\begin{figure}[ht]
\abovedisplayskip0.1em
\belowdisplayskip0.1em
\[
\raisebox{0.1cm}{\xy
(0,0)*{\includegraphics[scale=1.2]{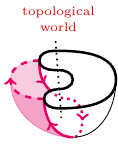}};
\endxy}
\leftrightsquigarrow
\xy
(0,0)*{\includegraphics[scale=1.2]{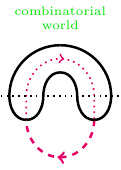}};
\endxy
\leftrightsquigarrow
\xy
(0,0)*{\includegraphics[scale=1.2]{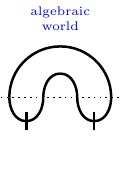}};
\endxy
\]
\caption{From foams to dotted webs to arc diagrams.
}\label{fig:three-worlds}
\end{figure}

For readability, the proofs requiring involved combinatorics and calculations are moved 
to \fullref{sec:proofs}, 
which is the technical heart of the paper.

Finally, using the subquotient construction 
explained e.g. in \cite[Section 5.1]{EST2},
one can immediately generalize the web algebra 
and one obtains an algebra in which the quasi-hereditary cover of $\arcalgintrod$, 
the so-called generalized type $\typeD$ arc algebra, embeds, cf. \fullref{remark:gen-as-well}. 
For these generalizations all of our constructions can be, mutatis mutandis, 
repeated and our results hold verbatim.
%%%%%%%%%%%%%%%%%%%%%%%%%%%%%%%%%%%%%%%%%%%%%%%%%%
\newline
\medskip
\newline

\noindent\textbf{Acknowledgments:}
We like to thank Jonathan Grant, Catharina Stroppel 
and Paul Wedrich for conversations about foams 
and arc algebras, Kevin Coulembier 
for a discussion about potential ``super webs'' underlying 
our foams, and the referee for a careful reading of the manuscript. 
We also thank Paul Wedrich for reminding us of cookie-cutters, 
and a nameless toilet paper roll for illustrating the corresponding strategy.

M.E. was partially supported by the Australian Research Council Grant DP150103431. 
A.W. was partially supported by a Hausdorff scholarship 
of the Bonn International Graduate School and a 
scholarship of the International Max Planck 
Research School. This paper is part of the third author's Ph.D. thesis.
\medskip
\begin{convention}\label{convention:first}
Throughout we work over a field $\K$ 
of arbitrary characteristic, and dimension is always 
meant with respect to $\K$. 
There are two exceptions: 
our proof of \fullref{proposition:cats-are-equal-yes} requires $\K=\overline{\K}$ 
(this can be avoided, but extends the proof considerably), 
while all connections to category $\mathcal{O}$ 
work over $\K=\C$ only.
Apart from these instances, working over $\Z$ 
is entirely possible. 

All algebras are assumed to be $\K$-algebras, but not necessarily associative nor
finite-dimensional nor unital.
(All the algebras which we use in this paper are associative, but this is, 
except for the web algebra, a non-trivial fact.) We abbreviate $\Z$-graded by graded and 
adopt the same conventions as in \cite[Conventions 1.1 and 1.2]{EST1} 
for the graded (finite-dimensional) representation theory of a graded algebra. 
In particular, graded biprojective means 
graded left and right projective, and $\{\cdot\}$ denotes 
grading shifts, with conventions as fixed in \fullref{convention:graded} below.
\end{convention}

\begin{convention}\label{convention:graded}
An additive, graded, $\K$-linear $2$-category is a category 
enriched over the category of additive, graded, $\K$-linear categories. 
(We only use small categories and $2$-categories in this paper.)
Additionally, in our setup, the morphisms of such a 
$2$-category admit grading shifts.
That is, given any morphism $X$ and any $s\in\Z$, 
there is a morphism $X\{s\}$ such that 
the identity $2$-morphism on $X$ gives rise to a degree $s$ homogeneous 
$2$-isomorphism from $X$ to $X\{s\}$.
General $2$-morphisms in such $2$-categories
are $\K$-linear combinations of homogeneous ones. 
Hereby, any $2$-morphism of degree $d$ between $X$ and $Y$ 
becomes a $2$-morphism of degree $d-s+t$ 
between $X\{s\}$ and $Y\{t\}$. 
\end{convention}

\begin{diagrams}\label{diagrams:colors} 
We read all diagrams from bottom to top and from left to right, 
and we often illustrate only local pieces.

Regarding colors: The important colors are 
the reddish (which appear as $\phantombox$) so-called phantom edges and facets of webs and foams. 
In a black-and-white version these can be 
distinguished since phantom edges are dashed and phantom facets are shaded.
\end{diagrams}

%%%%%%%%%%%%%%%%%%%%%%%%%%%%%%%%%%%%%%%%
%%%                                  %%%
%%%        Section 2                 %%%
%%%                                  %%%
%%%%%%%%%%%%%%%%%%%%%%%%%%%%%%%%%%%%%%%%
\section{Singular TQFTs and foams}\label{sec:foams}
%%%%%%%%%%%%%%%%%%%%%%%%%%%
In the present section we briefly recall the topological 
construction of foams via 
the singular TQFT approach outlined in \cite[Section 2]{EST1} 
and \cite[Section 2]{EST2}.  
We assume some familiarity 
with the foam construction and techniques used 
therein, but in order 
for our paper to be reasonably self-contained, we will now recall the 
most important aspects of the theory of foams.

In short, foams are constructed in three steps. In 
step one we construct their boundary, called webs. 
Prefoams are then given by certain singular 
cobordisms between these webs. In the final step we linearize 
and take a quotient that naturally arises from 
relations coming from a singular TQFT.

\subsection{Webs and prefoams}\label{subsec:reminder-webs-prefoam}

\subsubsection*{The boundary of foams} We start with step one.

\begin{definition}\label{definition:web}
A \textit{web} is a labeled, piecewise linear, 
one-dimensional CW complex (a graph with vertices 
and edges) embedded in $\R^2\times\{z\}\subset\R^3$ 
for some fixed $z\in\R$
with boundary supported in two horizontal lines, 
such that all horizontal slices consists only of a finite number
 of points.
(Hence, we can talk about the bottom and top boundary of webs.) 
Each vertex is either internal and of valency three, 
or a boundary vertex of valency one.

We assume that each edge carries a label from $\{\bl,\re\}$ 
(we say they are \textit{colored} by $\bl$ or $\re$).
Moreover, the $\re$-colored edges are assumed 
to be oriented, and each internal vertex has precisely one attached edge which is $\re$-colored.
By convention, the empty web $\emptyweb$ is also a web, and we allow 
circle components which consist of edges only.
Webs are considered modulo boundary preserving isotopies 
in $\R^2\times\{z\}$.
\end{definition}

Throughout we consider, not just for webs, labelings 
with $\bl$ or $\re$
and always illustrate them directly as colors using the 
convention that a reddish color means $\re$.
Moreover, both, webs and (pre)foams as defined below, contain 
$\re$-colored edges/facets. 
We call, everything related to 
these $\re$-colored edges/facets \textit{phantom} (illustrated reddish, dashed), 
anything else \textit{ordinary} (illustrated black).

\begin{example}\label{example:web}
Using these conventions, such webs are for example locally of the form:
\begin{gather*}
\raisebox{.075cm}{\xy
(0,0)*{
\includegraphics[scale=1.2]{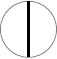}
\quad\raisebox{.4cm}{,}\quad
\includegraphics[scale=1.2]{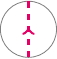}};
(0,9.5)*{\text{\small\phantom{ordinary}}};
(0,-9.5)*{\text{\small identities}};
\endxy}
\quad,\quad
\xy
(0,0)*{
\includegraphics[scale=1.2]{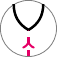}};
(0,9.5)*{\text{\small\phantom{ordinary}}};
(0,-9.5)*{\text{\small split}};
\endxy
\quad,\quad
\xy
(0,0)*{
\includegraphics[scale=1.2]{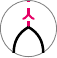}
};
(0,9.5)*{\text{\small\phantom{ordinary}}};
(0,-9.8)*{\text{\small merge}};
\endxy
\\
\,
\xy
(0,0)*{
\includegraphics[scale=1.2]{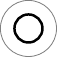}};
(0,9.5)*{\text{\small\phantom{ordinary}}};
(0,-9.5)*{\text{\small ordinary circle}};
\endxy
\;,\quad
\xy
(0,0)*{
\includegraphics[scale=1.2]{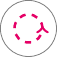}
\quad\raisebox{.4cm}{,}\quad
\includegraphics[scale=1.2]{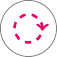}};
(0,9.5)*{\text{\small\phantom{ordinary}}};
(0,-9.5)*{\text{\small {\color{nonecolor}phantom} circles}};
\endxy
\end{gather*}
Here the outer circle indicates that these are local pictures.
(We omit it in what follows and hope no confusion can arise.)
\end{example}

\begin{definition}\label{definition:webs-monoidal}
Let $\webcatf$ be the \textit{monoidal category of webs} 
given as follows:
\smallskip
\begin{enumerate}[label=(\roman*)]

\setlength\itemsep{.15cm}

\item Objects are finite
words $\word{k}$ in $\{\bl,\re,\invo{\re}\}$. 
(The empty word $\emptyword$ is also allowed.)

\item The morphisms spaces $\Hom_{\webcatf}(\word{k},\word{l})$ 
are given by all webs with bottom boundary $\word{k}$ 
and top boundary $\word{l}$ using the following local conventions 
(read from bottom to top):
\begin{gather}\label{eq:web-gens}
\begin{aligned}
&
\raisebox{0.1cm}{
\xy
(0,0)*{
\raisebox{.025cm}{\includegraphics[scale=1.2]{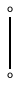}}
\quad\raisebox{.65cm}{,}\quad
\includegraphics[scale=1.2]{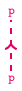}
\quad\raisebox{.65cm}{,}\quad
\includegraphics[scale=1.2]{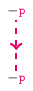}};
(0,10.5)*{\text{\small\phantom{caps}}};
(0,-10.5)*{\text{\small identities}};
\endxy}
\quad,\quad
\xy
(0,0)*{
\includegraphics[scale=1.2]{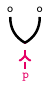}
\quad\raisebox{.65cm}{,}\quad
\includegraphics[scale=1.2]{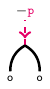}};
(0,10.5)*{\text{\small\phantom{caps}}};
(0,-10.5)*{\text{\small splits}};
\endxy
\quad,\quad
\xy
(0,0)*{
\includegraphics[scale=1.2]{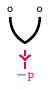}
\quad\raisebox{.65cm}{,}\quad
\includegraphics[scale=1.2]{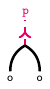}};
(0,10.5)*{\text{\small\phantom{caps}}};
(0,-10.5)*{\text{\small merges}};
\endxy
\\
&
\raisebox{0.1cm}{
\xy
(0,0)*{
\includegraphics[scale=1.2]{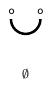}
\quad\raisebox{.65cm}{,}\quad
\includegraphics[scale=1.2]{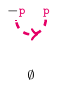}
\quad\raisebox{.65cm}{,}\quad
\includegraphics[scale=1.2]{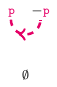}};
(0,10.5)*{\text{\small\phantom{caps}}};
(0,-10.5)*{\text{\small cups}};
\endxy}
\quad,\quad
\xy
(0,0)*{
\raisebox{.05cm}{\includegraphics[scale=1.2]{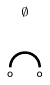}}
\quad\raisebox{.65cm}{,}\quad
\includegraphics[scale=1.2]{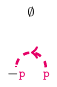}
\quad\raisebox{.65cm}{,}\quad
\includegraphics[scale=1.2]{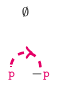}};
(0,10.5)*{\text{\small\phantom{caps}}};
(0,-10.5)*{\text{\small caps}};
\endxy
\end{aligned}
\end{gather}

\item The composition $uv=v\circ u$ is the 
evident gluing of $v$ on top of $u$, and 
monoidal product $\word{k}\otimes\word{l}$ or $u\otimes v$ 
given by putting $\word{k}$ or $u$ to the 
left of $\word{l}$ or $v$.

\end{enumerate}
\smallskip
A \textit{closed} web $\closure{w}$ 
is an endomorphism of the empty word $\emptyword$, i.e. 
$\closure{w}\in\End_{\webcatf}(\emptyword)$.
\end{definition}

The webs depicted in \eqref{eq:web-gens} are called 
\textit{identities}, \textit{splits}, 
\textit{merges}, \textit{cups} and 
\textit{caps}, and the latter four 
are the 
monoidal generators 
of $\webcatf$.

We use the topological and the algebraic 
notion of webs interchangeably 
(e.g. the generators from \eqref{eq:web-gens} 
are allowed to have their 
boundary points far apart in the sense that these need not be neighbored).

For later use, we denote by ${}^{\ast}$ the involution that mirrors a web 
along the top horizontal line and reverses orientations. 
Moreover, since the objects of $\webcatf$ can be read off from 
the webs, we omit to indicate them.

\begin{remark}\label{remark:gl-vs-sl-foams}
The reader familiar 
with \cite{EST1} 
or \cite{EST2} may note that our webs are slightly different from the 
ones considered in the aforementioned articles 
(see \fullref{remark:nogl2-webs} for a detailed comparison). However, 
the differences do not affect the construction of foams.
\end{remark}

\subsubsection*{Prefoams}

We briefly recall the notion of prefoams.
A \textit{closed prefoam} $\closure{f}$ is a singular surface obtained by gluing the 
boundary circles of a given set of orientable, compact, two-dimensional real 
surfaces. Some of these surfaces are called \textit{phantom surfaces} (those are 
colored reddish in the following) and we always glue along three circles, where 
exactly one of the circles comes from a boundary component of a phantom surface. 
Closed prefoams are assumed to be embedded in $\mathbb R^3$.

Note that the singularities which come from the gluing of three circles, called \textit{singular seams}, are locally 
of the form
\begin{gather}\label{eq:orientation}
\xy
(0,0)*{
\includegraphics[scale=1.2]{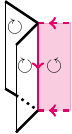}
\;
\raisebox{1cm}{$\colon$}
\raisebox{.65cm}{\includegraphics[scale=1]{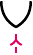}}
\raisebox{1cm}{$\,\to\,$}
\raisebox{.65cm}{\includegraphics[scale=1]{figs/2-22.pdf}}
};
(0,18)*{\text{\small\phantom{consistent}}};
(0,-18)*{\text{\small split}};
\endxy
\quad,\quad
\xy
(0,0)*{
\includegraphics[scale=1.2]{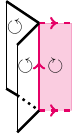}
\;
\raisebox{1cm}{$\colon$}
\raisebox{.65cm}{\includegraphics[scale=1]{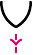}}
\raisebox{1cm}{$\,\to\,$}
\raisebox{.65cm}{\includegraphics[scale=1]{figs/2-24.pdf}}
};
(0,18)*{\text{\small\phantom{consistent}}};
(0,-18)*{\text{\small merge}};
\endxy
\quad,\quad
\xy
(0,0)*{
\includegraphics[scale=1.2]{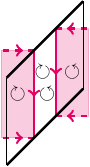}};
(0,18)*{\text{\small\phantom{consistent}}};
(0,-17.5)*{\text{\small no consistent choice}};
\endxy
\end{gather} 
Hereby we stress that we only consider those prefoams which 
can be embedded into $\R^3$ such that there is a choice of orientation 
of its facets as illustrated in \eqref{eq:orientation} (we fix this orientation); 
this choice of orientation is consistent in the sense that it induces orientations 
on the singular seams.
Moreover, prefoams are allowed to carry a finite number of markers on each connected component 
which we call \textit{dots} $\foamdot$ and which we illustrate as in \eqref{eq:usual2}.

\begin{remark}\label{remark:prefoams1b}
Due to these orientation conventions, there 
are no prefoams bounding
closed webs with an odd number of trivalent vertices. 
There are also no prefoams bounding a local situation which 
has \textit{ill-attached} phantom edges (cf.~\eqref{eq:orientation}), i.e.:
\[
\xy
(0,0)*{
\includegraphics[scale=1.2]{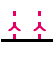}
\quad\raisebox{.7cm}{,}\quad
\includegraphics[scale=1.2]{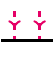}
\quad\raisebox{.7cm}{,}\quad
\includegraphics[scale=1.2]{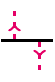}
\quad\raisebox{.7cm}{,}\quad
\includegraphics[scale=1.2]{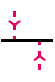}};
(0,-8.5)*{\text{\small ill-attached {\color{nonecolor}phantom} edges}};
\endxy
\]
All other local situations are said to have \textit{well-attached} 
phantom edges (see \eqref{eq:outgoing-pair}).
In contrast, there might be closed webs 
with an odd number of trivalent vertices 
or ill-attached phantom edges, but these play no role for us.
\end{remark}

\begin{remark}\label{remark:prefoams2}
Let $P^{\pm 1}_{xy}$ be the plane spanned by the first two coordinates in $\R^3$, 
embedded such that the third coordinate is $\pm 1$.
A (non-necessarily closed) \textit{prefoam} $f$ is the intersection 
of $\R^2\times[-1,+1]$ with some closed prefoam $\closure{f}$ such that 
$P^{\pm 1}_{xy}$ intersects $\closure{f}$ generically, now with orientation 
on its boundary
induced as in \eqref{eq:orientation}:
the orientation on the phantom facets agrees with the orientation 
on the phantom edges of the webs which we view as being 
the target sitting at the top 
and disagrees at the bottom. Clearly it suffices to 
indicate the orientations of the singular seams 
and we do so in the following.
In particular, the orientation of 
the singular seams point into splits and out of merges at the bottom of a prefoam.
\end{remark}

Next, step two: The bottom and top of a prefoam $f$ are webs $\closure{w}_b$ 
and $\closure{w}_t$, and we see $f$ 
as a cobordism from $\closure{w}_b$ to $\closure{w}_t$, 
as indicated in \eqref{eq:orientation}.
Using the cobordisms description, 
the whole data assembles into a monoidal category 
which we denote by $\prefoam$:

\begin{definition}\label{definition:prefoams-monoidal}
Let $\prefoam$ be the \textit{monoidal category of prefoams} 
given as follows:
\smallskip
\begin{enumerate}[label=(\roman*)]

\setlength\itemsep{.15cm}

\item Objects are closed 
webs $\closure{w}_b$ and $\closure{w}_t$ embedded in $\R\times\{-1\}$ 
respectively $\R\times\{+1\}$.

\item Morphisms 
are prefoams 
$f\colon\closure{w}_b\to\closure{w}_t$. 
(Including the empty prefoam $\emptyfoam$.)

\item Composition is the evident gluing, the monoidal structure 
is given by juxtaposition.\qedhere

\end{enumerate}
\end{definition}

Note that $\prefoam$ is symmetric monoidal, which can be seen 
by copying \cite[Section 1.4]{Ko1}.

\subsection{Obtaining relations via singular TQFTs}\label{subsec:relations}

Next, we recall a singular TQFT and discuss the relations 
in its kernel which play a crucial role for the definition of foams.

\subsubsection*{Singular TQFTs}

Let $\mathrm{F}$ be a (finite-dimensional, commutative, associative, unital) Frobenius algebra. 
Recall that $\mathrm{F}$ has a non-degenerate trace from $\mathrm{tr}$, and 
an associated TQFT (functor), see e.g. \cite{Ko1} for details.

For us this is needed as follows:
Given a closed prefoam $\closure{f}$, we 
can assign to it an element $\TQFT^{\star}(\closure{f})\in\K$. 
This element is obtained by first decomposing $\closure{f}$ into 
its ordinary as well as phantom pieces. The we apply the TQFT associated to the Frobenius 
algebra $\mathrm{F}_{\bl}=\K[X]/(X^2)$ with trace 
$\mathrm{tr}_{\bl}(1)=0,\mathrm{tr}_{\bl}(X)=1$ to the ordinary parts, 
and the TQFT associated 
to $\mathrm{F}_{\re}=\K$ with trace 
$\mathrm{tr}_{\re}(1)=-1$ to the phantom parts. (Note the minus sign.)
Following \cite[\S 2.2]{EST1}, the results 
can then be assembled into an \textit{evaluation of $\closure{f}$}, i.e. 
a value $\TQFT^{\star}(\closure{f})\in\K$.

Let $\closure{w}$ be a closed web and 
let $\K\{\Hom_{\prefoam}(\emptyweb,\closure{w})\}$ be 
the free $\K$-vector space with basis given by all prefoams from $\emptyweb$ 
to $\closure{w}$. We obtain a pairing 
\[
\beta\colon
\K\{
\Hom_{\prefoam}(\emptyweb,\closure{w})
\}
\times
\K\{
\Hom_{\prefoam}(\emptyweb,\closure{w})
\}
\to
\K
\]
by gluing a pair of two prefoams along their common boundary $\closure{w}$ and 
applying $\TQFT^{\star}$ to the resulting closed prefoam. 
Let $\mathrm{rad}(\beta)$ denote its radical, and let 
\[
\TQFT(\closure{w})
=
\K\{
\Hom_{\prefoam}(\emptyweb,\closure{w})
\}
/
\mathrm{rad}(\beta).
\]

\begin{lemma}\label{lemma:sing-TQFT}
$\TQFT(\closure{w})$ is finite-dimensional
for each closed web $\closure{w}$.
\end{lemma}

The proof of \fullref{lemma:sing-TQFT} is given in \fullref{sec:proofs},
but the point is the existence of
a symmetric monoidal functor from 
$\prefoam$ to the symmetric monoidal category $\Kmod$
of finite-dimensional $\K$-vector spaces. That is, we have the 
following theorem which can be proven as in \cite[Theorem 2.11]{EST1} or \cite[Theorem 2.10]{EST2}, 
using \fullref{lemma:sing-TQFT}:

\begin{theoremn}\label{theorem:sing-TQFT}
There exists a symmetric monoidal functor 
$\TQFT\colon\prefoam\to\Kmod$ which maps a 
closed web $\closure{w}$ to $\TQFT(\closure{w})$.
A prefoam $f\colon\closure{w}_b\to\closure{w}_t$ is sent to the $\K$-linear map 
$\TQFT(f)\colon\TQFT(\closure{w}_b)\to\TQFT(\closure{w}_t)$ 
obtained by composing with this prefoam.\makeqed
\end{theoremn}

$\TQFT$ is called 
a \textit{singular TQFT}, and its construction is based on ideas from \cite{BHMV1} 
(the so-called \textit{universal construction}), as well as \cite{Bl1}, 
which we sketched above.

\begin{example}\label{example:sing-TQFT}
The following examples
\begin{gather*}
\TQFT\left(
\xy
(0,0)*{\includegraphics[scale=1.2]{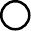}};
\endxy
\right)
\cong
\TQFT\left(
\xy
(0,0)*{\includegraphics[scale=1.2]{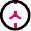}};
\endxy
\right)
\cong
\TQFT\left(
\xy
(0,0)*{\includegraphics[scale=1.2]{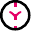}};
\endxy
\right)
\cong\K[X]/(X^2),
\\
\TQFT\left(
\xy
(0,0)*{\includegraphics[scale=1.2]{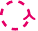}};
\endxy
\!\right)
\cong
\TQFT\left(
\xy
(0,0)*{\includegraphics[scale=1.2]{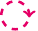}};
\endxy
\!\right)\cong\K,
\end{gather*}
play a crucial role and actually determine $\TQFT$ completely.
\end{example}

\subsubsection*{Various foamy relations}

We say that a relation $a_1f_1+\cdots+a_kf_k=b_1g_1+\cdots+b_lg_l$ 
between formal, finite, $\K$-linear 
combinations of prefoams lies \textit{in the kernel of} $\TQFT$, if 
$a_1\TQFT(f_1)+\cdots+a_k\TQFT(f_k)=b_1\TQFT(g_1)+\cdots+b_l\TQFT(g_l)$ 
holds as $\K$-linear maps.
Here are the first examples:

\begin{lemman}(See \cite[Lemmas 2.9 and 2.13]{EST1}.)\label{lemma:usual-rels}
The following 
relations
\begin{gather}\label{eq:usual1}
\xy
(0,6.5)*{
\includegraphics[scale=1.2]{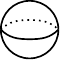}\;\;
\raisebox{.4cm}{$=\,0$}};
(0,-6.5)*{
\includegraphics[scale=1.2]{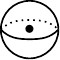}\;\;
\raisebox{.4cm}{$=\,1$}};
(0,11)*{\phantom{a}};
(0,-11)*{\phantom{a}};
\endxy
\end{gather}
\begin{gather}\label{eq:usual2}
\xy
(0,0)*{
\includegraphics[scale=1.2]{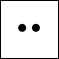}};
(0,11)*{\phantom{a}};
(0,-11)*{\phantom{a}};
\endxy
\,=\,0
\end{gather}
\begin{gather}\label{eq:usual3}
\xy
(0,0)*{
\includegraphics[scale=1.2]{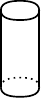}};
(0,11)*{\phantom{a}};
(0,-11)*{\phantom{a}};
\endxy
\,=\,
\xy
(0,0)*{
\includegraphics[scale=1.2]{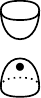}};
(0,11)*{\phantom{a}};
(0,-11)*{\phantom{a}};
\endxy
\,+\,
\xy
(0,0)*{
\includegraphics[scale=1.2]{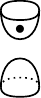}};
(0,11)*{\phantom{a}};
(0,-11)*{\phantom{a}};
\endxy
\end{gather}
called (from left to right) \textit{ordinary sphere, dot and neck cut} relations, 
as well as the \textit{phantom sphere, dot and neck cut} relations
\begin{gather}\label{eq:phantom1}
\xy
(0,0)*{
\includegraphics[scale=1.2]{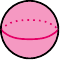}};
(0,8)*{\phantom{a}};
(0,-8)*{\phantom{a}};
\endxy
\,=\,-1
\end{gather} 
\begin{gather}\label{eq:phantom2}
\xy
(0,0)*{
\includegraphics[scale=1.2]{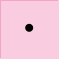}};
(0,8)*{\phantom{a}};
(0,-8)*{\phantom{a}};
\endxy
\,=
\,-\,
\xy
(0,0)*{
\includegraphics[scale=1.2]{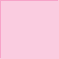}};
(0,8)*{\phantom{a}};
(0,-8)*{\phantom{a}};
\endxy
\,
\end{gather} 
\begin{gather}\label{eq:phantom3}
\xy
(0,0)*{
\includegraphics[scale=1.2]{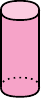}};
(0,8)*{\phantom{a}};
(0,-8)*{\phantom{a}};
\endxy
\,=\,-\,
\xy
(0,0)*{
\includegraphics[scale=1.2]{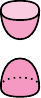}};
(0,8)*{\phantom{a}};
(0,-8)*{\phantom{a}};
\endxy
\end{gather}
are in the kernel of $\TQFT$. Similarly, if one considers 
ordinary and phantom parts separately, then 
all the usual TQFT relations (i.e. \textit{Frobenius isotopies}) as e.g. 
illustrated in \cite[Section 1.4]{Ko1} 
are in the kernel of $\TQFT$.
Furthermore, the \textit{theta foam} relations
\begin{gather}\label{eq:usual-rel-theta}
\xy
(0,0)*{
\includegraphics[scale=1.2]{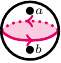}};
\endxy
\,=\,
\begin{cases}+1,& \text{if }a=1,b=0,\\ -1,& \text{if }a=0,b=1,\\ 0,&\text{otherwise},
\end{cases}
\end{gather}
\begin{gather}\label{eq:usual-rel-theta2}
\xy
(0,0)*{
\includegraphics[scale=1.2]{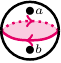}};
\endxy
\,=\,
\begin{cases}-1,& \text{if }a=1,b=0,\\ +1,& \text{if }a=0,b=1,\\ 0,&\text{otherwise},
\end{cases}
\end{gather}
are also in the kernel of $\TQFT$.\makeqed
\end{lemman}

Note that \eqref{eq:usual-rel-theta} and \eqref{eq:usual-rel-theta2} 
are the same relation, 
but reading bottom to top the orientation of the singular seam is reversed 
when comparing \eqref{eq:usual-rel-theta} to \eqref{eq:usual-rel-theta2}, 
which gives an asymmetry.

\begin{lemman}(See \cite[Lemma 2.14]{EST1} and \cite[(19)]{EST2}.)\label{lemma:special-rels}
The following relations 
are in the kernel of $\TQFT$. The 
\textit{dot moving} relations

\begin{gather}\label{eq:dotmigration}
\raisebox{.075cm}{\reflectbox{
\xy
(0,0)*{
\includegraphics[scale=1.2]{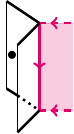}};
\endxy
}}
\,=\,-\,
\raisebox{.075cm}{\reflectbox{
\xy
(0,0)*{
\includegraphics[scale=1.2]{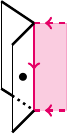}};
\endxy
}}
\quad,\quad
\xy
(0,0)*{
\includegraphics[scale=1.2]{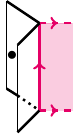}};
\endxy
\,=\,-\,
\xy
(0,0)*{
\includegraphics[scale=1.2]{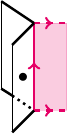}};
\endxy
\end{gather}
the \textit{singular sphere removal} relations
\begin{gather}\label{eq:s-sphere}
\xy
(0,0)*{
\includegraphics[scale=1.2]{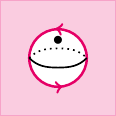}};
\endxy
\,=\,
\xy
(0,0)*{
\includegraphics[scale=1.2]{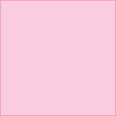}};
\endxy
\,=\,-\,
\xy
(0,0)*{
\includegraphics[scale=1.2]{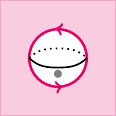}};
\endxy
\end{gather}
the \textit{singular neck cutting} and the \textit{closed seam removal} relations 
(the shaded dots in \eqref{eq:s-sphere} and \eqref{eq:neck} 
sit on the facets in the back)
\begin{gather}\label{eq:neck}
\xy
(0,0)*{
\includegraphics[scale=1.2]{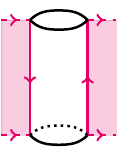}};
(0,14)*{\phantom{a}};
(0,-14)*{\phantom{a}};
\endxy
\,=\,
\xy
(0,0)*{
\includegraphics[scale=1.2]{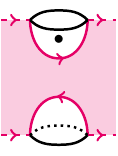}};
(0,14)*{\phantom{a}};
(0,-14)*{\phantom{a}};
\endxy
\,-\,
\xy
(0,0)*{
\includegraphics[scale=1.2]{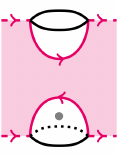}};
(0,14)*{\phantom{a}};
(0,-14)*{\phantom{a}};
\endxy
\end{gather} 
\begin{gather}\label{eq:sing-neck}
\xy
(0,0)*{
\includegraphics[scale=1.2]{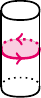}};
(0,14)*{\phantom{a}};
(0,-14)*{\phantom{a}};
\endxy
\,=\,
\xy
(0,0)*{
\includegraphics[scale=1.2]{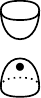}};
(0,14)*{\phantom{a}};
(0,-14)*{\phantom{a}};
\endxy
-
\xy
(0,0)*{
\includegraphics[scale=1.2]{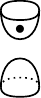}};
(0,14)*{\phantom{a}};
(0,-14)*{\phantom{a}};
\endxy
\end{gather}
the \textit{ordinary-to-phantom neck cutting} and the 
\textit{ordinary squeeze} relations
\begin{gather}\label{eq:op-neck}
\xy
(0,0)*{
\includegraphics[scale=1.2]{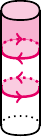}};
(0,14)*{\phantom{a}};
(0,-14)*{\phantom{a}};
\endxy
\,=\,-\,
\xy
(0,0)*{
\includegraphics[scale=1.2]{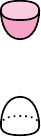}};
(0,14)*{\phantom{a}};
(0,-14)*{\phantom{a}};
\endxy
\end{gather} 
\begin{gather}\label{eq:squeeze}
\xy
(0,0)*{
\includegraphics[scale=1.2]{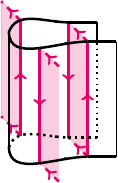}};
(0,14)*{\phantom{a}};
(0,-14)*{\phantom{a}};
\endxy
\,=\,-\,
\xy
(0,0)*{
\includegraphics[scale=1.2]{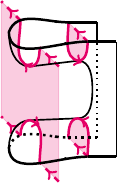}};
(0,14)*{\phantom{a}};
(0,-14)*{\phantom{a}};
\endxy
\end{gather}
the \textit{phantom cup removal} and \textit{phantom squeeze} relations (with the phantom facets facing towards the reader):
\begin{gather}\label{eq:s-sphere-fancy}
\xy
(0,0)*{
\includegraphics[scale=1.2]{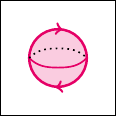}};
(0,14)*{\phantom{a}};
(0,-14)*{\phantom{a}};
\endxy
\,=\,-\,
\xy
(0,0)*{
\includegraphics[scale=1.2]{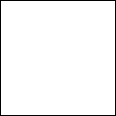}};
(0,14)*{\phantom{a}};
(0,-14)*{\phantom{a}};
\endxy
\end{gather}
\begin{gather}\label{eq:neck-fancy}
\xy
(0,0)*{
\includegraphics[scale=1.2]{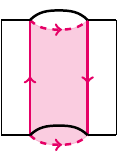}};
(0,14)*{\phantom{a}};
(0,-14)*{\phantom{a}};
\endxy
\,=\,-\,
\xy
(0,0)*{
\includegraphics[scale=1.2]{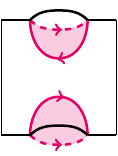}};
(0,14)*{\phantom{a}};
(0,-14)*{\phantom{a}};
\endxy
\end{gather}
(The relations \eqref{eq:s-sphere-fancy} and \eqref{eq:neck-fancy} do not appear in neither \cite{EST1} nor \cite{EST2},  
but can be proven similarly.)\makeqed
\end{lemman}

\begin{remark}\label{remark:turning-orientation-around}
The relations from \fullref{lemma:special-rels} 
exist in various differently oriented versions as well, 
as the reader is encouraged to check (see also \cite[Lemma 2.12]{EST1}).
It is crucial that the sign difference in the 
theta foam relations \eqref{eq:usual-rel-theta} and \eqref{eq:usual-rel-theta2} 
give opposite signs for the 
relations \eqref{eq:s-sphere}, \eqref{eq:neck}, \eqref{eq:sing-neck}, \eqref{eq:s-sphere-fancy} 
and \eqref{eq:neck-fancy} if we invert the orientation of all the appearing seams.
\end{remark}

\subsubsection*{Gradings}

Note that all Frobenius algebras used in the construction of $\TQFT$ 
carry a grading, i.e. $\mathrm{F}_{\bl}$ has $1$ in degree 
zero and $X$ in degree two, while $\mathrm{F}_{\re}$ 
is trivially graded. 
In particular, the functor $\TQFT$ from \fullref{theorem:sing-TQFT}
can be regarded as a functor to graded, finite-dimensional $\K$-vector spaces. 
Pulling this degree back to $\prefoam$ leads 
to:

\begin{definition}\label{definition:degree-foam}
Let $\hat{f}$
be the CW complex obtained from a prefoam $f$
by removing the phantom edges/facets, including dots on them, 
and let $\chi(\hat{f})$ denote its topological Euler characteristic. 
(The empty prefoam $\emptyfoam$ has, by definition, 
Euler characteristic zero.)
Further, let $\#\mathrm{dots}(\hat{f})$ denote the number of dots on $\hat{f}$, 
i.e. the number of dots on 
ordinary facets of $f$. Define
\begin{gather}\label{eq:degree-closed}
\mathrm{deg}(f)=-\chi(\hat{f})+2\cdot\#\mathrm{dots}(\hat{f}).
\end{gather}
This gives 
$\prefoam$ the structure of a graded category, i.e. 
the hom-spaces are graded $\K$-vector spaces and 
composition is additive with respect to these gradings.
\end{definition}

By the above, we see that $\TQFT$ is actually a graded, 
symmetric monoidal functor.

\subsection{Linearization of the foam \texorpdfstring{$2$}{2}-category}\label{subsec:foamy-category}

Next, we need 
the notion of an \textit{open prefoam}.
These are constructed similarly to closed prefoams,
but are embedded in 
$\R\times[-1,+1]^2\subset\R^3$, such that its vertical 
(second coordinate) boundary 
components are straight lines in $\R\times\{\pm 1\}\times[-1,+1]$, and its 
horizontal (third coordinate) boundary 
components are webs embedded in 
$\R\times[-1,+1]\times\{\pm 1\}$ (using the same conventions 
for orientation etc. as before, see 
e.g. \eqref{eq:orientation}). Again, we can see these as cobordisms 
between the (non-necessarily closed) webs $u$ and $v$. 
This gives rise to a vertical composition $\circ$ via gluing (and rescaling), as 
well as a horizontal composition $\otimes$ via juxtaposition (and rescaling). 
We consider such open prefoams modulo isotopies in $\R\times[-1,+1]^2$ 
which fix the vertical and horizontal boundary, 
and the condition that generic slices are webs.

Let $f$ be some open prefoam, and let
$\#\mathrm{vbound}(\hat{f})$ be the number of vertical boundary components 
of $\hat{f}$, the latter which is defined similarly as above.
We extend \fullref{definition:degree-foam} to open prefoams $f$ 
via:
\begin{gather}\label{eq:degree}
\mathrm{deg}(f)=-\chi(\hat{f})+2\cdot\#\mathrm{dots}(\hat{f})
+\tfrac{1}{2}\cdot\#\mathrm{vbound}(\hat{f}).
\end{gather}
(The reader should check that this definition is additive under composition.)

\subsubsection*{From prefoams to foams}

We are now ready to define \textit{foams}.

\begin{definition}\label{definition:foam-2-cat}
Let $\foamf$ be the additive 
closure of the graded, $\K$-linear $2$-category 
given by:
\smallskip
\begin{enumerate}[label=(\roman*)]

\setlength\itemsep{.15cm}

\item The underlying structure of objects and morphisms 
is given by the category $\webcatf$ of webs 
from \fullref{definition:webs-monoidal}.

\item The space of $2$-morphisms between two webs $u$ and $v$ 
is a quotient of the graded, free $\K$-vector space on basis given by all 
open prefoams from $u$ to $v$. 
The grading is defined to be induced by \eqref{eq:degree}.

\item The quotient is obtained by modding out the relations from 
\fullref{subsec:relations} as well as 
all relations they induce by closing prefoams via so-called 
bending or clapping, see e.g. \cite[(2.21)]{EST2} 
or \cite[Section 2.2.3)]{ETW} with
\begin{gather}\label{eq:bend}
\begin{tikzpicture}[anchorbase,scale=.7]
	\draw [ultra thick] (0,1) to (1,2);
	\draw [ultra thick, orchid, directed=.55] (0,-2) to (1,-1);
	\draw [thin] (0,1) to (0,-2);
	\draw [thin] (1,2) to (1,-1);
	\draw [thick, dotted] (-.5,1) to (.5,1);
	\draw [thick, dotted] (.5,2) to (1.5,2);
	\draw [thick, dotted] (.5,-1) to (1.5,-1);
	\draw [thick, dotted] (-.5,-2) to (.5,-2);
	\node at (.5,0) {$\bullet$};
	\node at (1,0.95) {$\phantom{a}$};
	\node at (1,-1.95) {$\phantom{a}$};
\end{tikzpicture}
\;
\begin{gathered}
\xrightarrow{\text{\phantom{u}bend\phantom{n}}}
\\[-12pt]
\phantom{\xleftarrow[\text{unbend}]{\phantom{.}}}
\end{gathered}
\;
\begin{tikzpicture}[anchorbase,scale=.7]
	\draw [ultra thick] (0,1.5) to [out=270, in=180] (.5,1) to [out=0, in=270] (1,1.5);
	\draw [ultra thick, orchid, directed=.55] (0,-1.5) to [out=270, in=180] (.5,-2) to [out=0, in=270] (1,-1.5);
	\draw [thin] (0,1.5) to (0,-1.5);
	\draw [thin] (1,1.5) to (1,-1.5);
	\draw [thick, dotted] (-.5,1.5) to (1.5,1.5);
	\draw [thick, dotted] (-.5,-1.5) to (1.5,-1.5);
	\node at (.5,-.25) {$\bullet$};
	\node at (1,0.95) {$\phantom{a}$};
	\node at (1,-1.95) {$\phantom{a}$};
\end{tikzpicture}
\;
\begin{gathered}
\xrightarrow{\text{\phantom{u}bend\phantom{n}}}
\\[-12pt]
\xleftarrow[\text{unbend}]{\phantom{.}}
\end{gathered}
\;
\begin{tikzpicture}[anchorbase,scale=.7]
	\draw [ultra thick] (1,5) to [out=270, in=180] (2,4.5) to [out=0, in=270] (3,5);
	\draw [ultra thick, orchid, directed=.55] (1,5) to [out=90, in=180] (2,5.5) to [out=0, in=90] (3,5);
	\draw [thin] (1,5) to [out=270, in=180] (2,3.5) to [out=0, in=270] (3,5);
	\draw [thick, dotted] (.5,5) to (3.5,5);
	\node at (2,4) {$\bullet$};
	\node at (2,5.5) {$\phantom{a}$};
\end{tikzpicture}
\end{gather} 
being an example of the bending of a foam, where the 
orientations and colorings are just for illustration purposes.

\item The vertical and the horizontal compositions are $\circ$ and $\otimes$ from above.
\end{enumerate}
\smallskip
Note that these relations are homogeneous which endows $\foamf$ 
with the structure of a graded $2$-category 
in the sense of \fullref{convention:graded}.
\end{definition}

We call $\foamf$ the \textit{(full) foam} $2$\textit{-category}. 
The $2$-morphisms from $\foamf$ are called \textit{foams}. 
(We also use the same notions as we had for prefoams for foams.)

\begin{remark}\label{remark:canopolis}
The $2$-category $\foamf$ can also be 
defined as a canopolises in the sense 
of Bar-Natan \cite[Section 8.2]{BN1}. We stay here with the 
$2$-categorical formulation since in this setup we obtain 
an equivalent algebraic description 
in \fullref{proposition:cats-are-equal-yes}.
\end{remark}

\subsubsection*{Comment on a representation theoretical interpretation of webs}

We now discuss the relation of our webs 
to categories arising in representation theory. 
(This section may be skipped on the first reading.) 
The reader unfamiliar with the translation 
from webs to intertwiners 
is referred to \cite{Ku1}.)

\begin{remark}\label{remark:nogl2-webs}
The webs we use in this paper 
do not carry orientations on ordinary 
edges. 
In contrast, phantom edges 
carry orientations, cf. \fullref{definition:webs-monoidal}.
If one sees the ordinary edges as corresponding to 
the vector representation $V_\g$ of 
an associated Lie (quantum) algebra $\g$ 
and phantom edges corresponding to the second exterior power $\bV^2 V_\g$ of it, then 
this translates to $(V_\g)^{\ast}\cong V_\g$ 
as $\g$-modules, but $(\bV^2 V_\g)^{\ast}\not\cong\bV^2 V_\g$. 
Thus, if we see webs as $\g$-intertwiners 
with $\emptyweb$ corresponding to the ground field 
$\K$, then we have the situation
\begin{gather*}
\raisebox{0.075cm}{\xy
(0,0)*{\includegraphics[scale=1.2]{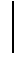}};
\endxy}
\leftrightsquigarrow
V_{\g}\to V_{\g},
\quad
\xy
(0,0)*{\includegraphics[scale=1.2]{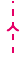}};
\endxy
\leftrightsquigarrow
\bV^2 V_{\g}\to\bV^2 V_{\g},
\\
\raisebox{0.075cm}{\xy
(0,0)*{\includegraphics[scale=1.2]{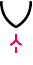}};
\endxy}
\leftrightsquigarrow
\bV^2 V_{\g}\hookrightarrow V_{\g}\otimes V_{\g},
\quad
\xy
(0,0)*{\includegraphics[scale=1.2]{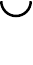}};
\endxy
\leftrightsquigarrow
\K\hookrightarrow V_{\g}\otimes V_{\g}.
\end{gather*}
These are examples of $\g$-webs as $\g$-intertwiners. 
The two right $\g$-intertwiners are given by inclusion, the others 
are identities.

For $\slt$-webs (a.k.a. Temperley--Lieb diagrams) 
one does not need orientations because
\[
(V_{\slt})^{\ast}\cong V_{\slt}
\quad\text{and}\quad 
(\bV^{2} V_{\slt})^{\ast}\cong\bV^{2} V_{\slt}.
\]
\makeautorefname{section}{Sections}
Since we basically use 
phantom edges to encode signs, repeating all constructions 
from \fullref{sec:foams}, \ref{sec:web-algebra} and \ref{sec:comb-model}
for $\slt$ 
is very easy and one obtains Bar-Natan's $\slt$-foams $\foamslt$ and the 
associated web and arc algebras as in \cite{BN1}, \cite{Kh1}.

For $\glt$-webs 
one would have to orient ordinary edges as well 
since
\[
(V_{\glt})^{\ast}\not\cong V_{\glt}
\quad\text{and}\quad 
(\bV^{2} V_{\glt})^{\ast}\not\cong\bV^{2} V_{\glt}.
\]
Again, copying \fullref{sec:foams}, \ref{sec:web-algebra} and \ref{sec:comb-model} appropriately 
would give a $\glt$-foam $2$-category 
$\foamglt$ as in \cite{EST1}, \cite{EST2}.
Note that ordinary circles in such $\glt$-webs are all 
isomorphic as morphisms of $\foamglt$, regardless 
of their orientation. Moreover, 
the isomorphisms between these (cf. \eqref{eq:usual3} 
and below \fullref{lemma:relation-webs}) are canonical  
in the sense that they do not introduce any signs. 
Thus, for a lot of application as e.g. functorial link homologies, 
$\foamf$ and $\foamglt$ are exchangeable.
In fact, we do not have a 
representation theoretical interpretation of $\foamf$, 
but it is the $2$-category 
which we can connect to the type $\typeD$ arc algebras.
(For more on the relation between $\slt$- and $\glt$-web categories 
see e.g. \cite[Remark 1.1]{TVW1}).
\end{remark}

\makeautorefname{section}{Section}
%%%%%%%%%%%%%%%%%%%%%%%%%%%

%%%%%%%%%%%%%%%%%%%%%%%%%%%%%%%%%%%%%%%%
%%%                                  %%%
%%%        Section 3                 %%%
%%%                                  %%%
%%%%%%%%%%%%%%%%%%%%%%%%%%%%%%%%%%%%%%%%
\section{The web algebra}\label{sec:web-algebra}
%%%%%%%%%%%%%%%%%%%%%%%%%%%
We aim to define the web algebra
following a similar strategy as for other such algebras.

\subsection{The algebra presenting foams}\label{subsec:webalg}

Recall that $\word{k},\word{l}$ etc. denote finite words in 
the symbols $\bl$ 
and $\re,\invo{\re}$. 
We call these \textit{balanced} in case they have an even number 
of symbols $\bl$.
The set of such balanced words
is denoted by $\bx$. Furthermore,  
we write $\bl_{\word{k}}$ to 
denote the total number of $\bl$'s in $\word{k}$.
For later use: A \textit{(balanced) block} $\word{K}$ is a set 
consisting of all words $\word{k}$ with $\bl_{\word{k}}=\brank$, 
for some fixed, even, non-negative integer $K$, called the \textit{rank 
of} $\word{K}$. (Note that there is only one block of a fixed rank, 
and we always match this block and its rank notation-wise.)
The set of these blocks
is denoted by $\bX$.

Further, denote by 
$\CUP^{\word{k}}=\Hom_{\webcatf}(\emptyword,\word{k})$, whose elements are called 
\textit{cup webs}.
Having two cup webs $u,v\in\CUP^{\word{k}}$, one obtains a closed web 
$uv^{\ast}=v^{\ast}\circ u$ with composition $\circ$ as 
in \fullref{definition:webs-monoidal}, i.e. we glue $v^{\ast}$ on top of $u$.

\begin{convention}\label{convention:fix-line}
Whenever we work with cup webs $u,v\in\CUP^{\word{k}}$ 
or closed webs of the form $uv^{\ast}$ we fix a line (which we  illustrate as a dotted line, cf. \eqref{eq:mult-saddle}) 
on which $\word{k}$ is located. 
This is the monoidal view on webs as in \fullref{definition:webs-monoidal}, which is important to 
define some notions later. (For example, 
the notions of a $\caseC$ shape and a $\Ccase$ shape make sense.)
\end{convention}

Following the terminology of \cite[Section 3]{Kh2}, and abusing notation a bit, we define the 
\textit{web homology} $\webhom{\closure{w}}=\twoHom_{\foamf}(\emptyweb,\closure{w})$, 
i.e. the graded, $\K$-linear vector space given by all foams 
from the empty web $\emptyweb$ to the closed web $\closure{w}$.

\subsubsection*{The web algebra as a \texorpdfstring{$\K$}{K}-vector space}

Let 
$d_{\word{k}}=\frac{1}{2}\bl_{\word{k}}$.

\begin{definition}\label{definition:webalg}
Given $u,v\in\CUP^{\word{k}}$ we set 
\[
{}_u(\webalg_{\word{k}})_v=\webhom{uv^\ast}\{d_{\word{k}}\}.
\]
The \textit{web algebra} $\webalg_{\word{k}}$ \textit{for} 
$\word{k}\in\bx$ is the graded $\K$-vector space 
\[
\webalg_{\word{k}}=
{\textstyle\bigoplus_{u,v\in\CUP^{\word{k}}}}\,
{}_u(\webalg_{\word{k}})_v
\]
and the \textit{(full) web algebra} $\webalg$ is the direct sum 
of all $\webalg_{\word{k}}$ for $\word{k}\in\bx$. These 
$\K$-vector spaces are equipped with the multiplication 
recalled below.
\end{definition}

We also define $\webalg_{\word{K}}=\bigoplus_{\word{k}\in\word{K}}\webalg_{\word{k}}$ 
for all $\word{K}\in\bX$
which we use in \fullref{sec:typeD}.

\subsubsection*{The web algebra as an algebra}

We sketch the multiplicative structure. Details (which are 
easily adapted to our setup) can be found in \cite[Section 3]{MPT1}.

\makeautorefname{section}{Sections}

\begin{convention}\label{convention:stacked-diagrams}
We sometimes need more general webs than webs of the form 
$uv^{\ast}$ for $u,v\in\CUP^{\word{k}}$. Namely, 
all possible webs which can turn up in multiplication 
steps which we recall below. We call such webs \textit{stacked
 webs}, and use 
the evident notions of \textit{stacked dotted webs} and \textit{stacked 
(circle) diagrams} for the two calculi in 
\fullref{sec:comb-model} and \ref{sec:typeD} 
as well. The example to keep in mind is given
in \eqref{eq:mult-saddle}, where also some terminology 
(\textit{dotted and stacking line}) is fixed.
Note that, as stacked webs, $uv^{\ast}$ has a middle part 
consisting of identities.
\end{convention}

\makeautorefname{section}{Section}

The multiplication
\begin{gather}\label{eq:mult-sum}
\boldsymbol{\mathrm{Mult}}^{\webalg}_{\word{k}}\colon\webalg_{\word{k}} \otimes \webalg_{\word{k}} \rightarrow \webalg_{\word{k}},
\quad\quad
f\otimes g\mapsto\boldsymbol{\mathrm{Mult}}^{\webalg}_{\word{k}}(f,g)
\end{gather}
is defined using the \textit{surgery rules}. 
That is, the multiplication of 
$f\in{}_{u_{\mathrm{bot}}}(\webalg_{\word{k}})_v$ and 
$g\in{}_{v^{\prime}}(\webalg_{\word{k}})_{u_{\mathrm{top}}}$ 
is zero if $v\neq v^{\prime}$.
An example in case $v=v^{\prime}$ is:
\begin{gather}\label{eq:mult-saddle}
\scalebox{.9}{$\xy
\xymatrix{
\raisebox{.08cm}{\reflectbox{
\xy
(0,0)*{
\includegraphics[scale=1.2]{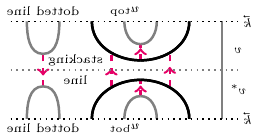}};
\endxy
}}
\ar@{|->}[rr]^{\text{saddle foam}}
&&
\raisebox{.08cm}{\reflectbox{
\xy
(0,0)*{
\includegraphics[scale=1.2]{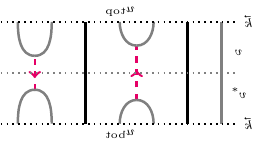}};
\endxy
}
}}
\endxy$}
\end{gather}
(the stacking line in \eqref{eq:mult-saddle} 
is omitted in the following)
where the \textit{saddle foam} locally 
looks as follows and is the identity elsewhere
\begin{gather}\label{eq:first-saddle}
\xy
(0,0)*{\includegraphics[scale=1.2]{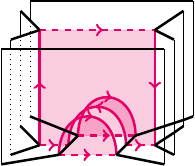}};
\endxy
\end{gather}
See e.g. \cite[Definition 3.3]{MPT1} 
or \cite[Definition 2.24]{EST1} for a detailed account.

Taking direct sums then defines $\boldsymbol{\mathrm{Mult}}^{\webalg}$.

\begin{remark}\label{remark:all-saddles}
We stress that the multiplication with a web $u^{\phantom{a}}_{\mathrm{bot}}v^{\ast}$ at the bottom 
and a web $v^{\prime}u_{\mathrm{top}}^{\ast}$ at the top is zero in case $v\neq v^{\prime}$. 
In particular, one has locally 
(read as in \eqref{eq:mult-saddle}) only the following non-zero 
surgery configurations:
\begin{gather}\label{eq:mult-non-zero}
\begin{gathered}
\,
\xy
(0,.225)*{
\includegraphics[scale=1.2]{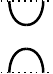}
};
(0,-10.5)*{\text{\small ordinary surgery}};
\endxy
\!,\quad
\xy
(0,0)*{
\includegraphics[scale=1.2]{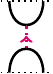}
\quad\raisebox{.55cm}{,}\quad
\includegraphics[scale=1.2]{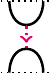}
};
(0,-10.5)*{\text{\small {\color{nonecolor}singular} surgery}};
\endxy
\\
\xy
(0,0)*{
\xy
\xymatrix{
\xy
(0,.9)*{
\includegraphics[scale=1.2]{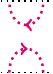}};
\endxy
\ar@{|->}[rrr]|{\;
\xy
(0,.9)*{
\includegraphics[scale=1.2]{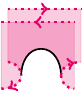}};
\endxy
\;
}
&&&
\xy
(0,.9)*{
\includegraphics[scale=1.2]{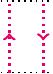}};
\endxy
}
\endxy
\quad,\quad
\xy
(0,0)*{
\includegraphics[scale=1.2]{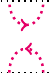}};
\endxy
};
(0,-10.5)*{\text{\small {\color{nonecolor}phantom} surgery}};
\endxy
\end{gathered}
\end{gather}
Hereby the multiplication foams are 
either ordinary, in the first case, singular 
saddles in the second and third cases (as illustrated 
in \eqref{eq:first-saddle}, which is a 
composition of two saddles as in \eqref{eq:mult-non-zero}), 
or phantom in the final two cases
(one of which we have illustrated in \eqref{eq:mult-non-zero}).
\end{remark}

By identifying the multiplication 
in $\webalg$ with the composition 
in $\foamf$ (which can be done 
analogously as in \cite[Lemma 3.7]{MPT1} via unbending 
or unclapping, cf. \eqref{eq:bend}) 
we obtain:  

\begin{propositionn}\label{proposition:asso!}
The multiplication 
$\boldsymbol{\mathrm{Mult}}^{\webalg}$ is independent 
of the order in which the surgeries are 
performed, which turns $\webalg$ into an 
associative, graded algebra.\makeqed
\end{propositionn}

\begin{remark}\label{remark:old-webalg}
The web algebras studied in \cite{EST1} and \cite{EST2} fit as follows into our picture here. They consist of 
only upwards pointing webs 
and corresponding foams, and can be seen as 
subalgebras of $\webalg$ by closing 
the diagrams in \cite{EST1} and \cite{EST2} in a braid closure fashion. 
(Hereby, the reader should keep \fullref{remark:nogl2-webs} in mind.)
Consequently, the signs that turn up in the combinatorial model presented in 
\fullref{sec:comb-model} are more sophisticated 
versions of the ones from e.g. \cite[Section 3]{EST1}.

Moreover, Khovanov's original arc algebra from \cite{Kh1} 
is a subalgebra of $\webalg$ in at least two different ways: 
First, by direct embedding without using any phantom facets. Second, by using 
one of the main results from \cite{EST2}, i.e. the isomorphism between 
the Blanchet--Khovanov algebra and the type $\typeA$ arc algebra. 
\end{remark}

\subsection{Its bimodule \texorpdfstring{$2$}{2}-category}\label{subsec:webalg-bimod}

Fix $\word{k},\word{l}\in\bx$. 

\begin{definition}\label{definition:bimoduleswebs}
Given $u\in\Hom_{\webcatf}(\word{k},\word{l})$, we consider the 
graded $\K$-vector space
\[
\M(u)=
{\textstyle\bigoplus_{v_{\mathrm{bot}},v_{\mathrm{top}}}}\,
\webhom{v^{\phantom{a}}_{\mathrm{bot}}uv_{\mathrm{top}}^{\ast}},
\]
with the sum running over all 
$v_{\mathrm{bot}}\in\CUP^{\word{k}},v_{\mathrm{top}}\in\CUP^{\word{l}}$.
We endow $\M(u)$
with a left and a right action of $\webalg$ as in \fullref{definition:webalg}.
\end{definition}

Noting that the left and right actions do not interact with each other, we 
see that all $\M(u)$'s are graded $\webalg$-bimodules 
referred to as \textit{web bimodules}.
In fact:

\begin{proposition}\label{proposition:webbimodules1}
The web bimodules $\M(u)$ are graded biprojective $\webalg$-bimodules
with finite-dimensional subspaces for all pairs
$v_{\mathrm{bot}}\in\CUP^{\word{k}},v_{\mathrm{top}}\in\CUP^{\word{l}}$.
\end{proposition}

\begin{proof}
They are clearly graded. 
The finite-dimensionality follows from the 
existence of an explicit cup foam basis, see \fullref{proposition:cup-foam-basis}. 
They are biprojective, because they are direct 
summands of some $\webalg_{\word{k}}$ (of some $\webalg_{\word{l}}$) 
as left (right) modules for suitable 
$\word{k}\in\bx$ (or $\word{l}\in\bx$), see also \cite[Proposition 5.11]{MPT1}.
\end{proof}

\begin{remark}\label{remark:finite-dim-webs}
The web bimodules as well as the web algebras 
(all of them, i.e. $\webalg_{\vec{k}},\webalg_{\vec{K}}$ and $\webalg$) 
are infinite-dimensional. However, note that many of the summands are isomorphic since the webs are isomorphic as
morphisms in $\foamf$, cf. \fullref{lemma:relation-webs}.
\end{remark}

Taking everything together, we can define:

\begin{definition}\label{definition:catbimodulesweb}
Let $\webMod$ be the following $2$-category.
\smallskip
\begin{enumerate}[label=(\roman*)]

\setlength\itemsep{.15cm}

\item Objects are the various balanced words $\word{k}\in\bx$.

\item The morphisms are finite sums and tensor products (taken 
over the algebra $\webalg$) of $\webalg$-bimodules $\M(u)$, with composition 
given by tensoring (over $\webalg$).

\item The $2$-morphisms are $\webalg$-bimodule homomorphisms, with vertical 
and horizontal composition given by 
composition and by tensoring (over $\webalg$).
\end{enumerate}
\smallskip
We consider $\webMod$ as a graded $2$-category as in \fullref{convention:graded}.  
\end{definition}

For the next proposition we assume 
$\K=\overline{\K}$. This can be avoided, but additional 
care needs to be taken in the proof.

\begin{proposition}\label{proposition:cats-are-equal-yes}
There is an equivalence of additive, graded, $\K$-linear $2$-categories
\[
\Upsilon\colon\foamf\overset{\cong}{\longrightarrow}\webMod,
\]
which is bijective on objects and essential surjective on morphisms.
\end{proposition}

The proof is given in \fullref{subsec:cup-foam}.
%%%%%%%%%%%%%%%%%%%%%%%%%%%

%%%%%%%%%%%%%%%%%%%%%%%%%%%%%%%%%%%%%%%%
%%%                                  %%%
%%%        Section 4                 %%%
%%%                                  %%%
%%%%%%%%%%%%%%%%%%%%%%%%%%%%%%%%%%%%%%%%
\section{The combinatorial model}\label{sec:comb-model}
%%%%%%%%%%%%%%%%%%%%%%%%%%%
Foams carry information about two-dimensional topological spaces sitting in three-space.
This makes direct (non-local) computations quite involved. 
The aim of this section is to define a 
version of the web algebra given by web-like objects sitting in the plane, called the \textit{combinatorial} model.
That is, we are going to define an 
algebra $\cwebalg$ with multiplication 
$\boldsymbol{\mathrm{Mult}}^{\cwebalg}$ and show:

\begin{theorem}\label{theorem:comb-model}
There is an isomorphism of graded algebras
\begin{gather*}
\isocomb\colon\cwebalg\overset{\cong}{\longrightarrow}\webalg.
\end{gather*}
(Similarly, denoted by $\isocomb_{\word{k}}$ or $\isocomb_{\word{K}}$, on summands.)
\end{theorem}

The proof of \fullref{theorem:comb-model} 
is given in \fullref{subsec:comb-model-proof}. Note that \fullref{theorem:comb-model} 
immediately gives the following, which would otherwise be 
rather involved to prove:

\begin{corollary}\label{corollary:asso!}
The multiplication 
$\boldsymbol{\mathrm{Mult}}^{\cwebalg}$ is independent 
of the order in which the surgeries are 
performed, turning $\cwebalg$ into an  
associative, graded algebra.
\end{corollary}

In order to define $\cwebalg$ we first need to introduce several combinatorial 
notions, all of which are dictated by our desire to see $\cwebalg$ 
as a projection of $\webalg$, cf. \fullref{fig:cup-to-web}.

\begin{figure}[ht]
\[
\xy
(0,0)*{\includegraphics[scale=1.2]{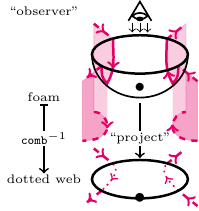}};
\endxy
\]
\caption{From foams to dotted webs: looking from the top to the bottom, a dotted web 
is obtained from a foam by projection.
}\label{fig:cup-to-web}
\end{figure}

These combinatorial notions assemble into what we call \textit{dotted webs}. 
The algebra $\cwebalg$ is then defined very much in the spirit 
of arc algebras: it has an underlying $\K$-linear structure given 
by dotted (basis) webs, and its multiplication 
is defined using a combinatorial surgery procedure, 
in contrast to the topologically defined surgery for web algebras.

The signs turning up are intricate and a major part of this section is just 
devoted to define combinatorial ways to calculate them.
The definition of 
the mapping $\isocomb\colon\cwebalg\to\webalg$ 
is then, up to details which we have migrated to \fullref{sec:proofs}, 
the inverse of the one from \fullref{fig:cup-to-web}.

\subsection{Basic notions}\label{subsec:comb-notions}

The first step toward the definition 
of a combinatorial model for foams is 
to replace foams by a decoration on webs. 
To this end, we fix a basis for foam spaces for which we define 
a combinatorial shadow and explicitly determine its structure 
constants.

\subsubsection*{The cup foam basis}

Note that \fullref{lemma:sing-TQFT} does not give an explicit basis. 
But we have a \textit{cup foam basis} 
whose construction is given in \fullref{sec:proofs}. 

\begin{proposition}\label{proposition:cup-foam-basis} 
Given $u,v\in\Hom_{\foamf}(\word{k},\word{l})$. 
There is a finite, homogeneous
cup foam basis for $\twoHom_{\foamf}(u,v)$ 
in the sense of \cite[Definition 4.12]{EST1}.
\end{proposition}

As we see below, up to signs, the construction 
is essentially dictated
by our desire to have dotted cups as our basis elements 
and this is how the reader should think of this basis for the time being, 
see e.g. in \fullref{example:evaluation-first}. 
Details will follow in \fullref{subsec:cup-foam}.
We also write $\Cupbbasis{u}{v}$ and $\Cupbasis{\closure{w}}=\Cupbbasis{\emptyweb}{\closure{w}}$ 
(for closed webs $\closure{w}$) whenever we mean the fixed cup foam 
bases given later in \fullref{definition:fixed-cup-foam}.

By \fullref{proposition:cup-foam-basis}, ${}_u(\webalg_{\word{k}})_v$ 
has a (fixed) cup foam basis 
which we denote by $\Cupcbasis{u}{v}{\word{k}}=\Cupbasis{uv^{\ast}}$. 
We also use the evident 
notation $\Cupcbasis{u}{v}{\word{K}}$ later on.

\subsubsection*{Dotted webs}

Since it follows from the existence of the cup foam 
basis, cf. \fullref{proposition:cup-foam-basis}, that 
there is a foam basis given by (potentially dotted) cups,
such a decoration for us is a dot $\mydot$ on some component 
of a web, as well as certain lines 
keeping track of the singular seams attached to 
cup foams basis elements.

Hereby, and throughout, a \textit{component} of a web is meant as a 
topological space after erasing all phantom edges.
Moreover, by our definition of webs, connected components 
are either arcs or circles. 
In this spirit 
(and recalling \fullref{convention:fix-line}), 
we also say \textit{cup} and \textit{cap} in a web meaning the evident notion 
obtained by erasing phantom edges, while a \textit{phantom cup/cap} 
are also to be understood in the evident way, cf. \eqref{eq:surgery-comb} 
where several cups and caps appear.
Having a circle, we can speak about its \textit{internal/external} by ignoring all other circles.

\begin{convention}\label{convention:no-stupid-webs}
Webs can have circles with an 
odd number of trivalent vertices 
or ill-attached phantom edges, but their associated endomorphism space in the 
foam $2$-category are zero, cf. \fullref{remark:prefoams1b}. 
We call such webs \textit{ill-oriented}, all others \textit{well-oriented}. 
Henceforth, if not stated otherwise, we consider only 
well-oriented webs (webs for short) with an
even number of trivalent vertices and well-attached phantom edges.
\end{convention}

A \textit{path} in a web $u$ is an embedding of $[0,1]$ 
into the CW complex given by $u$ after erasing 
all phantom edges.
Given a point $i$ on a 
web $u$, then the \textit{segment containing $i$} is
the maximal path containing $i$ which does not cross any phantom edges.
Recalling that webs are embedded in $\R^2\times\{z\}$, we make the following definition.
Hereby and throughout, \textit{points on $u$} 
are always meant to be on ordinary parts of the web $u$, and are 
always contained in some segment. (Which one will be evident.)

\begin{definition}\label{definition:outer-point}
Given a web $u$ and a circle $C$ of it. 
Then the \textit{base point} $\opoint(C)$ on it is defined 
to be any point in the bottom right segment of $C$, viewed in $x$-$y$-coordinates. 
\end{definition}

As in \cite[Section 3.1]{EST1}, $\opoint(C)$ is a choice of a rightmost point. 
We also write $\opoint=\opoint(C)$ for short if no confusion can arise.

\begin{definition}\label{definition:phantom-seams}
Given a web $u$, then a \textit{phantom seam} 
is a decoration of $u$ with an extra edge 
starting and ending at some trivalent vertices of $u$ 
which is oriented in the direction of the adjacent phantom edges of $u$, 
e.g.:
\begin{gather}\label{eq:phantom-seams}
\raisebox{0.2cm}{
\xy
(0,0)*{\includegraphics[scale=1.2]{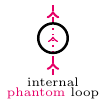}};
\endxy
}
\quad,\quad
\xy
(0,0)*{\includegraphics[scale=1.2]{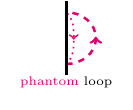}};
\endxy
\quad,\quad
\xy
(0,0)*{\includegraphics[scale=1.2]{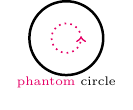}};
\endxy
\end{gather}
(We illustrate phantom seams dotted and slightly thinner than 
the other phantom edges.)
Hereby we also allow phantom circles, which are 
always assumed to be in some circle of a web, as on the right above 
in \eqref{eq:phantom-seams}. Moreover, 
the phantom seams have to be attached to a web such that the 
result does not have any intersections, 
and no trivalent vertex has more than one attached phantom seam.
\end{definition}

We are quite free to decorate webs. In order to match decorated webs with 
cup foam basis elements, we have to chose a decoration. 
This corresponds to choosing a cup foam basis 
as we see in \fullref{subsec:cup-foam}.
In particular, it depends on a choice of a point 
for each circle in question. Choosing such a point $i$, 
we call a phantom edge $i$\textit{-closest} 
if it is the first phantom edge one 
passes when going around anticlockwise, 
starting at $i$. (Similarly for other notions.)
If these points are 
the base points, then we call such a decoration 
$\opoint$-admissible.

In order to define such decorations, 
we fix the following choices 
of how to put phantom seams locally on webs, 
fixing a circle $C$ of it:
\begin{gather}\label{eq:good-deco1}
\xy
(0,0)*{\includegraphics[scale=1.2]{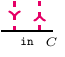}};
\endxy\!
\rightsquigarrow
\xy
(0,0)*{\includegraphics[scale=1.2]{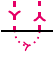}};
\endxy
\;,\;
\xy
(0,0)*{\includegraphics[scale=1.2]{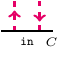}};
\endxy\!
\rightsquigarrow
\xy
(0,0)*{\includegraphics[scale=1.2]{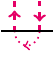}};
\endxy
\end{gather}
\begin{gather}\label{eq:good-deco2}
\xy
(0,0)*{\includegraphics[scale=1.2]{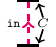}};
\endxy
\rightsquigarrow
\xy
(0,0)*{\includegraphics[scale=1.2]{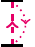}};
\endxy
\quad,\;
\xy
(0,0)*{\includegraphics[scale=1.2]{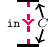}};
\endxy
\rightsquigarrow
\xy
(0,0)*{\includegraphics[scale=1.2]{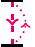}};
\endxy
\end{gather}
where $\mathtt{in}$ denotes the interior of $C$.
(We do not distinguish between putting the 
phantom seams to the bottom 
or top in \eqref{eq:good-deco1}, or right or left in \eqref{eq:good-deco2}, 
cf. \eqref{eq:dot-web-rels}).

Now, a $\opoint$\textit{-admissible} decoration 
is one obtained by applying \fullref{algorithm:decos} 
to a circle $C$ in the web $u$ with chosen base points $i=\opoint$.

\IncMargin{1em}
\begin{algorithm}
\SetKwData{Left}{left}\SetKwData{This}{this}\SetKwData{Up}{up}
\SetKwFunction{Union}{Union}\SetKwFunction{FindCompress}{FindCompress}
\SetKwInOut{Input}{input}\SetKwInOut{Output}{output}

 \Input{a web $u$, a circle $C$ in it and a chosen point $i$ of it\;}
 \Output{an $i$-admissible decoration $C_{\mathrm{dec}}$ of $C$\;}
 \BlankLine
 \textit{initialization}, let $C_{\mathrm{dec}}$ be the circle without decorations\;
 \While{$C$ has attached phantom edges}{
 \eIf{$C$ contains a pair as in \eqref{eq:good-deco1}}{
   apply \eqref{eq:good-deco1} to the $i$-closest such pair\;
   add the corresponding phantom seam to $C_{\mathrm{dec}}$\;
   remove the corresponding pair from $C$\;}
  {
   apply \eqref{eq:good-deco2} to any phantom digon not containing $i$\;
   add the corresponding phantom seam to $C_{\mathrm{dec}}$\;
   remove the corresponding phantom digon from $C$\;}
 }
 \BlankLine
 \caption{The $i$-admissible decoration algorithm.}\label{algorithm:decos}
\end{algorithm}\DecMargin{1em}

Then we piece everything together as in \fullref{subsec:cup-foam}.
(Details follow in \fullref{subsec:cup-foam}, 
e.g. the notion \textit{phantom digon} is defined therein.
Furthermore, for the time being, we ignore 
questions regarding well-definedness etc. 
The only thing the reader need to know at this point is that 
the $\opoint$-admissible decoration are the ones turning up as 
in \fullref{fig:cup-to-web}.)

\begin{definition}\label{definition:dot-deco}
Given a web $u$, 
we allow each circle of it to be decorated by a 
dot $\mydot$, where we assume that the dot is on 
a segment of $u$. 
Moreover, each trivalent vertex of $u$ is decorated by an 
attached phantom seam (there can be any finite number of 
phantom circles), which are not allowed to cross each other.
We call such a web 
with decorations a \textit{dotted web}. 

We call a dotted web a \textit{dotted basis web} in case 
the following are satisfied.
\smallskip
\begin{enumerate}[label=(\roman*)]

\setlength\itemsep{.15cm}

\item All dots  
are on the segments of the base points $\opoint$ for all circles $C$ in $u$.

\item All phantom seams decorations are $\opoint$-admissible.

\end{enumerate}
\smallskip
By convention, if some circle $C$ is not dotted 
at all, the first condition is satisfied for it.
\end{definition}

We stress that dotted basis webs never 
have phantom circles. Our choice of terminology comes from the 
fact (which we will see later on) that these webs play the role of a basis for foams 
in the setting of dotted webs.

We denote dotted webs using 
capital letters as e.g. $U,V,\closure{W}$ etc., 
and we say they are of \textit{shape} $u,v,\closure{w}$ etc.
In the following we consider dotted (basis) webs up to isotopies 
of these seen as decorated (by dots) planar graphs, as well as the relations
\begin{gather}\label{eq:dot-web-rels}
\xy
(0,0)*{\includegraphics[scale=1.2]{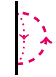}};
\endxy
\,=\,
\xy
(0,0)*{\includegraphics[scale=1.2]{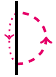}};
\endxy
\quad,\quad
\xy
(0,0)*{\includegraphics[scale=1.2]{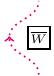}};
\endxy
=
\xy
(0,0)*{\includegraphics[scale=1.2]{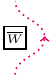}};
\endxy
\end{gather}
where $\closure{W}$ is some dotted web not 
connected to the displayed 
phantom seam. 
(Similar for all versions of these with 
different orientations.)

\begin{definition}\label{definition:dotted-degree}
The \textit{degree} of a dotted web $U$ is defined as
\begin{gather}\label{eq:dotted-degree}
\mathrm{deg}(U)=-\#C+2\cdot\#\mathrm{dots},
\end{gather}
where $\#C$ is the number of circles 
and $\#\mathrm{dots}$ the number of dots in $U$.
\end{definition}

\begin{example}\label{example:dotted-web}
Below we have illustrated three examples $\closure{W}_1,\closure{W}_2$ and 
$\closure{W}_3$ of 
possible decorations of a web $\closure{w}$. 
The dotted web $\closure{W}_3$ is not a dotted basis 
web: the dot is not on the $\opoint$-segment, 
there is a phantom circle and the two 
phantom seams are not $\opoint$-admissible. 
(We have indicated all of them using the word ``bad''.)
\begin{gather*}
\raisebox{.075cm}{\xy
(0,0)*{
\includegraphics[scale=1.2]{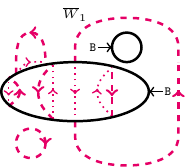}
};
(0,-19)*{\text{\small a dotted basis web}};
\endxy}
\quad,\quad
\xy
(0,0)*{
\includegraphics[scale=1.2]{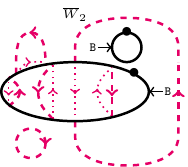}
};
(0,-19)*{\text{\small a dotted basis web}};
\endxy
\\
\xy
(0,0)*{
\includegraphics[scale=1.2]{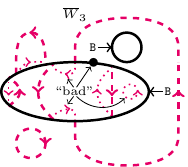}
};
(0,-19)*{\text{\small not a dotted basis web}};
\endxy
\end{gather*}
Thus we allow at most one dot per circle. From left to right, the degrees are $-2$, $2$ and $0$. 
\end{example}

Moreover, we define:

\begin{definition}\label{definition:signs-phantom-circles}
Given a dotted web $U$ we define $\apc(U)$ to be the total number of 
anticlockwise (negative) oriented phantom circles.
\end{definition}

\begin{example}\label{example:phantom-circle}
One has 
$\apc(\closure{W}_1)=\apc(\closure{W}_2)=0$, 
but $\apc(\closure{W}_3)=1$ for the three dotted webs in \fullref{example:dotted-web}.
\end{example}

\subsubsection*{Keeping track of the dot moving signs} 

In the following paths from a 
point $i$ to a point $j$ are denoted by $\dpath{i}{j}$.

\begin{definition}\label{definition:signs-dot-travel}
Given web $u$ 
and two fixed points $i,j$ on $u$ which are 
connected by a path $\dpath{i}{j}$, we define
\begin{gather*}
\begin{aligned}
\distt(\dpath{i}{j})&=\text{number of }{\color{nonecolor}\text{phantom}}\text{ edges attached to }\dpath{i}{j}.
\end{aligned}
\end{gather*}
We extend $\distt(\dpath{\placeholderout}{\placeholderout})$ 
additively for concatenations of distinct path. 
(Here and in the following $\placeholderout$ plays the role of 
a place holder.)
\end{definition}

\begin{example}\label{example:dot-travel}
A blueprint example is provided by the web from \fullref{example:loops},
using the same choice of a circle $C$ and points as therein.
If we choose the corresponding path going around anticlockwise, then
we have
\begin{gather*}
\distt(\dpath{i}{j})=3,\quad\quad \distt(\dpath{i}{k})=5,\quad\quad 
\distt(\dpath{i}{l})=8,
\end{gather*}
for example.
\end{example}

In general, 
$\distt(\dpath{i}{j})$ depends 
on which path connecting $i$ and $j$ 
is chosen. But we note the following lemma which 
we need to make the sign assignment below in \fullref{subsec:comb-model} 
well defined, and whose (very easy) proof is left to the reader 
(keeping \fullref{convention:no-stupid-webs} in mind):

\begin{lemman}\label{lemma:signs-dot-travel-well-defined}
The statistics from \fullref{definition:signs-dot-travel} taken modulo $2$
do not depend on the path between the points $i$ and $j$.\makeqed
\end{lemman}

\subsubsection*{Keeping track of the topological signs}

We call a situation as in the middle of \eqref{eq:phantom-seams} a
\textit{phantom loop}, no matter how many other 
phantom edges are in between the two trivalent vertices. 
In particular, phantom loops in closed webs are always 
associated to a circle, namely the one they start/end. 
Thus, we can say whether they are \textit{internal} or \textit{external}, 
cf. \eqref{eq:phantom-seams}.

\begin{definitionn}\label{definition:negative-p-loop}
Given a circle $C$ in some web $u$ 
and a fixed point $i$ on it. 
Let $L$ be an internal phantom loop attached to $C$. Then:
\smallskip
\begin{enumerate}[label=(\roman*)]

\setlength\itemsep{.15cm}

\item The internal phantom loop $L$ is said to be $i$\textit{-positive}, 
if it points out of $C$ first when reading from $i$ anticlockwise.

\item The internal phantom loop $L$ is said to be $i$\textit{-negative}, 
if it points into $C$ first when reading from $i$ anticlockwise.\makeqedtri
\end{enumerate}
\end{definitionn}

(Note that this is an asymmetric property heavily depending on $i$.)

We sometimes need to consider 
internal phantom loops attached to some circle $C$ 
after removing all circles nested in $C$, cf. \fullref{example:loops}.

The notion of an \textit{outgoing phantom edge} of some circle $C$ of some web $u$ is, 
by definition, a phantom edge in the exterior of $C$, counting the 
phantom loops in the exterior twice.
(For example, the right nested circle in the web from 
\fullref{example:loops} 
has two outgoing phantom edges and one phantom loop; the circle $C$ 
in the same example has two outgoing phantom edges given by
the phantom loop $L_3$.)

\begin{example}\label{example:loops}
Consider the following web with five fixed points on a circle of it.
\[
\xy
(0,0)*{
\includegraphics[scale=1.2]{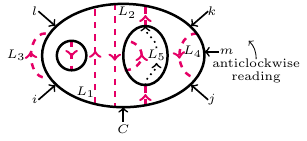}
};
\endxy
\]
The circle $C$ has 
four attached loops $L_1,L_2,L_3$ and $L_4$ as illustrated. 
Removing all circles in $C$, creates an extra loop 
formed by the two outgoing edges of the nested circle $L_5$.

With respect to positive or negative phantom loops 
we have to read anticlockwise starting from 
the various points. (Since we always read anticlockwise 
we just write \textit{read} in the following.) 
One sees that $L_1$ is positive 
with respect to $i$ and $l$, but negative 
with respect to $j,k$ and $m$. For $L_2$ the situation 
is vice versa, and $L_4$ is $m$-positive, but negative for all other points. 
The phantom loop $L_3$ is exterior 
and we do not need to 
check whether its $\placeholderout$-positive or 
$\placeholderout$-negative. The reader is encouraged to work out the situation for $L_5$.
\end{example}

\begin{definition}\label{definition:negative-p-loop-next}
A local 
situation 
(local in the sense that such edges might close to exterior 
phantom loops, cf. \eqref{eq:hwcircle})
of the following form
\begin{gather}\label{eq:outgoing-pair-first}
i\text{-positive}\colon
\xy
(0,0)*{
\includegraphics[scale=1.2]{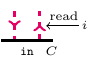}
};
\endxy
\quad,\quad
i\text{-negative}\colon
\xy
(0,0)*{
\includegraphics[scale=1.2]{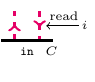}
};
\endxy
\end{gather}
is called an \textit{outgoing phantom edge pair of} $C$. Hereby, $\mathtt{in}$ 
denotes the interior of $C$. 
The notion of these being $i$-positive 
respectively $i$-negative is defined by reading from a point $i$ on $C$ 
in the anticlockwise fashion, and then seeing if the 
$i$-closest of the two outgoing phantom edges points outwards or 
inwards, see \eqref{eq:outgoing-pair-first}.
\end{definition}

\begin{definition}\label{definition:signs-intphantom-seams}
Given a web $u$ and a circle $C$ in $u$, 
and fix a point $i$ on it, and ignore all of its nested circles. 
With respect to the chosen point $i$ we define:
\begin{gather*}
\begin{aligned}
\dist(C,i)=&\text{ number of }i\text{-}\mathtt{negative}\text{ internal }{\color{nonecolor}\text{phantom}}\text{ loops attached to }C\\
&+\text{ number of }i\text{-}\mathtt{negative}\text{ outgoing }{\color{nonecolor}\text{phantom}}\text{ edge pairs of }C.
\end{aligned}
\end{gather*}
(Again, this depends on $i$.)
\end{definition}

\begin{example}\label{example:top-signs}
One has
$\dist(C,i)=\dist(C,l)=3$,
$\dist(C,j)=\dist(C,k)=4$ and $\dist(C,m)=3$ 
for the setting as in \fullref{example:loops}.
(The external phantom loop $L_3$ contributes 
to all of these as a negative outgoing phantom edge pair.)

Next, for the right nested circle $C_{\mathrm{in}}$ 
of the web in the very same example, i.e.
\[
\xy
(0,0)*{
\includegraphics[scale=1.2]{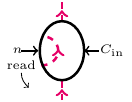}
};
\endxy
\] 
we get
$\dist(C_{\mathrm{in}},n)=1$, since the two outgoing 
phantom
edges are $n$-negative in the sense of \eqref{eq:outgoing-pair-first} 
and the internal phantom loop is 
$n$-positive.
\end{example}

\subsubsection*{Keeping track of the saddle signs}

Let $u$ be a stacked web. 
All three definitions 
below are with respect to the stacked web $u$, 
for which
we assume that we are in the situation 
of a cup-cap pair involved in an ordinary
or singular surgery (as e.g. in \eqref{eq:mult-saddle}) 
with fixed points $i$ and $j$ on them.

\begin{definition}\label{definition:signs-saddle1}
A phantom edge attached to the cup of the cup-cap pair 
is said to be $i$-\textit{positive} respectively 
$i$-\textit{negative} in cases
\begin{gather*}
i\text{-positive}\colon
\xy
(0,0)*{
\includegraphics[scale=1.2]{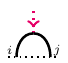}
};
\endxy
\quad,\quad
\xy
(0,0)*{
\includegraphics[scale=1.2]{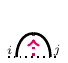}
};
\endxy
\\
i\text{-negative}\colon
\xy
(0,0)*{
\includegraphics[scale=1.2]{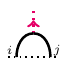}
};
\endxy
\quad,\quad
\xy
(0,0)*{
\includegraphics[scale=1.2]{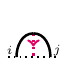}
};
\endxy
\end{gather*}
For $j$ instead of $i$ we swap positive ones with negative ones.
\end{definition}

\begin{definition}\label{definition:signs-saddle2} 
Then the \textit{saddle type} 
$\stype$ is defined to be 
\begin{gather*}
\stype(i,j)=\stype(j,i)=
\begin{cases} 0, &\text{if }\distt(\dpath{i}{j})\text{ is even},\\
1, &\text{if }\distt(\dpath{i}{j})\text{ is odd}. 
\end{cases}
\end{gather*}
(We write $\stype=\stype(i,j)=\stype(j,i)$ for short.)
\end{definition}

\begin{definition}\label{definition:signs-saddle3}
The $i$\textit{-saddle width} is defined to be
\begin{gather*}
\sdist(i)=
\text{number of }i\text{-}\mathtt{negative}{\color{nonecolor}\text{ phantom}}\text{ edges attached to the cup}.
\end{gather*} 
The $j$\textit{-saddle width} is defined similarly, 
but using $j$-negative phantom edges.
\end{definition}

\makeautorefname{definition}{Definitions}

\begin{remark}\label{remark:asymmetry}
The asymmetries in 
\fullref{definition:signs-phantom-circles}, \ref{definition:signs-intphantom-seams} 
and \ref{definition:signs-saddle3} come from our choice 
of evaluation in \eqref{eq:usual-rel-theta} 
and \eqref{eq:usual-rel-theta2}. We stress this 
using $\npguy$ as a prefix.
\end{remark}

\makeautorefname{definition}{Definition}

All of the above can be used for dotted webs as well, and 
we do so in the following. 

\begin{example}\label{example:distance-saddle}
Consider the 
surgery from \eqref{eq:mult-saddle}. 
Then $\sdist(i)=2$ for a point $i$ to the left of it, 
$\sdist(j)=1$ for a point $j$ to the right of it and $\stype=1$. 
More general, the saddle type can be thought of 
as being $0$ (usual) or $1$ (singular) 
with the convention as in \eqref{eq:surgery-comb}.
\end{example}

\subsection{Combinatorics of foams}\label{subsec:comb-model}

Given $\word{k}$, and two webs $u,v\in\CUP^{\word{k}}$, then the set
\[
\Cupcbasis{u}{v}{\word{k}}
=
\Cupbasis{uv^{\ast}}
=
\{
\text{all dotted basis webs of shape }uv^{\ast}
\}
\]  
plays the role of a combinatorial version 
of the cup foam basis.

\subsubsection*{The \texorpdfstring{$\K$}{K}-linear structure}

We start by defining the graded $\K$-vector space structure of the 
combinatorial model of the web algebra. Recall that 
$d_{\word{k}}=\frac{1}{2}\bl_{\word{k}}$.

\begin{definition}\label{definition:comb-webalg}
Given $u,v\in\CUP^{\word{k}}$ for $\word{k}\in\bx$ we set 
\[
{}_u(\cwebalg_{\word{k}})_v=\langle\Cupcbasis{u}{v}{\word{k}}\rangle_{\K}\{d_{\word{k}}\},
\]
that is the free $\K$-vector space on basis $\Cupcbasis{u}{v}{\word{k}}$.
The \textit{combinatorial web algebra} $\cwebalg_{\word{k}}$ 
\textit{for} $\word{k}\in\bx$
is the graded $\K$-vector space 
\[
\cwebalg_{\word{k}}=
{\textstyle\bigoplus_{u,v\in\CUP^{\word{k}}}}\,{}_u(\cwebalg_{\word{k}})_v
\]
with grading given on dotted basis webs via \eqref{eq:dotted-degree}.
The \textit{combinatorial (full) web algebra} $\cwebalg$ is the direct sum 
of all $\cwebalg_{\word{k}}$ for $\word{k}\in\bx$. These 
$\K$-vector spaces are equipped with the multiplication we define below.
\end{definition}

For later use in \fullref{sec:typeD}, we also define
$\cwebalg_{\word{K}}=\bigoplus_{\word{k}\in\word{K}}\cwebalg_{\word{k}}$ 
for all $\word{K}\in\bX$. Clearly, a basis of $\cwebalg_{\word{K}}$ 
is given by
$\Cupcbasis{u}{v}{\word{K}}
=\coprod_{\word{k}\in\word{K}}\Cupcbasis{u}{v}{\word{k}}$.

Note the crucial difference 
to \fullref{definition:webalg}: The multiplicative
structure of $\webalg$ was naturally given by 
the foam $2$-category $\foamf$, but we have to construct a basis for the algebra. In contrast, the basis 
of $\cwebalg$ is given, but we have 
to construct the multiplicative structure. That is what we are going to do next.

\subsubsection*{The multiplication without signs}

We define again $\boldsymbol{\mathrm{Mult}}^{\cwebalg}_{\word{k}}$ 
and then take direct sums to obtain 
$\boldsymbol{\mathrm{Mult}}^{\cwebalg}$. 
(We use notation similar to \fullref{subsec:webalg}.)

To define $\boldsymbol{\mathrm{Mult}}^{\cwebalg}_{\word{k}}$
as in \eqref{eq:mult-sum} we have to assign 
to each pair of dotted basis webs 
$U^{\phantom{a}}_{\mathrm{bot}}V^{\ast}$ and $VU_{\mathrm{top}}^{\ast}$
a sum of dotted basis webs of shape 
$u^{\phantom{a}}_{\mathrm{bot}}u_{\mathrm{top}}^{\ast}$. We do so by the usual 
inductive surgery process, where we first only change the shape 
(similarly to \cite[Section 3.3]{EST1})
and reconnect phantom seams.  

Now, for any (ordinary, singular or phantom) cup-cap pair in the middle section $V^{\ast}V$ 
we want to model the situation from \eqref{eq:mult-non-zero} and 
we do the following local replacements, where we also fix four 
points on the webs in question, e.g.:

\begin{gather}\label{eq:surgery-comb}
\begin{gathered}
\xy
(0,0)*{
\includegraphics[scale=1.2]{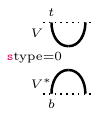}
\raisebox{1.1cm}{$\mapsto$}
\includegraphics[scale=1.2]{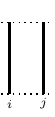}
};
(3.5,-12.5)*{\text{\small ordinary surgery}};
\endxy
\\
\xy
(0,0)*{
\includegraphics[scale=1.2]{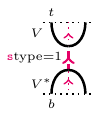}
\raisebox{1.1cm}{$\mapsto$}
\includegraphics[scale=1.2]{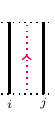}
\raisebox{.9cm}{$,$}
\includegraphics[scale=1.2]{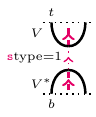}
\raisebox{1.1cm}{$\mapsto$}
\includegraphics[scale=1.2]{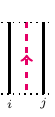}
};
(3.5,-12.5)*{\text{\small {\color{nonecolor}singular} surgery}};
\endxy
\\
\xy
(0,0)*{
\includegraphics[scale=1.2]{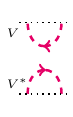}
\raisebox{1.1cm}{$\mapsto$}
\includegraphics[scale=1.2]{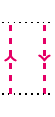}
};
(1,-12.5)*{\text{\small {\color{nonecolor}phantom} surgery}};
\endxy
\end{gathered}
\end{gather}
(There are also bigger configurations similar to these
with more phantom edges, but these are the only possible local configurations.)
If we are not in a situation as exemplified in \eqref{eq:surgery-comb}, 
then the multiplication is defined to be zero. (See also \eqref{eq:mult-non-zero}.)

\begin{remark}\label{remark:its-local}
The rules in \eqref{eq:surgery-comb} 
should be read locally in the sense that there might be several 
unaffected components in between, as e.g. in \eqref{eq:mult-saddle}. 
These do not matter for what happens to the shape, but 
the scalars depend on the precise form 
as we see below.
\end{remark}

We assume that we perform the local rules from \eqref{eq:surgery-comb} 
for the leftmost available cup-cap pair. 
Using these conventions, one directly checks that the rules 
presented below turn $\cwebalg$ into a graded algebra 
(not necessarily associative at this point).

\begin{remark}\label{remark:the-other-way-around}
There is a $\caseC$ shape within the 
multiplication procedure, cf. \fullref{example:comb-mult2} 
(which essentially defines the notion of a $\caseC$ shape). 
Its mirror, the $\Ccase$ shape, is ruled out by 
choosing the leftmost available cup-cap pair. 
Still, below we give the rule for this case as well, 
since it follows that one can actually 
choose any cup-cap pair, see 
\fullref{corollary:asso!}. 
This is in contrast to the type $\typeD$ situation 
as we see in \fullref{sec:typeD}.
\end{remark}

We first define the multiplication without signs. 
Hereby we say for short that we \textit{put a dot on a circle} 
$C$ and we mean that we put it 
on the segment of its base point $\opoint$.
The procedure from \eqref{eq:surgery-comb} either 
merges two circles into one, or it splits 
one into two.

The multiplication without signs is defined as follows, where 
we always perform the local procedure from \eqref{eq:surgery-comb}. 
We write e.g. $C_{\placeholder}$, then the corresponding circle should 
contain the point $\placeholderout$.
\newline

\noindent
\textbf{Merge.} Assume two circles $C_b$ and $C_t$
are merged into a circle $C_{\mathrm{af}}$.
\smallskip
\begin{enumerate}[label=(\alph*)]

\setlength\itemsep{.15cm}

\item If both are undotted, then nothing additionally needs to be done.

\item If one is undotted and the other one dotted, then put a dot on $C_{\mathrm{af}}$.

\item If both circles are dotted, then the result is zero.
\end{enumerate}
\smallskip

\noindent
\textbf{Split.} Assume a circle $C_{\mathrm{be}}$
splits into $C_i$ and $C_j$.
\smallskip
\begin{enumerate}[label=(\alph*)]

\setlength\itemsep{.15cm}

\item If $C_{\mathrm{be}}$ is undotted, then take the sum of two copies of the result, in one 
put a dot on $C_i$, in the other on $C_j$.

\item If $C_{\mathrm{be}}$ is dotted, then put a dot on either $C_i$ or $C_j$ 
such that both are dotted.

\end{enumerate}
\smallskip
In case of a 
$\caseC$ or $\Ccase$ shape, remove all phantom circles from the result.
\vspace*{.1cm}

\noindent
\textbf{Phantom surgery.} 
In this case nothing additionally needs to be done.
\vspace*{.1cm}

\noindent
\textbf{Turning inside out.} In the nested case (which we meet below) 
the interior of some circle turns into the exterior of another circle  
after surgery and vice versa. In those cases reconnect the phantom seams 
until they are $\opoint$-admissible, cf. \fullref{example:comb-mult1}.
(We show in \fullref{subsec:cup-foam} 
that this can always be done by reconnection locally 
as illustrated in \eqref{eq:square1}.)

\subsubsection*{The multiplication with signs}

We have to define several notions to fix the signs for the 
multiplication. The signs depend on the 
number and the positions of the phantom edges. 
As in \cite[Section 3.3]{EST1} there are \textit{dot moving signs}, 
\textit{topological signs}, 
\textit{saddle signs} and, a new type, \textit{phantom circle signs}.
All of the old signs are generalizations of the ones in \cite[Section 3.3]{EST1} (due to the fact that 
we deal here with more flexible situation).

Following \cite[Section 3.3]{EST1}, 
there are now several cases for the 
surgeries depending on whether 
a merge or a split involves nested circles or not.
In contrast, the phantom surgery only depends on whether the 
phantom cup-cap pair involved in the surgery forms a closed circle.
Then the multiplication result from above is modified as follows. 
(We use the notation from above. Moreover, the meticulous 
reader might note that we have to use \fullref{lemma:signs-dot-travel-well-defined} 
to make sure that the signs are well-defined.) 
Below all points $b,t,i,j$ are as in \eqref{eq:surgery-comb}, 
and we write $\opoint_{\placeholder}=\opoint(C_{\placeholder})$ 
for short.
\newline

\noindent
\textbf{Non-nested merge.}
In this case only one modification is made:
\smallskip
\begin{enumerate}[label=(\alph*)]

\setlength\itemsep{.15cm}

\item If $C_b$ is dotted and $C_t$ undotted, then 
we multiply by
\begin{gather}\label{eq:nnm-comb}
(-1)^{\distt(\dpath{\opoint_b}{\opoint_{\mathrm{af}}})}.
\end{gather}
This sign is called the 
\textit{(existing) dot moving sign}, and works in the same way if we exchange the roles of 
$b$ and $t$.
\end{enumerate}
\vspace*{.1cm}

\noindent
\textbf{Nested merge.}
Denote 
the inner of the two circles 
$C_b$ and $C_t$ by $C_{\mathrm{in}}$. 
Then this case is modified by 
\textit{(existing) dot moving signs}, 
\textit{topological signs} and 
\textit{saddle signs}:
\smallskip
\begin{enumerate}[label=(\alph*)]

\setlength\itemsep{.15cm}

\item If both circles are undotted, then 
we multiply the result by
\begin{gather}\label{eq:dot-nm1-comb}
(-1)^{\dist(C_{\mathrm{in}},i)}
(-1)^{\stype}
(-1)^{\sdist(i)}.
\end{gather}

\item If one of them is dotted, say $C_b$, 
then we multiply the result by
\begin{gather}\label{eq:dot-nm2-comb}
(-1)^{\distt(\dpath{\opoint_{b}}{\opoint_{\mathrm{af}}})}
(-1)^{\dist(C_{\mathrm{in}},i)}
(-1)^{\stype}
(-1)^{\sdist(i)}.
\end{gather}
Similarly for exchanged roles of $C_b$ and $C_t$.
\end{enumerate}
\vspace*{.1cm}

\noindent
\textbf{Non-nested split.}
Both cases are modified by \textit{(new and existing) dot moving signs} 
and \textit{saddle signs}:
\smallskip
\begin{enumerate}[label=(\alph*)]

\setlength\itemsep{.15cm}

\item If $C_{\mathrm{be}}$ is undotted, then 
we multiply the summand where $C_i$ is dotted by 
\begin{gather}\label{eq:dot-nns1-comb}
(-1)^{\distt(\dpath{i}{\opoint_i})}
(-1)^{\sdist(i)},
\end{gather}
and the one where $C_j$ is dotted by
\begin{gather}\label{eq:dot-nns1-comb-2}
(-1)^{\distt(\dpath{j}{\opoint_j})}
(-1)^{\stype}
(-1)^{\sdist(i)},
\end{gather}

\item If $C_{\mathrm{be}}$ is dotted, 
then we multiply the result by
\begin{gather}\label{eq:dot-nns2-comb}
(-1)^{\distt(\dpath{\placeholder}{\opoint_{\placeholder}})}
(-1)^{\sdist(\placeholder)}.
\end{gather}
Here $\placeholderout\in\{i,j\}$ is such that $C_{\placeholder}$ does not contain 
$\opoint_{\mathrm{be}}$.
\end{enumerate}
\vspace*{.1cm}

\noindent
\textbf{Nested split, $\caseC$ shape.}
Let $\closure{W}$ denotes the dotted 
web after the surgery and before 
removing the phantom circles.
Both cases are modified by \textit{(new and existing) dot moving signs}, 
\textit{topological signs} and \textit{phantom circle sign}:
\smallskip
\begin{enumerate}[label=(\alph*)]

\setlength\itemsep{.15cm}

\item If $C_{\mathrm{be}}$ is undotted, then 
we multiply the summand where $C_i$ is dotted by
\begin{gather}\label{eq:dot-ns1-comb}
(-1)^{\distt(\dpath{i}{\opoint_i})}
(-1)^{\dist(C_j,j)}(-1)^{\stype}(-1)^{\apc(\closure{W})},
\end{gather}
and the one where $C_j$ is dotted by
\begin{gather}\label{eq:dot-ns1-comb-2}
(-1)^{\distt(\dpath{j}{\opoint_j})}
(-1)^{\dist(C_j,j)}(-1)^{\apc(\closure{W})}.
\end{gather}

\item If $C_{\mathrm{be}}$ is dotted, then 
we multiply with
\begin{gather}\label{eq:dot-ns2-comb}
(-1)^{\distt(\dpath{j}{\opoint_j})}
(-1)^{\dist(C_j,j)}(-1)^{\apc(\closure{W})}.
\end{gather}
\end{enumerate}
\vspace*{.1cm}

\noindent
\textbf{Nested split, $\Ccase$ shape.}
This is slightly different from the $\caseC$ shape:
\smallskip
\begin{enumerate}[label=(\alph*)]

\setlength\itemsep{.15cm}

\item If $C_{\mathrm{be}}$ is undotted, then 
we multiply the summand where $C_i$ is dotted by
\begin{gather}\label{eq:dot-ns1-comb-C}
(-1)^{\distt(\dpath{i}{\opoint_i})}
(-1)^{\dist(C_i,i)}(-1)^{\apc(\closure{W})},
\end{gather}
and the one where $C_j$ is dotted by
\begin{gather}\label{eq:dot-ns1-comb-2-C}
(-1)^{\distt(\dpath{j}{\opoint_j})}
(-1)^{\dist(C_i,i)}
(-1)^{\stype}(-1)^{\apc(\closure{W})}.
\end{gather}

\item If $C_{\mathrm{be}}$ is dotted, then 
we multiply with
\begin{gather}\label{eq:dot-ns2-comb-C}
(-1)^{\distt(\dpath{i}{\opoint_i})}
(-1)^{\dist(C_i,i)}(-1)^{\apc(\closure{W})}.
\end{gather}
\end{enumerate}
\vspace*{.1cm}

\noindent
\textbf{Phantom surgery.}
Only one modification has to be made, i.e.:
\smallskip
\begin{enumerate}[label=(\alph*)]

\setlength\itemsep{.15cm}

\item If the phantom cup-cap pair forms a 
circle, then we multiply by $-1$.
\end{enumerate}
\vspace*{.1cm}

\noindent
\textbf{Turning inside out.} No additional changes need to be done.

\subsubsection*{Examples of the multiplication}

Let us give some examples. 
Note that we always omit the step called \textit{collapsing}, that is, 
getting rid of the identity piece in the 
middle which remains after the surgery, cf. \cite[(27)]{EST1}.
Moreover, the reader can find several examples in \cite[Examples 3.15 and 3.16]{EST1} 
of which we encourage her/him to convert to our situation here, 
see also \fullref{remark:old-webalg}.

\begin{example}\label{example:comb-mult1}
As already in the setup of \cite{EST1}, the most involved example 
is a nested merge, which now comes in plenty of varieties. 
Here is one:
\begin{gather*}
\,
\xy
(0,0)*{
\includegraphics[scale=1.2]{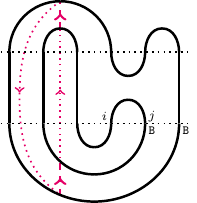}
};
\endxy
\mapsto
+
\xy
(0,0)*{
\includegraphics[scale=1.2]{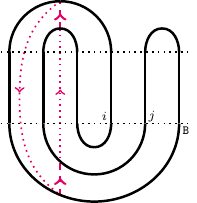}
};
\endxy
\\
\mapsto
+
\xy
(0,0)*{
\includegraphics[scale=1.2]{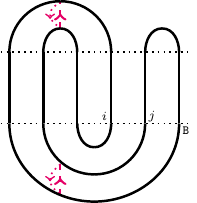}
};
\endxy
\end{gather*}
We have $\stype=0$, $\dist(C_{\mathrm{in}},i)=0$ and $\sdist(i)=0$, giving a positive sign. 
The last move, which never gives any signs, 
reconnects the phantom 
seams to fit our choice of basis (which is 
formally defined in \fullref{subsec:cup-foam}).
\end{example}

\begin{example}\label{example:comb-mult2}
The basic splitting situation are the $\Hcase$ and the $\caseC$, which 
come in different flavors (depending on the various attached phantom edges).
Here are two small blueprint examples illustrating some new phenomena 
which do not appear in the only upwards oriented setup. 
First, an $\Hcase$:
\[
\xy
(0,0)*{
\includegraphics[scale=1.2]{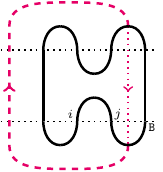}
};
\endxy
\!\!\mapsto
+
\xy
(0,0)*{
\includegraphics[scale=1.2]{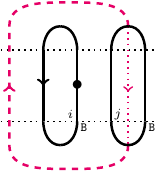}
};
\endxy
-
\xy
(0,0)*{
\includegraphics[scale=1.2]{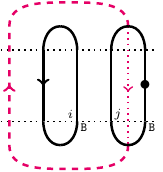}
};
\endxy
\]
Note that $\stype=0$ and the 
saddle sign $\sdist(i)$ is 
trivial in this case, but the rightmost summand acquires a 
dot moving sign.
Next, a $\caseC$:
\begin{gather*}
\,
\xy
(0,0)*{
\includegraphics[scale=1.2]{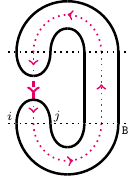}
};
\endxy
\!\!\!\mapsto
+\!
\xy
(0,0)*{
\includegraphics[scale=1.2]{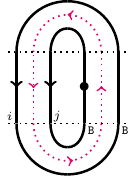}
};
\endxy
\!\!\!-\!\!
\xy
(0,0)*{
\includegraphics[scale=1.2]{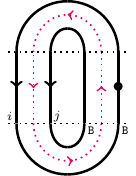}
};
\endxy
\\
\mapsto
-\!
\xy
(0,0)*{
\includegraphics[scale=1.2]{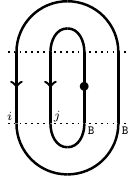}
};
\endxy
\!\!\!+\!\!
\xy
(0,0)*{
\includegraphics[scale=1.2]{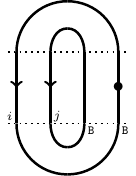}
};
\endxy
\end{gather*}
In the first step the only non-trivial sign comes from $\stype=1$ 
(which gives the minus sign for the element in the middle).
While in the second step we have removed 
the phantom circle at the cost of an overall minus sign. 
(Again, the phantom seams dictate the non-trivial 
manipulation we need to do to bring the result 
into the form of our chosen basis.)
\end{example}

\subsection{The combinatorial realization}

Next, we define the combinatorial isomorphism $\isocomb$. 
Morally it is given as in \fullref{fig:cup-to-web}. Formally it is given by using our algorithmic construction 
(where we use subscripts to distinguish between 
the two cup foam bases):

\begin{definition}\label{definition:iso-comb}
Given $\closure{w}=uv^{\ast}$ 
with $u,v\in\CUP^{\word{k}}$, we define a $\K$-linear map
\[
\isocomb_u^v\colon
\langle\Cupbasis{\closure{w}}_{\cwebalg}\rangle_{\K}\to
\langle\Cupbasis{\closure{w}}_{\webalg}\rangle_{\K},\quad \closure{W}\mapsto f(\closure{W}),
\] 
by sending a dotted basis web $\closure{W}$ with phantom seam 
structure as obtained 
from \fullref{algorithm:decos} 
(pieced together as in \fullref{subsec:cup-foam}) 
to the foam $f(\closure{W})$ of the shape as obtained by 
using \fullref{algorithm:evaluation} 
with the dot placement matched in the sense that $f(\closure{W})$ 
has a dot on the facet attached a segment which carries a base point $\opoint$ 
if and only if $\closure{W}$ has a dot on the very same segment.
Similarly, by taking direct sums, we define $\isocomb_{\word{k}}$, 
$\isocomb_{\word{K}}$ and $\isocomb$.
\end{definition}

We see in \fullref{sec:proofs} that $\isocomb$ extends 
to the isomorphism from \fullref{theorem:comb-model}.

\begin{remark}\label{remark:web-bimodules-combinatorial}
We point out that one could upgrade \fullref{theorem:comb-model} 
to include combinatorial description for the web bimodules $\M(u)$ as well. 
In principal, the steps one has to do are the same as for the algebras, 
but more different local situations as in \eqref{eq:surgery-comb} have to be considered, 
cf. \cite[Sections 4.2, 4.3 and 4.4]{EST2} where the same was done 
in the setup where webs are oriented upwards.
In order to keep the length 
of this paper in reasonable boundaries, we 
omit the rather involved details.
\end{remark}

\begin{remark}\label{remark-flows}
For other web algebras the situation 
is more delicate, cf. \fullref{remark:cup-foams}, and 
combinatorial definitions are mostly missing.
We prefer the combinatorially easier model using dots on webs
in this paper, but for e.g. web algebras as in \cite{MPT1}
one would need more sophisticated notions 
as e.g. flows in the sense of \cite{KK1} 
as decorations. 
\end{remark}
%%%%%%%%%%%%%%%%%%%%%%%%%%%

%%%%%%%%%%%%%%%%%%%%%%%%%%%%%%%%%%%%%%%%
%%%                                  %%%
%%%        Section 5                 %%%
%%%                                  %%%
%%%%%%%%%%%%%%%%%%%%%%%%%%%%%%%%%%%%%%%%
\section{Foams and the type \texorpdfstring{$\typeD$}{D} arc algebra}\label{sec:typeD}
%%%%%%%%%%%%%%%%%%%%%%%%%%%
The purpose of this section is to give a topological 
interpretation of the type $\typeD$ arc algebra from \cite{ES1}, 
which will be recalled in \fullref{subsec:recall-typeD}.
In fact, as we will see, this algebra is a subalgebra 
of the web algebra which in some sense can be seen as 
``the good definition of it''.

To elaborate, if we denote 
the type $\typeD$ arc algebra by $\arcalg=\arcalgintrod$, 
then:

\begin{theorem}\label{theorem:top-model}
There is an embedding of graded algebras
\begin{gather*}
\isotop\colon\arcalg\xhookrightarrow{\phantom{f\circ}}\webalg.
\end{gather*}
(Similarly, denoted by $\isotop_{\Lambda}$, on each summand.)
\end{theorem}

(Again, all proofs are given in \fullref{subsec:typeD-foams-proof}.) 

In order to define $\isotop$, we have to sign adjust the multiplication structure of $\arcalg$, 
denoted by $\arcalgs$, which we do in \fullref{subsec:adjust}, 
where we also give an isomorphism $\isosign\colon\arcalg\overset{\cong}{\rightarrow}\arcalgs$.
In fact, up to signs, the algebra $\arcalgs$ is defined similarly as $\arcalg$, 
namely in the usual spirit of arc algebras as a $\K$-linear vector space 
on certain diagrams called \textit{(marked) arc or circle diagrams}. 
(With markers displayed as $\typeDbulletout$.)
Having $\arcalgs$ and 
$\isocomb\colon\cwebalg\overset{\cong}{\rightarrow}\webalg$ 
from \fullref{sec:comb-model}, it is almost 
a tautology to define 
an embedding $\isotops\colon\arcalgs\hookrightarrow\cwebalg$. 
We do the latter in \fullref{subsec:realization}, but the picture to keep 
in mind how to go from $\arcalgs$ to $\cwebalg$ is provided by (another)
\textit{cookie-cutter strategy} as in \fullref{fig:cup-to-foams}.
This gives $\isotop=\isocomb\,\circ\,\isotops\,\circ\,\isosign$.

\begin{figure}[ht]
\[
\xy
(0,0)*{
\includegraphics[scale=1.2]{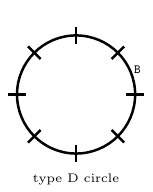}
};
\endxy
\rightsquigarrow
\xy
(0,0)*{
\includegraphics[scale=1.2]{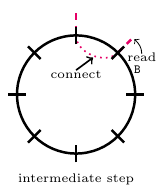}
};
\endxy
\rightsquigarrow
\cdots
\rightsquigarrow
\xy
(0,0)*{
\includegraphics[scale=1.2]{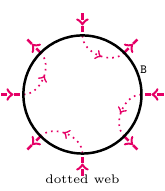}
};
\endxy
\]
\caption{From cup diagrams to dotted webs, the \textit{even} case: 
see a marked circle as a topological space, cut it from the 
rest, read anticlockwise starting 
at a fixed point $\opoint$, and connect
neighboring markers via phantom seams, choosing 
the first to be oriented outwards. In the \textit{odd} case go around clockwise.
}\label{fig:cup-to-foams}
\end{figure}

Note hereby that we 
do not assume $\arcalg$ or $\arcalgs$ to be associative, and 
associativity follows 
immediately from \fullref{theorem:top-model} and \fullref{proposition:asso!}:

\begin{corollary}\label{corollary:asso-typeD}
The multiplication 
$\boldsymbol{\mathrm{Mult}}^{\arcalgs}$ of $\arcalgs$ is independent 
of the order in which the surgeries are 
performed, turning $\arcalg$ and $\arcalgs$ into  
associative, graded algebras.
\end{corollary}

\subsection{Recalling the type \texorpdfstring{$\typeD$}{D} arc algebra}\label{subsec:recall-typeD}

\makeautorefname{remark}{Remarks}

First, we briefly recall the definition of 
$\arcalg$. All of 
this closely follows \cite{ES1}, where the reader 
can find more details (and examples).

\makeautorefname{remark}{Remark}

\subsubsection*{Weight combinatorics}
We call $\R \times\{0\}$ the 
\textit{dotted line}, cf. \eqref{eq:degrees_cups}. 
We identify $\Z$ with 
the integral points on the dotted line, called \textit{vertices}. 
(Hence, in contrast to the more flexible setup for the web algebra 
$\webalg$, points $i,j$ etc. are integers.)

A labeling of the vertices $\{1,\dots,2\brank\}$ by the symbols in the 
set $\{\Upp,\Downn\}$ is called a \textit{weight} (of \textit{rank} $\brank$).  
We identify weights of rank $\brank$ with $2\brank$-tuples
$\lambda =(\lambda_i)_{1 \leq i \leq 2\brank}$ 
with entries $\lambda_i \in \{\Upp,\Downn\}$.
We say that two weights $\lambda,\mu$ are in the 
same \textit{(balanced) block} $\Lambda$ if $\mu$ is obtained from 
$\lambda$ by finitely many combinations of basic 
linkage moves, i.e. of swapping neighbored labels 
$\Upp$ and $\Downn$, or swapping the pairs 
{$\Upp$}\makebox[1em]{$\Upp$} and {$\Downn$}\makebox[1em]{$\Downn$}
at the first two positions. 
Thus, a block is fixed by the number $\brank$ and the 
parity of the occurrence of $\Upp$ in the weight. We denote 
by $\X$ the set of blocks and denote by the rank of a block the rank of the weights in the block.

\subsubsection*{Diagram combinatorics}

Following \cite[Section 3.1]{ES1} we define \textit{cup diagrams} 
of rank $\brank$. This is a collection of crossingless 
arcs $\{\gamma_1,\dots,\gamma_\brank\}$, i.e. 
embeddings of the interval $[0,1]$ into $\R\times [-1,0]$, 
such that the collection of endpoints of all arcs coincides one-to-one 
with the set $\{1,\dots,2\brank\}$. As in \cite[Definition 3.5]{ES1} we 
allow arcs whose interior can be connected to $(0,0)$ in $\R\times [-1,0]$ 
without crossing any other arcs to carry a decoration (this condition is called \textit{admissibility}), in which case we call 
the arc \textit{marked} and otherwise \textit{unmarked}. 

For short, we simply say \textit{diagram} for any kind of cup, cap or circle 
diagram (where we also allow stacked circle 
diagrams, cf. \fullref{convention:stacked-diagrams}). 
Moreover, we use the evident notion 
of a circle $C$ in a diagram $D$ in what follows.

\begin{example}\label{example:admissibility}
Examples for admissible and non-admissible 
diagrams can be found in \cite[Section 3]{ES1}. We stress 
additionally that admissible diagrams never have circles 
with markers that are nested in other circles.
\end{example}

\begin{beware}\label{beware:beware1}
For illustration, we decorate marked arcs by $\typeDbulletout$, 
which are the same as the dots in \cite{ES1}. 
But since dots $\foamdot$ already turn up in the foam picture 
(as e.g. illustrated in \eqref{eq:usual2}), this notation had to 
be altered - our deepest apologies. Also note the difference in 
terminology, our marked arcs are called dotted in \cite{ES1}.
\end{beware}

Reflecting a cup diagram $d$ along the horizontal 
axis produces a \textit{cap diagram} $d^{\ast}$, and 
putting such a cap diagram $d^{\ast}$ on top of a cup diagram $c$ of the same rank produces 
a \textit{circle diagram}, denoted by $cd^{\ast}$, of 
the corresponding rank. As mentioned above, this is a special 
case of \cite[Definition 3.2]{ES1}. 
In all three cases we do not distinguish diagrams whose 
arcs connect the same points.

For a block $\Lambda$ of rank $\brank$, a triple 
$c \lambda d^{\ast}$ consisting of two cup diagrams $c,d$ 
of rank $\brank$ and a weight $\lambda \in \Lambda$ is called 
an \textit{oriented circle diagram} if all unmarked arcs 
connect an $\Upp$ and a $\Downn$ in $\lambda$, while all marked 
arcs either connect two $\Upp$'s or two $\Downn$'s, see 
e.g. \eqref{eq:degrees_cups}. In this case we call $\lambda$ 
the \textit{orientation} of the diagram $cd^{\ast}$.

By $\Cupbasis{\Lambda}$ we denote 
the set of all oriented circle diagrams 
(with orientations from $\Lambda$). 
Similarly, 
for cup diagrams $c,d$ of rank $\brank$, we 
denote by $\Cupcbasis{c}{d}{\Lambda}$ the set of all oriented 
circle diagrams of the form $c\lambda d^{\ast}$ with 
$\lambda \in \Lambda$. In case $\Cupcbasis{c}{d}{\Lambda}=\emptyset$ 
we say that $cd^{\ast}$ is 
\textit{non-orientable} (by weights in $\Lambda$), otherwise 
it is called \textit{orientable}.

\begin{remark}\label{remark:non-orientable}
By direct observation one sees that 
a circle diagram $cd^{\ast}$ is orientable if and only if 
all of its circles have an even number of decorations on them.
\end{remark}

Further, we equip the elements of these sets with a 
degree by declaring that arcs have the degrees given 
locally via (which are added globally):

\begin{gather}\label{eq:degrees_cups}
\begin{gathered}
\raisebox{.075cm}{\xy
(0,0)*{
\includegraphics[scale=1.2]{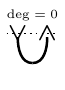}
\;\raisebox{0.6cm}{,}\;
\includegraphics[scale=1.2]{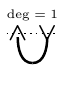}
};
(0,-10)*{\text{\small unmarked cups}};
\endxy}
\;,\;
\xy
(0,0)*{
\includegraphics[scale=1.2]{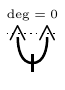}
\;\raisebox{0.6cm}{,}\;
\includegraphics[scale=1.2]{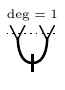}
};
(0,-10)*{\text{\small{\color{nonecolor}marked} cups}};
\endxy
\\
\raisebox{.075cm}{\xy
(0,0)*{
\includegraphics[scale=1.2]{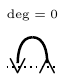}
\;\raisebox{0.6cm}{,}\;
\includegraphics[scale=1.2]{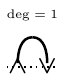}
};
(0,-10)*{\text{\small unmarked caps}};
\endxy}
\;,\;
\xy
(0,0)*{
\includegraphics[scale=1.2]{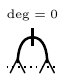}
\;\raisebox{0.6cm}{,}\;
\includegraphics[scale=1.2]{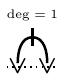}
};
(0,-10)*{\text{\small{\color{nonecolor}marked} caps}};
\endxy
\end{gathered}
\end{gather}
Here, as in the following, the dotted line indicates 
$\R\times\{0\}$.

The \textit{degree} of an oriented circle diagram 
is then in turn the sum of the degrees of all arcs 
contained in it, both in the cup and the cap diagram.

\subsubsection*{The type \texorpdfstring{$\typeD$}{D} arc algebra as a \texorpdfstring{$\K$}{K}-vector space}

Very similar as before we define:

\begin{definition}\label{definition:arcalg}
Given a block $\Lambda$ of rank $\brank$ and cup diagrams $c,d$ of 
rank $\brank$, we define the 
graded $\K$-vector space
\[
{}_c(\arcalg_{\Lambda})_d=
\langle \Cupcbasis{c}{d}{\Lambda}\rangle_{\K},
\]
that is, the free $\K$-vector space on basis given by all 
oriented circle diagrams $c \lambda d^{\ast}$ with $\lambda \in \Lambda$.
(The grading is hereby defined to be the one induced via \eqref{eq:degrees_cups}.)
The \textit{type} $\typeD$ \textit{arc algebra} $\arcalg_{\Lambda}$ 
\textit{for} $\Lambda\in\X$
is the graded $\K$-vector space 
\[
\arcalg_{\Lambda}=
{\textstyle\bigoplus_{c,d}}\,{}_c(\arcalg_{\Lambda})_d,
\]
with the sum running over all pairs of cup diagrams of rank $\brank$. Finally
the \textit{(full) type} $\typeD$ \textit{arc algebra} $\arcalg$ is the direct sum 
of all $\arcalg_{\Lambda}$, where $\Lambda$ varies over all blocks. 
The multiplication of $\arcalg$ is described 
in \cite[Section 4.3]{ES1} and we summarize it below.
\end{definition}

\subsubsection*{The type \texorpdfstring{$\typeD$}{D} arc algebra as an algebra}

As usual, to define 
$\boldsymbol{\mathrm{Mult}}^{\arcalg}\colon\arcalg\otimes\arcalg\to\arcalg$,
we do it for fixed $\Lambda$ of rank $\brank$. 
Hereby, the product of two basis elements $c_b \lambda d^{\ast}$ 
and $d^{\prime} \mu c_t^{\ast}$ in $\arcalg_\Lambda$ is declared to 
be zero unless $d=d^{\prime}$. Otherwise, to obtain 
the product of $c_b \lambda d^{\ast}$ and $d \mu c_t^{\ast}$, put $d \mu c_t^{\ast}$ 
on top of $c_b \lambda d^{\ast}$, producing a diagram that has $d^{\ast}d$ 
as its middle piece. (In this notation, $c_b,d,c_t$
are cup diagrams of rank $\brank$.)

For a cup-cap pair (possibly marked) in the middle 
section $d^{\ast}d$, which can be connected without 
crossing any arcs and such that to the right of this 
pair there are no marked arcs, we replace the cup and 
cap by using the (un)marked surgery:

\begin{gather}\label{eq:surgery-typeD}
\xy
(0,0)*{
\includegraphics[scale=1.2]{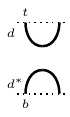}
\raisebox{1.1cm}{$\,\mapsto\,$}
\includegraphics[scale=1.2]{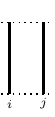}
};
(1,-12.5)*{\text{\small unmarked surgery}};
\endxy
\quad,\quad
\xy
(0,0)*{
\includegraphics[scale=1.2]{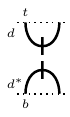}
\raisebox{1.1cm}{$\,\mapsto\,$}
\includegraphics[scale=1.2]{figs/5-14.pdf}
};
(1,-12.5)*{\text{\small{\color{nonecolor}marked} surgery}};
\endxy
\end{gather}
To avoid questions of well-definedness, we assume that we always 
pick the leftmost available cup-cap pair as above 
in what follows. One easily checks that this turns
$\arcalg$ into a graded algebra (not necessarily associative at this point).

The surgery procedure itself is, as usual, performed inductively 
until there are no cup-cap pairs left in the middle section of the diagram.
The final result is a $\K$-linear combination 
of oriented circle diagrams, all of 
which have $c_bc_t^{\ast}$ as the underlying circle 
diagram. The result in each step depends on the local 
situation, i.e. whether two components are merged 
together, or one is split into two.
One then has to add orientations and scalars to the 
corresponding diagrams.
Before we discuss how to obtain these, we need some 
notions 
to define the scalars for the multiplication.

For the next few definitions fix a 
cup or cap $\gamma=\pathd{i}{j}$ in a cup, cap, 
or circle diagram connecting vertex $i$ and $j$.

\begin{definition}\label{definition:saddle-type-d}
With the notation as above define
\[
\unmarked(\gamma)=
\begin{cases}
0, &\text{if }\gamma\text{ is }{\color{nonecolor}\mathtt{marked}},\\
1, &\text{if }\gamma\text{ is }\mathtt{unmarked},
\end{cases}
\quad\quad
\unmarkedf(\gamma)=
\begin{cases}
0, &\text{if }\gamma\text{ is unmarked},\\
1, &\text{if }\gamma\text{ is }{\color{nonecolor}\mathtt{marked}},
\end{cases}
\]
which we call the \textit{unmarked}
respectively \textit{marked saddle type}.
\end{definition}

\begin{definition}\label{definition:length}
We define the \textit{unmarked} and 
\textit{marked distance} of $\gamma=\pathd{i}{j}$
by
\begin{gather*}
\length_\Lambda(\pathd{i}{j}) = 
\unmarked(\gamma)\cdot\left| i - j \right|,\quad\quad
\lengthf_{\Lambda}(\pathd{i}{j}) = 
\unmarkedf(\gamma)\cdot\left| i - j \right|.
\end{gather*}
We extend both
additively for sequences of distinct cups and caps.
\end{definition}

Similar to \fullref{definition:outer-point} 
we define:

\begin{definition}\label{definition:rpoint} 
The \textit{base point} 
$\rpoint(C)$ is the rightmost vertex on a circle $C$ inside a circle diagram.
\end{definition}

We usually omit the subscript $\Lambda$ in the following, 
and we also write $\rpoint=\rpoint(C)$ and $\unmarked=\unmarked(\gamma)$ etc. for short.
The same works for diagrams $D$ as well.

Let $i<j$ denote the left respectively right vertex for the cup-cap pair where 
the surgery is performed, see \eqref{eq:surgery-typeD}.  
Below $D_{\mathrm{be}}$ denotes the diagram 
before the surgery {\`a} la \eqref{eq:surgery-typeD}, while 
$D_{\mathrm{af}}$ denotes the diagram 
after the surgery. Moreover, 
a circle $C$ is said to be (oriented) clockwise if the rightmost vertex 
$\rpoint(C)$ it contains is labeled with $\Downn$; otherwise 
it is said to be (oriented) anticlockwise.

Then the multiplication result is defined as follows. 
(If we write e.g. $C_{\placeholder}$, then the corresponding circle should 
contain the vertex $\placeholderout$ as in \eqref{eq:surgery-typeD}. 
Moreover, as in \fullref{subsec:comb-model}, 
we write $\rpoint_{\placeholder}=\rpoint(C_{\placeholder})$ for short.)
\newline

\noindent
\textbf{Merge.} Assume two circles $C_b$ and $C_t$ 
are merged into a circle $C_{\mathrm{af}}$.
\smallskip
\begin{enumerate}[label=(\alph*)]

\setlength\itemsep{.15cm}

\item If both are anticlockwise, then apply \eqref{eq:surgery-typeD} and orient the result anticlockwise.

\item If one circle is anticlockwise and one is clockwise, then
apply \eqref{eq:surgery-typeD}, orient the result clockwise and also multiply with 
\begin{gather}\label{eq:merge-typeD} 
(-1)^{\length(\pathd{\rpoint_{\placeholder}}{\rpoint_{\mathrm{af}}})}, 
\end{gather}
where $C_{\placeholder}$ (for $\placeholderout\in\{b,t\}$) is the clockwise circle, and 
$\pathd{\rpoint_{\placeholder}}{\rpoint_{\mathrm{af}}}$ is some concatenation 
of cups and caps connecting $\rpoint_{\placeholder}$ and $\rpoint_{\mathrm{af}}$.

\item If both circles are clockwise, then the result is zero.
\end{enumerate}
\vspace*{.1cm}

\noindent
\textbf{Split.} Assume a circle $C_{\mathrm{be}}$
splits into $C_i$ and $C_j$. If, 
after applying \eqref{eq:surgery-typeD}, the resulting 
diagram is non-orientable, the result is zero. Otherwise:
\smallskip
\begin{enumerate}[label=(\alph*)]

\setlength\itemsep{.15cm}

\item If $C_{\mathrm{be}}$ is anticlockwise, then 
apply \eqref{eq:surgery-typeD} and take two copies of 
the result. 
In one copy orient $C_i$ clockwise and $C_j$ 
anticlockwise, in the other vice versa. 
Multiply the summand where $C_i$ is oriented clockwise by
\begin{gather}\label{eq:split-typeD-a}
(-1)^{\length(\pathd{i}{\rpoint_i})}
(-1)^{\unmarked}
(-1)^{i},
\end{gather}
and the one where $C_j$ is oriented clockwise by
\begin{gather}\label{eq:split-typeD-b}
(-1)^{\length(\pathd{j}{\rpoint_j})}
(-1)^{i},
\end{gather}
using a notation similar to \eqref{eq:merge-typeD} 
with $\pathd{i}{\rpoint_i}$ and $\pathd{j}{\rpoint_j}$ 
appropriately chosen sequences of cups and caps connecting the indicated points.

\item If $C_{\mathrm{be}}$ is clockwise, then apply \eqref{eq:surgery-typeD} 
and orient $C_i$ and $C_j$ clockwise. Finally multiply with
\begin{gather}\label{eq:split-typeD-c}
(-1)^{\length(\pathd{\rpoint_{\mathrm{be}}}{\rpoint_j})}
(-1)^{\length(\pathd{i}{\rpoint_i})}
(-1)^{\unmarked}
(-1)^{i}.
\end{gather}
Again, $\pathd{\rpoint_{\mathrm{be}}}{\rpoint_j}$ 
and $\pathd{i}{\rpoint_i}$ are appropriate 
concatenations of cups and caps connecting the indicated points.
\end{enumerate}

\subsection{A sign adjusted version}\label{subsec:adjust}

The construction of the embedding from \fullref{theorem:top-model} 
splits into two pieces. First we define 
a \textit{sign adjusted version} $\arcalgs$ of the type $\typeD$ arc algebra $\arcalg$ 
and show (the proof is again given in \fullref{subsec:typeD-foams-proof}):

\begin{proposition}\label{proposition:sign-adjust}
There is an isomorphism of graded algebras
\begin{gather*}
\isosign\colon\arcalg\overset{\cong}{\longrightarrow}\arcalgs.
\end{gather*}
(Similarly, denoted by $\isosign_{\Lambda}$, on each summand.)
\end{proposition}

The sign adjusted type $\typeD$ arc algebra $\arcalgs$ is then 
easy to embed into the web algebra 
$\webalg$ as we will explain below.

By definition, the algebra $\arcalgs$ has the same graded 
$\K$-vector space structure as given in \fullref{definition:arcalg}, 
but a multiplication modeled on the one from \fullref{sec:comb-model}.

\begin{remark}\label{remark:order-important}
By \cite[Example 6.7]{ES1}, the order of surgeries is important for 
$\arcalg$. In contrast, the order is not important for 
$\arcalgs$, cf. \fullref{example:no-problem-with-order}. 
The reason is that, by \fullref{theorem:top-model} and 
what we see in \fullref{subsec:realization}, 
$\arcalgs$ has a multiplication rule which is correct in some sense, 
and the signs are easier for $\arcalgs$ than for $\arcalg$.
\end{remark}

\subsubsection*{The sign adjusted type \texorpdfstring{$\typeD$}{D} arc algebra as an algebra} \label{sec:sign-adjusted-factors}

By definition, up to signs, the surgery procedures for both multiplications 
$\boldsymbol{\mathrm{Mult}}^{\arcalg}$ and $\boldsymbol{\mathrm{Mult}}^{\arcalgs}$ 
coincide. The multiplication procedure in contrast follows 
closely the one from \fullref{subsec:comb-model}. 
(As we see in \fullref{subsec:realization}, it is the one 
from \fullref{subsec:comb-model} specialized to the
more rigid setup of the type $\typeD$ arc algebra.)
That is, we change the steps in the multiplication as follows.
As usual, all vertices $b,t,i,j$ are as in \eqref{eq:surgery-typeD}. 
(And we also use the same notations and conventions 
as in \fullref{subsec:comb-model} adjusted in the evident way.) 
Moreover, as in \fullref{remark:the-other-way-around}, we give the rules
for the $\caseC$ and the $\Ccase$ shapes.
\newline

\noindent
\textbf{Non-nested merge, nested merge and non-nested split.}
Take the signs as in \eqref{eq:nnm-comb} to \eqref{eq:dot-nns2-comb}, 
and change:
\begin{gather}\label{eq:sign-adjust-one}
\begin{aligned}
(-1)^{\distt(\dpath{\placeholder}{\placeholder})}
&\rightsquigarrow
(-1)^{\lengthf(\pathd{\placeholder}{\placeholder})},\quad\text{(keep dot moving signs)}\\
\text{Rest}
&\rightsquigarrow +1,\quad\text{(trivial other signs).}
\end{aligned}
\end{gather}
\vspace*{.1cm}

\noindent
\textbf{Nested split, $\caseC$ and $\Ccase$ shape.}
These cases are the most elaborate. 
That is, 
take the signs as in \eqref{eq:dot-ns1-comb} to \eqref{eq:dot-ns2-comb-C}, 
and change:
\begin{gather}\label{eq:sign-adjust-two}
\begin{aligned}
(-1)^{\distt(\dpath{\placeholder}{\placeholder})}
&\rightsquigarrow
(-1)^{\lengthf(\pathd{\placeholder}{\placeholder})},\quad\text{(keep dot moving signs)}\\
(-1)^{\stype}
&\rightsquigarrow
(-1)^{\unmarkedf},\quad\text{(keep the saddle type)}\\
\text{Rest}
&\rightsquigarrow +1,\quad\text{(trivial other signs).}
\end{aligned}
\end{gather}

\begin{example}\label{example:mult-D}
One of the key examples why one needs to be careful 
with the multiplication in the (original) 
type $\typeD$ arc algebra $\arcalg$ is \cite[Example 6.7]{ES1}, 
which is the case of the following $\caseC$ shape:
\begin{gather}\label{eq:mult-D-example}
\scalebox{.95}{$
\xy
(0,0)*{
\includegraphics[scale=1.2]{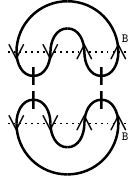}
};
\endxy
\mapsto
\xy
(0,0)*{
\includegraphics[scale=1.2]{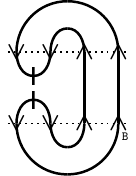}
};
\endxy
\mapsto
+\!
\xy
(0,0)*{
\includegraphics[scale=1.2]{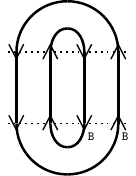}
};
\endxy
\!\!\!-\!\!
\xy
(0,0)*{
\includegraphics[scale=1.2]{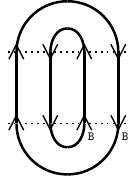}
};
\endxy$}
\end{gather}
Here we have used the sign adjusted multiplication. 
The reader should check that doing 
the $\Ccase$ gives the same result for $\arcalgs$, but not for $\arcalg$.
\end{example}

\subsubsection*{The isomorphism \texorpdfstring{$\isosign$}{sign}}

Next, we define the isomorphism $\isosign$ from 
\fullref{proposition:sign-adjust}. 
We stress that $\isosign$ is, surprisingly, quite easy.

To this end, recall that $\arcalg$ and $\arcalgs$ have basis 
given by orientations on certain diagrams. The isomorphism $\isosign$, 
seen as a $\K$-linear map $\isosign\colon\arcalg\to\arcalgs$, 
is given by rescaling each of these basis elements. 
In order to give the scalar, we first fix some diagram $D$ and 
let $\Cupbasis{D}$ denote the set of all possible 
orientations of $D$.

We write $\coeff_{D}^C$ to indicate the contribution 
of a circle $C$ inside $D$ to the coefficient, 
which we define to be
\begin{gather*}
\coeff_{D}^C(D^{\mathrm{or}})=
\left\lbrace \begin{array}{rl}
1,&\text{if }C\text{ is anticlockwise in }D^{\mathrm{or}},\\
-(-1)^{\rpoint(C)},&\text{if }C\text{ is clockwise in }D^{\mathrm{or}}.
\end{array} \right.
\end{gather*}
Here $D^{\mathrm{or}}$ denotes $D$ together with a choice of orientation, 
which induces an orientation for $C$.

\begin{definition}\label{definition:the-coefficient-map}
We define a $\K$-linear 
map via:
\begin{gather*}
\coeff_{D} \colon \left\langle \Cupbasis{D} 
\right\rangle_{\K} \longrightarrow \left\langle 
\Cupbasis{D} \right\rangle_{\K},\quad\quad
D^{\mathrm{or}}
\longmapsto 
\left( 
{\textstyle \prod_{\text{circles $C$ in }D}}\,
\coeff_{D}^C(D^{\mathrm{or}})
\right)D^{\mathrm{or}}.
\end{gather*}
(With $\coeff_{D}^C(D^{\mathrm{or}})$ as above.)
\end{definition}

Thus, we can use \fullref{definition:the-coefficient-map} to define 
$\K$-linear maps
\begin{gather}\label{eq:the-sign-iso}
\isosign_c^d\colon
{}_c(\arcalg_{\Lambda})_d\to
{}_c(\arcalgs_{\Lambda})_d,\quad\quad
\isosign_{\Lambda}\colon
\arcalg_{\Lambda}\to
\arcalgs_{\Lambda},\quad\quad
\isosign\colon
\arcalg\to
\arcalgs,
\end{gather}
for every blocks $\Lambda$ of some rank $\brank$ and all cup and cap diagrams $c,d$ of rank $\brank$.

Note that, by construction, the maps in \eqref{eq:the-sign-iso} 
are isomorphisms of graded $\K$-vector spaces.
We prove in \fullref{sec:proofs} that 
the (two right) isomorphisms of graded 
$\K$-vector spaces from \eqref{eq:the-sign-iso} 
are actually isomorphisms 
of graded algebras.

\subsection{The embedding}\label{subsec:realization}

To define $\isotop$, we first define $\isotops\colon\arcalgs\to\cwebalg$. 
Morally, it is defined as in \fullref{fig:cup-to-foams}. 
The formal definition 
is obtained from an algorithm, cf. \fullref{algorithm:diagrams-to-webs}. 
Before we give \fullref{algorithm:diagrams-to-webs}, we need to 
close up the result from \fullref{fig:cup-to-foams}:

\begin{definition}\label{definition-closing-phantoms}
Let $\word{\re}$ denote a finite word in the symbols 
$\re$ and $\invo{\re}$ only, which alternates in these.
Give two webs $u,v$ such that $uv^{\ast}\in\End_{\webcatf}(\word{\re})$. 
Then we obtain from it 
a closed web by first connecting 
neighboring (counting from right to left) outgoing phantom edges of 
$u$ and $v$ separately, and then finally the remaining 
outgoing phantom edge of $u$ with the one of $v$ to the very left of $uv^{\ast}$.
\end{definition}

\begin{example}\label{example-closing-phantoms}
Two basic examples of \fullref{definition-closing-phantoms} 
are:
\[
\xy
(0,0)*{
\includegraphics[scale=1.2]{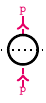}
};
\endxy
\rightsquigarrow
\xy
(0,0)*{
\includegraphics[scale=1.2]{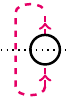}
};
\endxy
\quad,\quad
\xy
(0,0)*{
\includegraphics[scale=1.2]{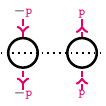}
};
\endxy
\rightsquigarrow
\xy
(0,0)*{
\includegraphics[scale=1.2]{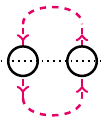}
};
\endxy
\]
Observe that one has a phantom edge passing 
from $u$ to $v$ if and only if $\word{\re}$ has an odd number 
of symbols in total. Note also, 
as in particular the right example illustrates, it is crucial 
for this to work that $\word{\re}$ 
alternates in the symbols $\re$ and $\invo{\re}$.
\end{example}

\IncMargin{1em}
\begin{algorithm}
\SetKwData{Left}{left}\SetKwData{This}{this}\SetKwData{Up}{up}
\SetKwFunction{Union}{Union}\SetKwFunction{FindCompress}{FindCompress}
\SetKwInOut{Input}{input}\SetKwInOut{Output}{output}

 \Input{a diagram $D$ with circles labeled even or odd\;}
 \Output{a dotted web $\closure{W}(D)$\;}
 \BlankLine
 \textit{initialization}, let $\closure{W}(D)$ be the empty web\;
 \For{circles $C$ in $D$}{
 \eIf{$C$ is marked}{
   run the procedure from \fullref{fig:cup-to-foams}\;
   add the corresponding web to $\closure{W}(D)$\;
   remove the corresponding circle from $D$\;}
  {
   add $C$ as a circle in a web to $\closure{W}(D)$\;
   remove the corresponding circle from $D$\;}
 }
 close the phantom edges as in \fullref{definition-closing-phantoms}\;
 \BlankLine
 \caption{Turning a marked circle into a web.}\label{algorithm:diagrams-to-webs}
\end{algorithm}\DecMargin{1em}

For any circle diagram $cd^{\ast}$ of rank $\brank$ 
we obtain via \fullref{algorithm:diagrams-to-webs} webs 
\begin{gather}\label{eq:web-and-D}
\closure{w}(cd^{\ast})\quad\text{and}\quad 
u(c),u(d)\in\CUP^{\word{k}},\quad\word{k}\in\word{K}
\end{gather}
by considering the shape of $\closure{W}(cd^{\ast})$, where 
we label the circles alternatingly from right to left even and odd, starting with an even circle. 
(This is made precise in \fullref{definition:cup-orientation}.) 
Hence, for an oriented circle diagram 
$c\lambda d^{\ast}$ with $\lambda\in\Lambda$ for a block $\Lambda$ of rank $\brank$, 
we obtain a dotted basis web
\begin{gather}\label{eq:dot-web-and-D}
\closure{W}(c\lambda d^{\ast})\in\Cupcbasis{u(c)}{u(d)}{\vec{K}}
\end{gather}
by putting a dot on each circle in 
$\closure{W}(cd^{\ast})$ for which the corresponding circle 
in $c\lambda d^{\ast}$ is oriented clockwise.
We call the dotted basis web 
from \eqref{eq:dot-web-and-D} 
\textit{the dotted basis web associated to} $c\lambda d^{\ast}$.
(The careful reader might want to check that this is actually 
well-defined by observing that \fullref{algorithm:diagrams-to-webs} 
gives a well-defined result.)

This almost concludes the definition of $\isotops$, 
but we also need a certain sign 
which 
corrects the sign turning up for 
the nested splits, see \fullref{example:comb-mult2}.

\begin{definition}\label{definition:ccase-sign}
For a stacked dotted web $\closure{W}$ we 
define 
\begin{gather*}
\begin{aligned}
\upguys(\closure{W})=&\text{ number of }\mathtt{anticlockwise}{\color{nonecolor}\text{ phantom }}\text{(edge+seam) circles}\\
&
\text{touching the top dotted line of }\closure{W},
\end{aligned}
\end{gather*} 
where phantom (edge+seam) circles are the circles obtained by 
considering phantom edges and seams. 
(These have a well-defined notion of being anticlockwise.)
\end{definition}

(Note that \fullref{definition:ccase-sign} is again asymmetric in the sense that 
we only count anticlockwise phantom (edge+seam) circles.) 

\begin{definition}\label{definition:from-cups-to-foams}
We define a $\K$-linear 
map via:
\begin{gather*}
\isotops_c^d \colon 
\langle \Cupcbasis{c}{d}{\Lambda}\rangle_{\K}
\to
\langle \Cupcbasis{u(c)}{u(d)}{\vec{K}}\rangle_{\K},
\quad
c\lambda d^{\ast}
\mapsto
(-1)^{\upguys(\closure{W}(c\lambda d^{\ast}))}\cdot
\closure{W}(c\lambda d^{\ast}),
\end{gather*}
by using the notion from \eqref{eq:dot-web-and-D}. Taking  
direct sums then defines $\isotops_{\Lambda}$ and $\isotops$.
\end{definition}

\begin{example}\label{example:no-problem-with-order}
Again, back to \cite[Example 6.7]{ES1}. The 
diagrams in \eqref{eq:mult-D-example} 
are sent to the following dotted basis webs:
\begin{gather*}
\xy
(0,0)*{
\includegraphics[scale=1.2]{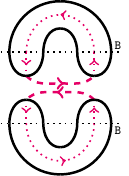}
};
\endxy
\mapsto
\xy
(0,0)*{
\includegraphics[scale=1.2]{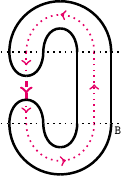}
};
\endxy
\\
\mapsto
+
\xy
(0,0)*{
\includegraphics[scale=1.2]{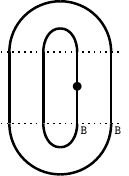}
};
\endxy
\!-
\xy
(0,0)*{
\includegraphics[scale=1.2]{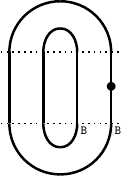}
};
\endxy
\end{gather*}
Note the difference to the result 
calculated in \fullref{example:comb-mult2}, i.e. 
there is a phantom circle sign turning up, which is corrected 
by $\isotops$: For the leftmost stacked dotted basis web $\closure{W}_1$ 
we have $\upguys(\closure{W}_1)=1$ (since it has an anticlockwise 
phantom edge-seam circle at the top), the middle stacked dotted basis web $\closure{W}_2$ 
also has $\upguys(\closure{W}_2)=1$. In contrast, for the 
leftmost stacked dotted basis web $\closure{W}_3$ one has $\upguys(\closure{W}_3)=0$, 
and the last step is where the sign goes wrong.
\end{example}

The definition of $\isotop$ is now dictated:
\begin{gather}\label{eq:one-main-iso}
\isotop=\isocomb\,\circ\,\isotops\,\circ\,\isosign
\colon
\arcalg
\overset{\cong}{\longrightarrow}
\arcalgs
\xhookrightarrow{\phantom{f\circ}}
\cwebalg
\overset{\cong}{\longrightarrow}
\webalg,
\end{gather}
and similarly for $\isotop_c^d$ and $\isotop_\Lambda$.
As usual, we show in \fullref{sec:proofs} that 
$\isotop_{\Lambda}$ and $\isotop$ are isomorphisms of graded algebras.

\subsection{Its bimodules, category \texorpdfstring{$\mathcal{O}$}{0} and foams}\label{subsec:typeD-bimodules}

Comparing with \cite{ES1} there are two generalizations 
that need to be addressed from the point of view of this section.

\begin{remark}\label{remark:blocks-as-well}
The first generalization is towards more 
general blocks as they are defined 
in \cite[Section 2.2]{ES1}. This includes 
defining weights supported on the positive 
integers with allowed symbols from the set 
$\{\raisebox{.05cm}{${\scriptstyle \bigcirc}$},\Upp,\Downn,\times\}$ where $\Upp$ and $\Downn$ 
correspond to $\bl$ in the foam language, while
$\raisebox{.05cm}{${\scriptstyle \bigcirc}$}$ should be thought of as the empty symbol 
and $\times$ as $\re$ in terms of foams.
As long as one 
restricts to the situation of balanced blocks, 
i.e. blocks where the total combined number of symbols of the form $\Upp$ and $\Downn$ is even and almost all symbols are equal to $\raisebox{.05cm}{${\scriptstyle \bigcirc}$}$, 
the whole construction presented in this section 
can be used with one key difference: whenever a 
formula in any of the multiplications (or the 
isomorphism from \fullref{subsec:typeD-foams-proof}) 
includes a power $(-1)^{\placeholder}$ where 
$\placeholderout$ is some index of 
a vertex this must be interchanged with 
$(-1)^{\mathrm{p}(\placeholder)}$ (with $\mathrm{p}$ as defined in \cite[(3.12)]{ES1}). 
The rest works verbatim.
\end{remark}

\begin{remark}\label{remark:gen-as-well} 
The second generalization is towards the \textit{generalized type} $\typeD$ \textit{arc algebra} 
$\garcalg$, 
and a \textit{generalized web algebra} $\gwebalg$ topologically presenting it. 
The algebra $\garcalg$ is the algebra as 
defined in \cite[Section 5]{ES1} including rays 
in addition to cups and caps, while, as we explain now, 
the algebra $\gwebalg$ can be thought of as a foamy type $\typeD$ version 
of the algebra defined by 
Chen--Khovanov \cite{CK1}:

The discussion in \cite[Section 5.3]{ES1} is the analog of 
the (type $\typeA$) \textit{subquotient construction} 
from \cite[Section 5.1]{EST2}, and
the analog of \cite[Theorem 5.8]{EST2} holds in the type 
$\typeD$ setup as well (by using \cite[Theorem A.1]{ES1}). Hereby, 
the main difference to \cite[Section 5.1]{EST2} lies in the 
fact that for the closure of a weight (as defined in \cite[Definition 5.1]{EST2}) 
one only uses additional symbols $\Upp$ to the right of the non-trivial 
vertices of the weight, similar to the type $\typeA$ situation, and one 
adds a total number of $\Upp$'s equal to the combined number of $\Upp$ 
and $\Downn$ occurring in the weight.

Copying the subquotient construction for 
the web algebra (as done in the type $\typeA$ situation in \cite[Section 5.1]{EST2}) 
defines $\gwebalg$, 
which then can be seen to be similar to Chen--Khovanov's (type $\typeA$) construction, 
cf. \cite[Remark 5.7]{EST2}. The corresponding 
\textit{generalized foam} $2$\textit{-category} $\foamg$ 
can then, keeping \fullref{proposition:cats-are-equal-yes} in 
mind, defined to be $\biMod{\gwebalg}$.
\end{remark}

\makeautorefname{remark}{Remarks}

Using \fullref{remark:blocks-as-well} and \ref{remark:gen-as-well}, we see that 
everything from above generalizes to 
$\garcalg$, $\gwebalg$, $\foamg$ etc. In particular, one gets an embedding
\[
\isotopg\colon
\garcalg
\xhookrightarrow{\phantom{f\circ}}
\gwebalg.
\] 
In fact, $\isotopg(\garcalg)$ is an idempotent 
truncation of $\gwebalg$.
Hence, we can 
actually define \textit{type} $\typeD$ \textit{arc bimodules} 
in the 
spirit of \fullref{definition:bimoduleswebs} 
for any stacked circle diagram.

\makeautorefname{remark}{Remark}

Thus, recalling that $\garcalg$ 
is the algebra presenting $\funnyOd$ 
(see \cite[Theorem 9.1]{ES1}), we can say that we get 
a graded topological presentation of $\funnyOd$, with the grading being 
(basically) the Euler characteristic of foams, cf. \eqref{eq:degree}.
%%%%%%%%%%%%%%%%%%%%%%%%%%%

%%%%%%%%%%%%%%%%%%%%%%%%%%%%%%%%%%%%%%%%
%%%                                  %%%
%%%        Section 6                 %%%
%%%                                  %%%
%%%%%%%%%%%%%%%%%%%%%%%%%%%%%%%%%%%%%%%%
\section{Proofs}\label{sec:proofs}
%%%%%%%%%%%%%%%%%%%%%%%%%%%
In this section we give all intricate proofs. 
There are essentially three things to prove: in the first part we 
construct the cup foam basis, in the second we show that 
$\cwebalg$ is a combinatorial model of the web algebra, 
and in the last we prove that the type $\typeD$ 
arc algebra embeds into $\cwebalg$.
 
Let us stress that we only consider (well-oriented) webs as 
in \fullref{convention:no-stupid-webs}, if not stated otherwise. 
For ill-oriented webs all 
foam spaces are zero and these also do not show up in the translation 
from type $\typeD$ to the foam setting. (Hence, there is no harm in ignoring them.)

\subsection{Proofs of \texorpdfstring{\fullref{proposition:cup-foam-basis}}{\ref{proposition:cup-foam-basis}} and \texorpdfstring{\fullref{proposition:cats-are-equal-yes}}{\ref{proposition:cats-are-equal-yes}}}\label{subsec:cup-foam}

We start by constructing the cup foam basis 
and prove all the consequences of its existence/construction.

\subsubsection*{Proof of the existence of the cup foam basis}

Our next goal is to describe isomorphisms among the 
morphisms of $\foamf$ which we call \textit{relations among webs}.

\begin{lemma}\label{lemma:relation-webs}
There exist isomorphisms in $\foamf$ realizing the following 
relations among webs. First, the \textit{ordinary} and \textit{phantom circle removals}:
\begin{gather}\label{eq:circle1}
\xy
(0,0)*{
\includegraphics[scale=1.2]{figs/2-64.pdf}
};
\endxy
\,
\cong
\,
\emptyweb\{-1\}
\oplus
\emptyweb\{+1\}
\end{gather}
\begin{gather}\label{eq:circle2}
\xy
(0,0)*{
\includegraphics[scale=1.2]{figs/2-67.pdf}
};
\endxy
\,
\cong
\,
\emptyweb
\,
\cong
\,
\xy
(0,0)*{
\includegraphics[scale=1.2]{figs/2-68.pdf}
};
\endxy
\end{gather}
Second, the \textit{phantom saddles} and the \textit{phantom digon removal}:
\begin{gather}\label{eq:square1}
\xy
(0,0)*{
\includegraphics[scale=1.2]{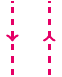}
};
\endxy
\cong
\xy
(0,0)*{
\includegraphics[scale=1.2]{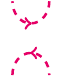}
};
\endxy
\;,\;
\xy
(0,0)*{
\includegraphics[scale=1.2]{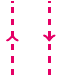}
};
\endxy
\cong
\xy
(0,0)*{
\includegraphics[scale=1.2]{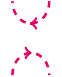}
};
\endxy
\end{gather}
\begin{gather}\label{eq:digon2}
\xy
(0,0)*{
\includegraphics[scale=1.2]{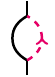}
};
\endxy
\,
\underset{-1}{\cong}
\,
\xy
(0,0)*{
\includegraphics[scale=1.2]{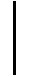}
};
\endxy
\,
\underset{+1}{\cong}
\,
\xy
(0,0)*{
\includegraphics[scale=1.2]{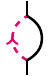}
};
\endxy
\end{gather}
(The signs indicated in \eqref{eq:digon2} are related to our choice 
of foams lifting these, see below.)
There are isotopy relations of webs as well.
\end{lemma}

Note that each phantom digon is a phantom loop, but not vice versa since 
a phantom loop might have additional phantom edges in between its 
trivalent vertices.

\begin{proof}
All of these can be proven in the usual fashion, i.e. 
by using the corresponding relation of foams and cutting the pictures in half, 
see e.g. \cite[Lemma 4.3]{EST1}. 

For the relations among webs the corresponding relations of foams are:
\smallskip
\begin{enumerate}[label=(\Alph*)]

\setlength\itemsep{.15cm}

\item The foams corresponding to \eqref{eq:circle1} are the ones in \eqref{eq:usual1} 
and \eqref{eq:usual3}.

\item The foams corresponding to \eqref{eq:circle2} are the 
ones in \eqref{eq:phantom1} 
and \eqref{eq:phantom3}.

\item The relations \eqref{eq:square1} among webs
are, by \eqref{eq:phantom3}, lifted by phantom saddle foams. 

\item The foams corresponding to \eqref{eq:digon2} are the 
ones in \eqref{eq:s-sphere-fancy} 
and \eqref{eq:neck-fancy} 
(as well as their orientation reversed counterparts).\qedhere 
\end{enumerate}
\end{proof}

\begin{lemma}\label{lemma:relation-webs-consequences}
The \textit{digon} and \textit{square removals}
\begin{gather}\label{eq:digon1}
\xy
(0,0)*{
\includegraphics[scale=1.2]{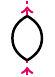}
};
\endxy
\,
\cong
\,
\xy
(0,0)*{
\includegraphics[scale=1.2]{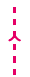}
};
\endxy
\{-1\}
\oplus
\xy
(0,0)*{
\includegraphics[scale=1.2]{figs/6-9.pdf}
};
\endxy
\{+1\}
\end{gather}
\begin{gather}\label{eq:square2}
\xy
(0,0)*{
\includegraphics[scale=1.2]{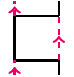}
};
\endxy
\,
\cong
\,
\xy
(0,0)*{
\includegraphics[scale=1.2]{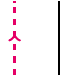}
};
\endxy
\end{gather}
\begin{gather}\label{eq:square4}
\xy
(0,0)*{
\includegraphics[scale=1.2]{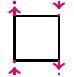}
};
\endxy
\,
\cong
\,
\xy
(0,0)*{
\includegraphics[scale=1.2]{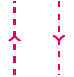}
};
\endxy
\{-1\}
\oplus
\xy
(0,0)*{
\includegraphics[scale=1.2]{figs/6-13.pdf}
};
\endxy
\{+1\}
\end{gather}
\begin{gather}\label{eq:square3}
\xy
(0,0)*{
\includegraphics[scale=1.2]{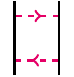}
};
\endxy
\,
\cong
\,
\xy
(0,0)*{
\includegraphics[scale=1.2]{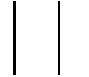}
};
\endxy
\end{gather}
are consequences of the relations among webs from \fullref{lemma:relation-webs}. 
(There are various reoriented versions as well.)
\end{lemma}

\begin{proof}
We indicate where we can apply 
phantom saddle relations \eqref{eq:square1}:
\[
\xy
(0,0)*{
\includegraphics[scale=1.2]{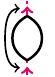}
};
\endxy
\quad,\quad
\xy
(0,0)*{
\includegraphics[scale=1.2]{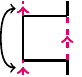}
};
\endxy
\quad,\quad
\xy
(0,0)*{
\includegraphics[scale=1.2]{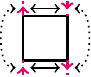}
};
\endxy
\quad,\quad
\xy
(0,0)*{
\includegraphics[scale=1.2]{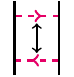}
};
\endxy
\]
(For \eqref{eq:square4}, there is a choice where to apply the phantom saddles, 
cf. \fullref{example:fix-basis}.)
One can then continue using the phantom digon \eqref{eq:digon2}, 
and removing the circle \eqref{eq:circle1} in case of \eqref{eq:digon1} and \eqref{eq:square4}. 
The corresponding foams inducing the relations 
from \fullref{lemma:relation-webs} then induce the isomorphisms in $\foamf$ realizing 
the above relations among webs.
\end{proof}

When referring to these relations among webs 
we fix the isomorphisms that we have chosen in the proof 
of \fullref{lemma:relation-webs} realizing these relations. 
(These induce the corresponding isomorphisms 
lifting the relations from \fullref{lemma:relation-webs-consequences}, 
except for \eqref{eq:square4} where there is no preferred choice where to apply 
the phantom saddles.) We call these \textit{evaluation foams}.
Note hereby, as indicated in \eqref{eq:digon2}, the foams 
realizing the phantom digon removal might come with a plus or a minus sign, 
cf. \fullref{remark:turning-orientation-around}.

The point of the relation among webs is that they 
evaluate closed webs:

\begin{lemma}\label{lemma:evaluation}
For closed web $\closure{w}$ there exists 
a sequence $(\phi_1,\dots,\phi_r)$ of relations among webs 
and some shifts $s\in\Z$
such that
\[
\closure{w}\overset{\phi_1}{\cong}\cdots\overset{\phi_r}{\cong}
{\textstyle\bigoplus_{s}}\,
\emptyweb\{s\} \quad\quad (\text{in }\foamf).
\]
Such a sequence is called \textit{an evaluation of} $\closure{w}$.
\end{lemma}

Note that we do not ask for a unique 
evaluation, but only for some evaluation.

\begin{proof}
By induction on the number $n$ of vertices of $\closure{w}$. 

\makeautorefname{lemma}{Lemmas}

If $n\leq 4$, the statement is clear by \fullref{lemma:relation-webs} 
and \ref{lemma:relation-webs-consequences}. 
(Recall that we consider well-oriented webs only.)
So assume that $n>4$. 

\makeautorefname{lemma}{Lemma}

First, we can view a closed (well-oriented) web $\closure{w}$ 
as a planar, trivalent graph in $\R^2$ with 
all faces having an even number of adjacent vertices. 
Thus, by Euler characteristic arguments,  
$\closure{w}$ must contain at least a circle face (zero adjacent vertices), 
a digon face (two adjacent vertices) or a square face (four adjacent vertices).
By \eqref{eq:circle1} and \eqref{eq:circle2} we can assume that $\closure{w}$ 
does not have circle faces. Hence, we are done by induction, since 
using \eqref{eq:digon2}, \eqref{eq:digon1}, \eqref{eq:square2}, \eqref{eq:square4}
or \eqref{eq:square3} reduces $n$. 
(Observe that these are all possibilities of what such digon or square faces could look like.)
\end{proof}

\begin{proof}[Proof of \fullref{lemma:sing-TQFT}]
This is immediate from \fullref{lemma:evaluation}.
\end{proof}

\makeautorefname{example}{Examples}

We are now ready to prove \fullref{proposition:cup-foam-basis}. 
The main ingredient is the \textit{cup foam basis algorithm} 
as provided by \fullref{algorithm:cup-basis}. 
We stress that we will fix an evaluation $(\phi_1,\dots,\phi_r)$ 
and the result will depend on this choice, cf. 
\fullref{example:evaluation-first} and \ref{example:fix-basis} below.

\makeautorefname{example}{Example}

\IncMargin{1em}
\begin{algorithm}
\SetKwData{Left}{left}\SetKwData{This}{this}\SetKwData{Up}{up}
\SetKwFunction{Union}{Union}\SetKwFunction{FindCompress}{FindCompress}
\SetKwInOut{Input}{input}\SetKwInOut{Output}{output}

 \Input{a closed web $\closure{w}$ and an evaluation $(\phi_1,\dots,\phi_r)$ of it\;}
 \Output{a sum of evaluation foams in $\webhom{\closure{w}}=\twoHom_{\foamf}(\emptyweb,\closure{w})$\;}
 \BlankLine
 \textit{initialization}, let $f_0$ be the identity foam in $\twoEnd_{\foamf}(\closure{w})$\;
 \For{$k=1$ \KwTo $r$}{
 apply the isomorphism lifting $\phi_k$ to the bottom of $f_{k-1}$ and obtain $f_k$\;}
 \caption{The cup foam basis algorithm.}\label{algorithm:cup-basis}
\end{algorithm}\DecMargin{1em}

(Hereby, if $\closure{w}$ 
has more than one connected component, it is important to 
evaluate nested components first and we do so without saying.)

\begin{proof}[Proof of \fullref{proposition:cup-foam-basis}]
Given a closed web $\closure{w}$, 
by \fullref{lemma:evaluation}, there exists 
some evaluation of it which we fix.

\makeautorefname{lemma}{Lemmas}

Hence, using \fullref{algorithm:cup-basis}, 
we get a sum of 
evaluation foams, all of which are $\K$-linear independent by construction.
Thus, by taking the set of all summands produced this way, one gets a 
basis of $\twoHom_{\foamf}(\emptyweb,\closure{w})$ 
by \fullref{lemma:relation-webs} 
and \ref{lemma:relation-webs-consequences}.

\makeautorefname{lemma}{Lemma}

For general webs $u,v$, we use the 
bending trick. Define $b(u)$ to be
\begin{gather}\label{eq:bending}
u=
\xy
(0,0)*{
\includegraphics[scale=1.2]{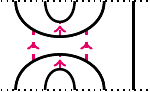}
};
\endxy
\quad,\quad
b(u)=
\xy
(0,0)*{
\includegraphics[scale=1.2]{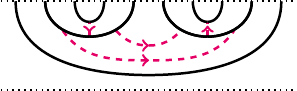}
};
\endxy
\end{gather}
Similarly for $b(v)$.
Next, using the very same arguments 
as above, 
we can write down a basis for 
$\twoHom_{\foamf}(\emptyweb,b(u)b(v)^{\ast})$.
Bending this basis back proves the statement.

Scrutiny in the above process (keeping track of grading shifts) 
actually shows that everything works graded as well and the resulting basis is homogeneous.
\end{proof}

\begin{remark}\label{remark:cap-basis}
Indeed, almost verbatim, 
we also get a dual cup foam basis, called \textit{cap foam basis}, 
i.e. a basis of $\twoHom(\closure{w},\emptyweb)$, which is dual 
in the sense that the evident pairing given by stacking a cap foam basis element 
onto a cup foam basis element gives $\pm 1$ for precisely one pair, and zero else.
\end{remark}

\begin{example}\label{example:evaluation-first}
Let us consider an easy example, namely:
\begin{gather*}
\scalebox{.9}{$
\xy
\xymatrix{
\raisebox{0.075cm}{\xy
(0,0)*{
\includegraphics[scale=1.2]{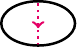}
};
\endxy}
\ar[rr]^{ \eqref{eq:digon2}\text{ (left)}}
\ar@{<~>}[d]|{\,\text{basis}\phantom{\underline{\overline{b}}}}
&&
\raisebox{0.075cm}{\xy
(0,0)*{
\includegraphics[scale=1.2]{figs/2-64.pdf}
};
\endxy}
\ar[rr]^/-.6cm/{ \eqref{eq:circle1}}
\ar@{<~>}[d]|{\,\text{basis}\phantom{\underline{\overline{b}}}}
&&
\emptyweb\{-1\}\oplus\emptyweb\{+1\}
\ar@{<~>}[d]|{\,\text{basis}\phantom{\underline{\overline{b}}}}
\\
\raisebox{0.075cm}{\xy
(0,0)*{
\includegraphics[scale=1.2]{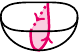}
};
\endxy}
\;+\;
\xy
(0,0)*{
\includegraphics[scale=1.2]{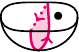}
};
\endxy
&&
\raisebox{0.075cm}{\xy
(0,0)*{
\includegraphics[scale=1.2]{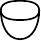}
};
\endxy}
\;+\;
\xy
(0,0)*{
\includegraphics[scale=1.2]{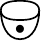}
};
\endxy
\ar[ll]^/-.25cm/{ \eqref{eq:neck-fancy}\text{ (left)}}
&&
\emptyfoam+\emptyfoam
\ar[ll]^{ \eqref{eq:usual3}}
}
\endxy$}
\end{gather*}
(Here we apply \eqref{eq:digon2} to the left face.)
Each summand is a basis element in 
$\twoHom_{\foamf}(\emptyweb,\closure{w})$ 
with the signs 
depending on whether 
we apply \eqref{eq:digon2} on the left or right face of $\closure{w}$. 
(Note that the lift of \eqref{eq:digon2} gives an overall plus sign in this case.)
\end{example}

\begin{example}\label{example:fix-basis}
The following (local) example illustrates 
the choice we have to make with respect to the 
topological shape of our cup foam basis elements:
\[
\scalebox{.95}{$
\xy
(0,0)*{
\includegraphics[scale=1.2]{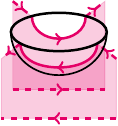}
};
(0,-15)*{\text{\small apply {\color{nonecolor}phantom} saddles at $l$}};
\endxy
\;\overset{l}{\longleftarrow}\quad
\xy
(0,0)*{
\includegraphics[scale=1.2]{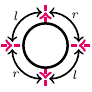}
};
\endxy
\quad\overset{r}{\longrightarrow}\;
\xy
(0,0)*{
\includegraphics[scale=1.2]{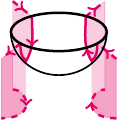}
};
(0,-15)*{\text{\small apply {\color{nonecolor}phantom} saddles at $r$}};
\endxy$}
\]
The two possible cup foam basis elements illustrated above
(obtained by applying the 
phantom saddle relation among webs \eqref{eq:square1} to either 
the pair indicated by $l$=left or $r$=right, 
and then by applying \eqref{eq:digon2} to remove the 
phantom digons) differ 
in shape, but after gluing an additional phantom saddle to the bottom 
of the left foam, 
they are the same up to a minus 
sign.
\end{example}

\begin{remark}\label{remark:cup-foams}
Our proof of the existence of a cup foam basis using 
\fullref{algorithm:cup-basis} works more general 
for any kind of web algebra (as e.g. the one studied in \cite{MPT1}).

Note that \fullref{algorithm:cup-basis} 
heavily depends on the choice of an evaluation and it is 
already quite delicate to choose an evaluation 
such that one can control the structure constants 
within the multiplication. For general web algebras 
this is very complicated and basically unknown at the time of writing this paper, 
cf. \cite{Tu1}, \cite{Tu2}.
This makes the 
proof of the analog of \fullref{proposition:cats-are-equal-yes} 
(as given below)
much more elaborate for general web algebras.
\end{remark}

\subsubsection*{Presenting foam \texorpdfstring{$2$}{2}-categories}

Next, we prove \fullref{proposition:cats-are-equal-yes}. 

\begin{proof}[Proof of \fullref{proposition:cats-are-equal-yes}]
Similar to \cite[Proof of Proposition 2.43]{EST2}, we can define
the $2$-functor $\Upsilon$, which is given by sending a web $u$ 
to the $\webalg$-bimodule $\M(u)$.
Moreover, by following \cite[Proposition 2.43]{EST2}, one can see 
that $\Upsilon$ is bijective on objects, essential surjective on morphisms 
and faithful on $2$-morphisms.

To see fullness, fix two webs $u,v$. We need to compare the dimension $\dim(\twoHom_{\foamf}(u,v))$ 
with $\dim(\Hom_{\webalg}(\M(u),\M(v)))$. 
The former is easy to compute using bending, since we already know 
that it has a cup foam basis by 
\fullref{proposition:cup-foam-basis}. 
In order to compute $\dim(\Hom_{\webalg}(\M(u),\M(v)))$ we need to find the 
filtrations of the $\webalg$-bimodules $\M(u)$ and $\M(v)$ 
by simples. (Here we need $\K=\overline{\K}$.)

This is done as follows. 
By using the cup foam basis for $\M(u)$, we see that 
$\M(u)$ has one simple $\webalg$-sub-bimodule $L_1$ spanned by 
the cup foam basis element with a dot on each component 
corresponding to a circle in $\M(u)$ (called maximally dotted). 
Then $\M(u)/L_1$ has one $\webalg$-sub-bimodule given 
as the $\K$-linear span of all cup foam basis elements 
with one dot less than the maximally dotted cup foam basis element. 
Continue this way computes the filtration of $\M(u)$ by simple $\webalg$-bimodules.
The same works verbatim for $\M(v)$ which in the end shows that 
\[
\dim(\twoHom_{\foamf}(u,v))=\dim(\Hom_{\webalg}(\M(u),\M(v))).
\]
We already know faithfulness and $\Upsilon$ 
is, by birth, a structure preserving $2$-functor.
\end{proof}

\subsubsection*{Choosing a cup foam basis}

Up to this point, having some basis 
was enough. For all further 
applications, e.g. for computing the multiplication 
explicitly, we have to fix a basis.
That is what we are going to do next.

Note hereby, that, as illustrated in \fullref{example:fix-basis}, 
the cup foam basis algorithm depends on the choice of an evaluation. 
Hence, what we have to do is to choose an 
evaluation for every closed web $\closure{w}$. Then, by 
choosing to bend to the left as in \eqref{eq:bending}, 
we also get a fixed cup foam basis for $\twoHom_{\foamf}(u,v)$ 
for all webs $u,v$.

We start by giving an algorithm 
how to evaluate a fixed circle $C$ in a web $u$ 
into a web without ordinary edges. 
This depends on a choice of a point $i$ on $C$.

Before giving this algorithm, which we call the 
\textit{circle evaluation algorithm},
note that one is locally 
always in one of the following situations (cf. \fullref{remark:prefoams1b} 
and \eqref{eq:outgoing-pair-first}):
\begin{gather}\label{eq:outgoing-pair}
\begin{gathered}
\raisebox{.075cm}{\xy
(0,0)*{
\includegraphics[scale=1.2]{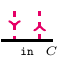}
\raisebox{.6cm}{\;,\;\;}
\includegraphics[scale=1.2]{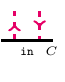}
};
(0,-6.5)*{\text{\small outgoing {\color{nonecolor}phantom}}};
(0,-9.5)*{\text{\small edge pairs}};
\endxy}
\;,\;
\xy
(0,0)*{
\includegraphics[scale=1.2]{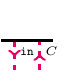}
};
\endxy
\;,\;
\xy
(0,0)*{
\includegraphics[scale=1.2]{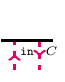}
};
\endxy
\\
\,
\xy
(0,0)*{
\includegraphics[scale=1.2]{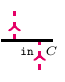}
};
\endxy
\;,\;
\xy
(0,0)*{
\includegraphics[scale=1.2]{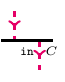}
};
\endxy
\;,\;
\xy
(0,0)*{
\includegraphics[scale=1.2]{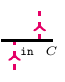}
};
\endxy
\;,\;
\xy
(0,0)*{
\includegraphics[scale=1.2]{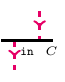}
};
\endxy
\end{gathered}
\end{gather}
Again, $\mathtt{in}$ denotes the interior 
of the circle $C$. 
The two leftmost situations are called 
\textit{outgoing phantom edge pairs}.
We say, such a pair is closest to the point $i$, or \textit{$i$-closest}, 
if it is the first 
such pair reading anticlockwise starting from $i$.

An $\varepsilon$-neighborhood $\clocal{C}$ of a circle
$C$ is called a \textit{local neighborhood} if $\clocal{C}$
contains the whole interior of $C$ and $\clocal{C}$ 
has no phantom loops in the exterior, e.g.
\begin{gather}\label{eq:hwcircle}
\xy
(0,0)*{
\includegraphics[scale=1.2]{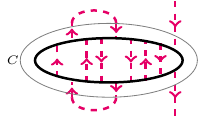}
\raisebox{.9cm}{$\rightsquigarrow$}
\raisebox{.35cm}{\includegraphics[scale=1.2]{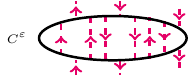}}
};
(1,-13.5)*{\text{\small using cookie-cutters}};
\endxy
\end{gather}

\begin{lemma}\label{lemma:hwcircle}
Let $C$ be a circle with a point $i$ on it.
There is a way to evaluate $\clocal{C}$ while 
keeping $i$ fixed till 
the end, using only \eqref{eq:square2} followed by \eqref{eq:digon2} 
to detach outgoing phantom edges, and 
removing all internal phantom edges 
using \eqref{eq:digon2} only.
\end{lemma}

\begin{proof}
Local situations of the following forms
\begin{gather*}
\xy
(0,0)*{
\includegraphics[scale=1.2]{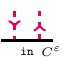}
};
\endxy
\overset{\eqref{eq:square1}}{\longmapsto}
\xy
(0,0)*{
\includegraphics[scale=1.2]{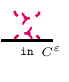}
};
\endxy
\overset{\eqref{eq:digon2}}{\longmapsto}
\xy
(0,0)*{
\includegraphics[scale=1.2]{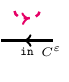}
};
\endxy
\\
\,
\xy
(0,0)*{
\includegraphics[scale=1.2]{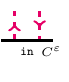}
};
\endxy
\overset{\eqref{eq:square1}}{\longmapsto}
\xy
(0,0)*{
\includegraphics[scale=1.2]{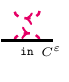}
};
\endxy
\overset{\eqref{eq:digon2}}{\longmapsto}
\xy
(0,0)*{
\includegraphics[scale=1.2]{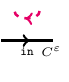}
};
\endxy
\end{gather*}
can always be simplified as indicated above. 
Thus, we can assume that 
$\clocal{C}$ does not have outgoing phantom edge pairs. 
But this means that $\clocal{C}$ is of the form as 
in \eqref{eq:hwcircle} (right side), 
which then can be evaluated recursively using \eqref{eq:digon2} only.
\end{proof}

To summarize, we have two basic situations for $\clocal{C}$'s:
\smallskip
\begin{enumerate}[label=(\Alph*)]

\setlength\itemsep{.15cm}

\item Outgoing phantom edge pairs, cf. \eqref{eq:outgoing-pair} (left two pictures).

\item Phantom digons, cf. \eqref{eq:digon2}.

\end{enumerate}
\smallskip
Now, the \textit{circle evaluation algorithm} is 
defined in \fullref{algorithm:circle-evaluation}.

\IncMargin{1em}
\begin{algorithm}
\SetKwData{Left}{left}\SetKwData{This}{this}\SetKwData{Up}{up}
\SetKwFunction{Union}{Union}\SetKwFunction{FindCompress}{FindCompress}
\SetKwInOut{Input}{input}\SetKwInOut{Output}{output}

 \Input{a circle $C$ in a web $u$ and a point $i$ on it\;}
 \Output{an evaluation $\phi=(\phi_1,\dots,\phi_r)$ of the circle $C$\;}
 \BlankLine
 \textit{initialization}; let $\phi=()$\;
 \While{$\clocal{C}$ contains two ordinary edges}{
  \eIf{$\clocal{C}$ contains an outgoing phantom edge pair}{
   apply \eqref{eq:square1} to the $i$-closest such pair\;
   add the corresponding relation among webs to $\phi$\;}
  {
   remove any phantom digon not containing $i$ using \eqref{eq:digon2}\;
   add the corresponding relation among webs to $\phi$\;}
 }
 remove the circle containing $i$ using \eqref{eq:circle1} and all phantom circles using \eqref{eq:circle2}\;
 add the corresponding relations among webs to $\phi$\;
 \BlankLine
 \caption{The circle evaluation algorithm.}\label{algorithm:circle-evaluation}
\end{algorithm}\DecMargin{1em}

\begin{lemma}\label{lemma:alg-terminates}
\fullref{algorithm:circle-evaluation} 
terminates and is well-defined.
\end{lemma}

\begin{proof}
That it terminates follows by its very definition via \fullref{lemma:hwcircle}. 

To see well-definedness, 
observe that the used 
phantom digon removals \eqref{eq:digon2} are far apart and hence, 
the corresponding foams realizing these commute by height reasons. 
Similarly, for the relations \eqref{eq:circle1} and \eqref{eq:circle2}.
This shows that the resulting evaluation foams are the same 
$2$-morphisms in $\foamf$.
\end{proof}

\makeautorefname{algorithm}{Algorithms}

Before we can finally define our choice of a cup foam basis, 
we need to piece \fullref{algorithm:cup-basis} 
and \ref{algorithm:circle-evaluation} together \makeautorefname{algorithm}{Algorithm} to the 
\textit{evaluation algorithm}, see \fullref{algorithm:evaluation}.

\IncMargin{1em}
\begin{algorithm}
\SetKwData{Left}{left}\SetKwData{This}{this}\SetKwData{Up}{up}
\SetKwFunction{Union}{Union}\SetKwFunction{FindCompress}{FindCompress}
\SetKwInOut{Input}{input}\SetKwInOut{Output}{output}

 \Input{a closed web $\closure{w}$ and a fixed point on each of its circles\;}
 \Output{an evaluation $\phi=(\phi_1,\dots,\phi_r)$ of $\closure{w}$\;}
 \BlankLine
 \textit{initialization}; let $\phi=()$\;
 \While{$\closure{w}$ contains a circle}{
 \eIf{$C$ does not contain a nested circle}{
 take the circle $C$ with its fixed point and apply \fullref{algorithm:circle-evaluation}\;
 add the result to $\phi$\;
 remove $C$ from $\closure{w}$\;}
 {
 remove all remaining phantom circles using \eqref{eq:circle2}\; 
 add the corresponding relation among webs to $\phi$\;}
 }
 \BlankLine
 \caption{The evaluation algorithm.}\label{algorithm:evaluation}
\end{algorithm}\DecMargin{1em}

\begin{lemma}\label{lemma:alg-terminates-2}
\fullref{algorithm:evaluation} terminates and is well-defined.
\end{lemma}

\begin{proof}
That the algorithm terminates 
is clear. That it is well-defined 
(i.e. that the resulting evaluation foams do not depend on the choice 
of which circles are taken first to be evaluated) 
follows because of the \textit{cookie-cutter strategy} 
(cf. \fullref{example:cookie}) taken within the algorithm 
which ensures that the resulting foam parts are far apart 
and thus, height commute.
\end{proof}

Armed with these notions, we are ready to fix a cup foam basis.

\begin{definition}\label{definition:fixed-cup-foam}
For any closed web $\closure{w}$ 
together with a fixed choice of a base point for each of its circles, 
we define the cup foam basis $\Cupbasis{\closure{w}}$ attached to 
it to be the evaluation foams turning up 
by applying \fullref{algorithm:cup-basis} 
to the evaluation of $\closure{w}$ obtained by
applying \fullref{algorithm:evaluation} to $\closure{w}$.
More generally, by choosing to bend to the 
left as in \eqref{eq:bending}, we also fix a cup foam basis $\Cupbbasis{u}{v}$ 
for any two webs $u,v$.
\end{definition}

\makeautorefname{lemma}{Lemmas}

Note that, by \fullref{lemma:alg-terminates} and \ref{lemma:alg-terminates-2}, 
the notion of $\Cupbasis{\closure{w}}$ is well-defined, while 
\fullref{proposition:cup-foam-basis} guarantees 
that $\Cupbasis{\closure{w}}$ is a basis 
of $\webhom{\closure{w}}=\twoHom_{\foamf}(\emptyweb,\closure{w})$.

\makeautorefname{lemma}{Lemma}

\begin{example}\label{example:cup-foam-basis-choice}
Depending on the choice of a base point, 
the cup foam basis attached to the local situation as 
in \fullref{example:fix-basis} gives either 
of the two results.
\end{example}

\begin{example}\label{example:cookie}
Our construction follows 
a \textit{cookie-cutter strategy}:
\begin{gather*}
\xy
(0,0)*{
\includegraphics[scale=1.2]{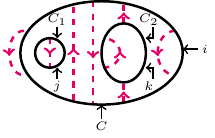}
};
\endxy
\rightsquigarrow
\\
\,
\xy
(0,0)*{
\xy
(0,0)*{
\includegraphics[scale=1.2]{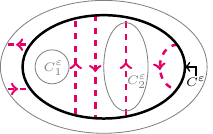}
};
\endxy
};
(-3.5,15)*{\phantom{\text{\small second inner cookie}}};
(0,-15)*{\text{\small outer cookie}};
\endxy
\quad,\quad
\xy
(0,0)*{
\xy
(0,0)*{
\includegraphics[scale=1.2]{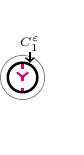}
};
\endxy
};
(-3.5,15)*{\phantom{\text{\small second inner cookie}}};
(-3.5,-15)*{\text{\small first inner cookie}};
\endxy
\quad,\quad
\xy
(0,0)*{
\xy
(0,0)*{
\includegraphics[scale=1.2]{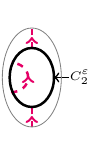}
};
\endxy
};
(-3.5,15)*{\phantom{\text{\small second inner cookie}}};
(-3.5,-15)*{\text{\small second inner cookie}};
\endxy
\end{gather*}
To this web the algorithm applies 
the \textit{cookie-cutter strategy} by first cutting out 
$\clocal{C_1}$ and $\clocal{C_2}$ and evaluate them 
using \fullref{algorithm:circle-evaluation}. 
(The resulting evaluation foams in the first 
case are as in \fullref{example:evaluation-first}; 
the reader is encouraged to work out the resulting evaluation foams in the second case.)
Then its cuts out $\clocal{C}$ (with $C_1$ and $C_2$ already removed) and 
applies \fullref{algorithm:circle-evaluation} again. The resulting cup foam 
basis elements are then obtained by piecing everything back together.
\end{example}

\subsection{Proof of \texorpdfstring{\fullref{theorem:comb-model}}{\ref{theorem:comb-model}}}\label{subsec:comb-model-proof}

The aim is to show that 
the combinatorial algebra defined in \fullref{sec:comb-model} 
gives a model for the web algebra. 
To this end, we follow the ideas from \cite[proof of Theorem 4.18]{EST1}, but 
carefully treat the more flexible situation we are in.

In particular, it is very important to keep in mind that we have fixed a 
cup foam basis, and we say a foam is (locally) of \textit{cup foam basis 
shape} if it is topologically as the corresponding 
foam showing up in our choice of the cup foam basis.

\subsubsection*{Simplifying foams}

First, we give three 
useful lemmas how to simplify foams. 
Before we state and prove these lemmas, we need some terminology.

Take a web $u$, a circle $C$ in it 
and a local neighborhood $\clocal{C}$ of it, and consider the identity 
foam in $\twoEnd_{\foamf}(u)$. Then 
$\clocal{C}\times[-1,+1]$ is called a \textit{singular cylinder}.
Blueprint examples are the foams in \eqref{eq:usual3} 
or \eqref{eq:neck} (seeing bottom/top as webs 
containing $\clocal{C}$), but 
also the situation in \eqref{eq:gen-cylinder}.

Similarly, a \textit{singular sphere} in a foam is a part of it that is 
a sphere after removing all phantom edges/facets, 
cf. \eqref{eq:usual1}, \eqref{eq:s-sphere} or \eqref{eq:gen-s-sphere}.

\begin{gather}\label{eq:gen-cylinder}
\xy
(0,0)*{
\raisebox{-1.2cm}{\includegraphics[scale=1.2]{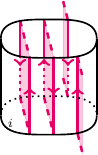}}
\in\twoEnd_{\foamf}\left(
\raisebox{-.5cm}{\includegraphics[scale=1.2]{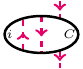}}
\right)};
(0,17)*{\text{\small\phantom{{\color{nonecolor}singular} cylinder}}};
(0,-17)*{\text{\small {\color{nonecolor}singular} cylinder}};
\endxy
\end{gather}
\begin{gather}\label{eq:gen-s-sphere}
\xy
(0,0)*{
\raisebox{-1.05cm}{\includegraphics[scale=1.2]{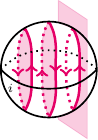}}
\leftrightsquigarrow
\raisebox{-.5cm}{\includegraphics[scale=1.2]{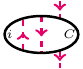}}
};
(0,17)*{\text{\small\phantom{{\color{nonecolor}singular} sphere}}};
(0,-17)*{\text{\small {\color{nonecolor}singular} sphere}};
\endxy
\end{gather}

\makeautorefname{lemma}{Lemmas}

Next, the local situation \eqref{eq:gen-s-sphere} 
has an associated web 
with an associated circle given 
by cutting the pictures 
in halves. 
(This is exemplified in \eqref{eq:gen-s-sphere}, i.e. cutting 
the singular sphere around the equator gives the web on the right side.)
Hence, from the bottom/top web 
for singular cylinders, and the webs associated to singular spheres 
we obtain the numbers
as defined in \fullref{subsec:comb-notions}. 
Hereby we use the points indicated above, which 
we also fix for \fullref{lemma:sing-cylinder} 
and \ref{lemma:s-sphere}.

\makeautorefname{lemma}{Lemma}

Now we can state the three main lemmas on our way 
to prove \fullref{theorem:comb-model}, 
namely the signs turning up by 
simplifying singular cylinders and spheres. 
For short, we say $\placeholder$-facets for foam facets touching 
the web segments containing the point $\placeholderout$.

\begin{lemma}\label{lemma:sing-cylinder}
Given a singular cylinder. Then we can simplify it to
\[
(-1)^{\dist(C,i)}\cdot\left(\,
\xy
(0,0)*{
\includegraphics[scale=1.2]{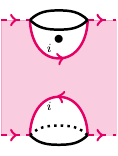}
};
\endxy
\,+\,
\xy
(0,0)*{
\includegraphics[scale=1.2]{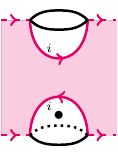}
};
\endxy
\,\right)
\]
(There might be more or fewer 
attached phantom edges/facets as well - depending on the starting 
configuration.)
Both dots sit on the $i$-facets and the coefficients 
are obtained from the associated circle $C$ and base point $i$ on it
in the bottom/top web, and the cup and cap are in cup foam 
basis shape in case $i=\opoint$.
\end{lemma}

\begin{proof}
This follows by a recursive squeezing procedure lowering the number 
of trivalent vertices attached to the circle $C$ in question.

This recursive squeezing procedure should be read as starting from bottom/top 
of the singular cylinder, applying some foam 
relations giving a thinner singular cylinder on the next level 
of the recursion until one ends with a usual cylinder which we can cut 
using \eqref{eq:usual3}. (The pictures to keep in mind are \eqref{eq:squeeze} 
and \eqref{eq:neck-fancy}.)

The main technical point  
is that we want to end with 
a cup and a cap of cup foam basis shape 
with respect to the point $\opoint$. Thus, the squeezing process 
depends on the particular way how to squeeze 
the singular cylinder.

Luckily, an easy trick enables us to always end up with 
cup foam basis shapes with respect to the point $\opoint$. Namely, 
we squeeze the cylinder 
by first evaluating the bottom circle using \fullref{algorithm:circle-evaluation}
and then antievaluate the result by reading 
\fullref{algorithm:circle-evaluation} backwards. We obtain 
from this a sequence of relations among webs and a foam lifting them 
which correspond to a situation as in the lemma:
\[
(\phi_1,\dots,\phi_r,\phi^{-1}_r,\dots,\phi^{-1}_1)
\leftrightsquigarrow
f\in\twoEnd_{\foamf}(\clocal{C}).
\]
(Here $\phi^{-1}_k$ means the other halves of the foams 
chosen in \fullref{subsec:cup-foam}.)
By construction, the foam $f$ is the identity and it remains 
to analyze the signs turning up by the 
foams lifting the concatenations $\phi_k\phi^{-1}_k$ of the relations among webs.

Now, \fullref{algorithm:circle-evaluation} 
gives us the following:
\smallskip
\begin{enumerate}[label=(\Alph*)]

\setlength\itemsep{.15cm}

\item Outgoing phantom edge pairs are squeezed using
\begin{gather}\label{eq:opep}
\begin{gathered}
\eqref{eq:phantom3}\text{ and }\eqref{eq:neck-fancy}\overset{\text{lift}}{\rightsquigarrow}
\xy
(0,0)*{
\includegraphics[scale=1.2]{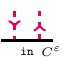}
};
\endxy
\longmapsto
\xy
(0,0)*{
\includegraphics[scale=1.2]{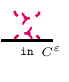}
};
\endxy
\longmapsto
\xy
(0,0)*{
\includegraphics[scale=1.2]{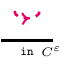}
};
\endxy
\\
\longmapsto
\xy
(0,0)*{
\includegraphics[scale=1.2]{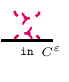}
};
\endxy
\longmapsto
\xy
(0,0)*{
\includegraphics[scale=1.2]{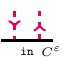}
};
\endxy
\end{gathered}
\end{gather}
(Or its reoriented version.)

\item Internal phantom loops are squeezed as in \eqref{eq:opep}, 
but using \eqref{eq:neck-fancy} only.

\end{enumerate}
\smallskip
Note that we only use \eqref{eq:neck-fancy}, which gives a minus or plus 
sign depending on the 
local situation (cf. \eqref{eq:digon2}), and \eqref{eq:phantom3}, 
which always gives a minus sign. 
Carefully keeping track of these signs (e.g. the sign 
turns around for outgoing edge pairs compared to internal phantom loops since we use 
both, \eqref{eq:neck-fancy} and \eqref{eq:phantom3}) 
shows that we get the claimed coefficients.

By construction, the dots sit at the $i$-facet 
(since the facet with the point $i$ is the last one to remove 
in \fullref{algorithm:circle-evaluation}), 
and the cup and cap are of cup foam basis 
shape in case $i=\opoint$. 
The lemma follows.
\end{proof}

\begin{lemma}\label{lemma:s-sphere}
Given a singular sphere with a dot sitting on some $i$-facet. 
Then this singular sphere evaluated to
\[
(-1)^{\dist(C,i)}.
\]
The coefficient 
is obtained from the associated 
web and its statistics. 
In case the singular sphere has not 
precisely one dot, then it evaluates to zero.
\end{lemma}

\begin{proof}
In fact, the steps for the evaluation of singular 
spheres are the inverses of the steps for recursively squeezing 
singular cylinder. Thus, the first statement follows, mutatis 
mutandis, as in \fullref{lemma:sing-cylinder}. 
The second statement is evident by the described 
squeezing procedure and \eqref{eq:usual1}.
\end{proof}

\begin{lemma}\label{lemma:all-the-same-up-to-sign}
In the setting of \fullref{lemma:sing-cylinder}: 
if $f_{\placeholder}$ denote the foams obtained by cutting the singular 
cylinder with respect to the points $\placeholderout\in\{i,j\}$, 
then
\[
f_i=(-1)^{\distt(i\rightarrow j)}\cdot f_j.
\]
(Note that $f_i$ and $f_j$ have their dots on different facets and are of different shape.)  
\end{lemma}

\begin{proof}
Take the 
foam $f_i$ and close its cap/cup at bottom/top such that the result 
are two singular spheres as in \fullref{lemma:s-sphere}, 
that is, with dots on the $i$-facets. \makeautorefname{lemma}{Lemmas}
Hence, by \fullref{lemma:sing-cylinder} 
and \ref{lemma:s-sphere}, the result is $+1$ times 
a foam which consists of parallel phantom facets only. 
Applying the same to $f_j$ also gives a foam which consists 
of parallel phantom facets only 
but with a different sign: 
The bottom/top singular sphere are topologically equal 
(not necessarily to the ones for $f_i$, but equal to each other), 
but they have a different dot placement. One of them 
has a dot on the $i$-facet, one of them on the $j$-facet.
Thus, after moving the dot from the $i$-facet to the $j$-facet 
(giving the claimed sign), the two created 
singular spheres can be evaluated in the same way 
and all other signs cancel. \makeautorefname{lemma}{Lemma}
\end{proof}

Next, a \textit{singular neck} is a local situation 
of the form
\begin{gather}\label{eq:gen-sing-neck}
\xy
(0,0)*{
\includegraphics[scale=1.2]{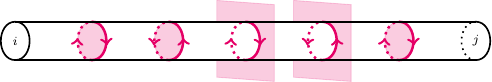}
};
(0,-12)*{\text{\small {\color{nonecolor}singular} neck}};
\endxy
\end{gather}
Again, for \eqref{eq:gen-sing-neck} one has 
an associated web, cup-cap-pair and points, and we get:

\begin{lemma}\label{lemma:sing-neck}
Given a singular neck. Then we can simplify it to
\[
(-1)^{\sdist(i)}\cdot
\xy
(0,0)*{
\includegraphics[scale=1.2]{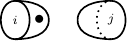}
};
\endxy
\,+\,
(-1)^{\stype}(-1)^{\sdist(i)}\cdot
\xy
(0,0)*{
\includegraphics[scale=1.2]{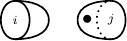}
};
\endxy
\]
(There might be phantom facets in between as well - depending on the starting 
configuration.)
The coefficients 
are obtained from the associated web. 
\end{lemma}

\begin{proof}
Assume that the singular neck has $n$ singular seams in total. By using 
neck cutting \eqref{eq:usual3} in between all of these (cutting to the left 
and the right of the two outermost singular seams as well) we obtain 
$2^{n+1}$ summands of the form
\[
\xy
(0,0)*{
\includegraphics[scale=1.2]{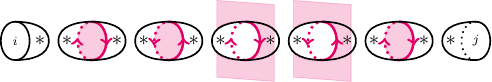}
};
\endxy
\]
with the $\ast$ indicating that there might be a dot.
By \eqref{eq:usual-rel-theta}, \eqref{eq:usual-rel-theta2}, \eqref{eq:s-sphere} 
as well as the second, reoriented, version of \eqref{eq:s-sphere} 
we see that all of them but two die. The two remaining 
summands have a dot on $i$ and $j$, 
respectively. 
The other dots coming from neck cutting \eqref{eq:usual3} for these two are
always placed on the opposite side of the singular seams 
in question (looking form $i$ respectively $j$). 
So we are left with the 
foam we want plus a bunch of dotted theta foams and dotted singular spheres.

Next, 
removing now the theta foams and the singular spheres
(using again the 
relations \eqref{eq:usual-rel-theta}, \eqref{eq:usual-rel-theta2}, \eqref{eq:s-sphere} 
as well as the second, reoriented, version of \eqref{eq:s-sphere})
gives signs depending on the orientations 
of the singular seams. In total, the sign for the $i$-dotted respectively 
$j$-dotted component 
is given by $(-1)^{\sdist(i)}$ respectively by $(-1)^{\sdist(j)}$. 
But, clearly, $\sdist(i)+\sdist(j)=\stype$.
\end{proof}

We stress that we abuse language: singular cylinders, spheres 
and necks might contain no phantom facets at all. 
The above lemmas still work and all appearing coefficients 
are $+1$.

\subsubsection*{The combinatorics of the multiplication}
First, we complete the definition 
of dotted basis webs. This is easy: copy almost word-by-word 
\fullref{algorithm:evaluation} and then 
\fullref{definition:fixed-cup-foam}. 
The resulting dotted basis webs correspond to our choice 
of cup foam basis from \fullref{definition:fixed-cup-foam}.
Now we prove \fullref{theorem:comb-model}:

\begin{proof}[Proof of \fullref{theorem:comb-model}]
First note that the $\K$-linear maps 
$\isocomb_u^v$ defined in 
\fullref{definition:iso-comb} are 
isomorphisms of $\K$-vector spaces because dotted 
basis webs of shape 
$uv^\ast$ are clearly in bijection with the cup foam basis 
elements in $\Cupbasis{uv^\ast}$, 
and the latter is a basis of ${}_u(\webalg_{\word{k}})_v$.

These isomorphisms are homogeneous which 
basically follows by definition. That is, 
a cup foam basis element with some dots is, after forgetting phantom edges/facets, 
topologically just a bunch of dotted cups. Thus, by direct 
comparison of \eqref{definition:dotted-degree} and \eqref{eq:degree-closed}, 
we see that all these isomorphisms are homogeneous.

Hence, it remains to show that they 
intertwine the inductively given multiplication.
To this end, similar to \cite[Section 4.5]{EST1}, we distinguish some cases, 
with some new cases turning up due to our more flexible setting:
\smallskip
\begin{enumerate}

\setlength\itemsep{.15cm}

\renewcommand{\theenumi}{(i)}
\renewcommand{\labelenumi}{\theenumi}

\item \label{enum:casei} \textbf{Non-nested merge}. Two non-nested components are merged.

\renewcommand{\theenumi}{(ii)}
\renewcommand{\labelenumi}{\theenumi}

\item \label{enum:caseii} \textbf{Nested merge}. Two nested components are merged.

\renewcommand{\theenumi}{(iii)}
\renewcommand{\labelenumi}{\theenumi}

\item \label{enum:caseiii} \textbf{Non-nested split}. One component splits into two non-nested components.

\renewcommand{\theenumi}{(iv)}
\renewcommand{\labelenumi}{\theenumi}

\item \label{enum:caseiv} \textbf{Nested split}. One component splits into two nested components.

\renewcommand{\theenumi}{(v)}
\renewcommand{\labelenumi}{\theenumi}

\item \label{enum:casev} \textbf{Phantom surgery}. We are in the phantom surgery situation.

\renewcommand{\theenumi}{(vi)}
\renewcommand{\labelenumi}{\theenumi}

\item \label{enum:casevi} \textbf{Turning inside out}. Reconnection of phantom seams.

\end{enumerate}
\smallskip
The cases \ref{enum:casei} to \ref{enum:caseiv} are 
the main cases, and we start with these. 
The other 
cases follow almost directly by construction 
(as we can see below).

We follow \cite[Proof of Theorem 4.18]{EST1} 
or \cite[Proof of Theorem 4.7]{EST2}: First, one observes that 
all components of the webs which are not involved in the 
multiplication step under consideration can be moved far away 
(and, consequently, can be ignored). Second, 
there are three circles involved in the multiplication. After 
the multiplication process the resulting foam might not be 
of the topological form of a basis cup foam and some non-trivial 
manipulation has to be done:
\smallskip
\begin{enumerate}

\setlength\itemsep{.15cm}

\renewcommand{\theenumi}{(I)}
\renewcommand{\labelenumi}{\theenumi}

\item \label{enum:caseI} In all of the main cases, it might be necessary to move 
existing or newly created dots to the adjacent facets of the chosen base points.

\renewcommand{\theenumi}{(II)}
\renewcommand{\labelenumi}{\theenumi}

\item \label{enum:caseII} The sign $(-1)^{\apc(\closure{W})}$ only appears in the nested split case 
and comes precisely as stated.

\renewcommand{\theenumi}{(III)}
\renewcommand{\labelenumi}{\theenumi}

\item \label{enum:caseIII} In all of the main cases, we 
cut (one or two) singular cylinders and remove 
(one or two) singular spheres.

\end{enumerate}
\smallskip
Note that the manipulation that we need to do in \ref{enum:caseI} is, 
on the side of foams, given by the dot moving 
relations \eqref{eq:dotmigration}. 
Clearly, these are combinatorially modeled by 
the (old and new) dot moving signs, and we ignore these in the following.

Regarding \ref{enum:caseII}: Phantom circles correspond 
to singular phantom cups which one creates at the bottom 
of a cup foam basis element and needs to be removed. 
By \eqref{eq:s-sphere-fancy} and its reoriented counterpart, 
we see that only anticlockwise oriented phantom seams contribute 
while removing it, giving precisely $(-1)^{\apc(\closure{W})}$. 
This sign can only turn up for the 
nested split, since the corresponding phantom circles have to in the inside 
of the circle resulting from the surgery (which rules out the non-nested cases 
as well as the nested merge).

Note also that
the operation from \ref{enum:caseIII} is 
more complicated than the corresponding ones in \cite[Proof of Theorem 4.18]{EST1} 
or \cite[Proof of Theorem 4.7]{EST2}, but
ensures that the resulting foam is of cup foam basis shape.

Hence, it remains to analyze what happens case-by-case. 
The procedure 
we are going to describe in detail 
is always basically be the same for all cases. Namely, in order to 
ensure that the result is of cup foam basis shape, 
we cut singular cylinders 
which correspond 
to circles after the surgery in the way described in \fullref{lemma:sing-cylinder}. 
Since we started already with a foam which is of 
cup foam basis shape, this creates singular spheres 
corresponding 
to circles before the surgery. We call both of these simplification moves.
The total sign depends on the difference 
between the signs picked up from the simplification moves.
\newline

\noindent
\textbf{Non-nested merge.} 
Here the picture (for arbitrary attached phantom edges, topological situations and orientations):
\begin{gather}\label{eq:cases-comb-iso-non-merge}
\scalebox{.85}{$
\xy
\xymatrix{
\xy
(0,0)*{
\includegraphics[scale=1.2]{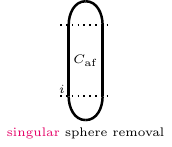}
};
\endxy
&
\xy
(0,0)*{
\includegraphics[scale=1.2]{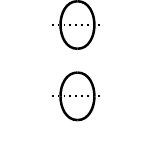}
};
\endxy
\ar[l]_{\text{flatten}}
\ar[r]^{\hspace*{-.35cm}\text{surgery}}
&
\xy
(0,0)*{
\includegraphics[scale=1.2]{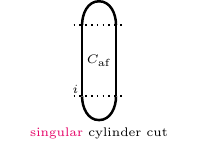}
};
\endxy
}
\endxy$}
\end{gather}
Above in \eqref{eq:cases-comb-iso-non-merge},
we have illustrated the circles (and points) where we 
perform the simplification moves. 
Note hereby that we can flatten the 
singular saddle and the singular sphere removal actually 
takes place in the left picture in \eqref{eq:cases-comb-iso-non-merge}.

\makeautorefname{lemma}{Lemmas}

One now directly observes that both simplification moves 
can be performed with respect to the same circle $C_{\mathrm{af}}$ 
and point $i$ on it.
Thus, in this case, by \fullref{lemma:sing-cylinder} 
and \ref{lemma:s-sphere}, all obtained signs cancel and we are left 
with no signs at all (as claimed). \makeautorefname{lemma}{Lemma}
\closeqed\newline

\noindent
\textbf{Nested merge.} 
The picture is as follows
(again, for arbitrary attached phantom edges, topological situations and orientations. 
Note hereby that we can again flatten the situation (also vice versa as in the 
non-nested merge case) because we can grab the bottom of $C_{\mathrm{in}}$ in 
the created singular sphere and pull it 
straight to the top, and we get:
\begin{gather}\label{eq:cases-comb-iso-merge}
\scalebox{.875}{$
\xy
\xymatrix{
\hspace*{-.2cm}
\xy
(0,0)*{
\includegraphics[scale=1.2]{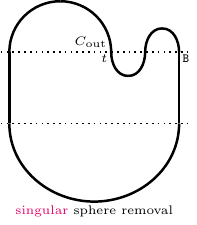}
};
\endxy
\hspace*{-.05cm}
&\hspace*{-.05cm}
\xy
(0,0)*{
\includegraphics[scale=1.2]{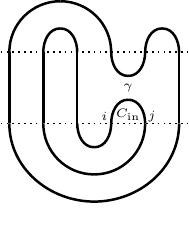}
};
\endxy
\hspace*{-.05cm}
\ar[l]_{\text{flatten}}
\ar[r]^{\hspace*{.05cm}\text{surgery}}
&
\hspace*{-.01cm}
\xy
(0,0)*{
\includegraphics[scale=1.2]{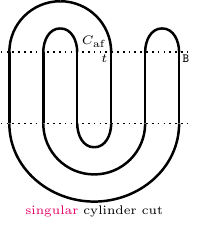}
};
\endxy
\hspace*{-.2cm}
}
\endxy$}
\end{gather}
(For later use, we have also illustrated the 
circle $C_{\mathrm{in}}$ and point $i$ for which 
we read off the sign in the combinatorial model, as well as 
the cup-cap pair $\gamma$ of the surgery and another point $t$ 
which play a role.) \makeautorefname{lemma}{Lemmas}
Thus, by using \fullref{lemma:sing-cylinder} 
and \ref{lemma:s-sphere}, we end up with the 
sign \makeautorefname{lemma}{Lemma}
\begin{gather}\label{eq:cases-comb-iso-merge-sign}
(-1)^{\dist(C_{\mathrm{af}},\opoint)}(-1)^{\dist(C_{\mathrm{out}},\opoint)}.
\end{gather}

Hence, it remains to rewrite the sign from \eqref{eq:cases-comb-iso-merge-sign} 
in terms of the circle $C_{\mathrm{in}}$ and the chosen point $i$ 
as in \eqref{eq:cases-comb-iso-merge}.
\newline

\noindent\textit{Claim}\,: In the setup from \eqref{eq:cases-comb-iso-merge} 
one has
\begin{gather}\label{eq:cases-comb-iso-merge-next-sign}
\dist(C_{\mathrm{af}},t)+\stype
=\dist(C_{\mathrm{out}},t)+\dist(C_{\mathrm{in}},i)+\sdist,
\end{gather}
where $\stype$ and $\sdist$ are to be calculated with respect to the points $i,j$.\openqed
\newline

\noindent\textit{Proof of the claim}\,:
We prove the claim inductively, where the basic case with no 
phantom edges whatsoever is clear.

If we attach a phantom edge to the situation from \eqref{eq:cases-comb-iso-merge} 
which does not touch neither $C_{\mathrm{in}}$ nor $\gamma$, then, 
clearly, $\dist(C_{\mathrm{af}},t)$ changes in the same way as 
$\dist(C_{\mathrm{out}},t)$ does, while everything else stays the same. 
Hence, the equation \eqref{eq:cases-comb-iso-merge-next-sign} stays true.

Similarly, if we attach a phantom edge 
which does not touch neither $C_{\mathrm{out}}$ nor $\gamma$, then, 
$\dist(C_{\mathrm{af}},t)$ changes in the same way as 
$\dist(C_{\mathrm{in}},i)$ does. To see this, note 
the attached phantom edge is in the internal of $C_{\mathrm{af}}$ 
if and only if it is in the external of $C_{\mathrm{in}}$. If its in the 
internal of $C_{\mathrm{af}}$, then it 
is $t$-positive if and only if its corresponding outgoing
phantom edge pair for $C_{\mathrm{in}}$ is $i$-negative. 
Similarly, when its in the external of $C_{\mathrm{af}}$.
Everything else stays the same and  
thus, \eqref{eq:cases-comb-iso-merge-next-sign} stays true.

Next, if we attach a phantom edge touching $C_{\mathrm{out}}$ and $C_{\mathrm{in}}$, 
but not $\gamma$, then the 
equation \eqref{eq:cases-comb-iso-merge-next-sign} still stays true. 
To see this, note that such a phantom edge forms an outgoing pair 
for $C_{\mathrm{in}}$, but an internal phantom loop for $C_{\mathrm{out}}$ and 
two internal phantom loops for $C_{\mathrm{af}}$.
We observe that precisely one of the new internal phantom loops for $C_{\mathrm{af}}$ 
are counted since they come by splitting the new internal phantom loop of $C_{\mathrm{out}}$ 
into two pieces, one pointing into $C_{\mathrm{af}}$, one out. 
Hence, $\dist(C_{\mathrm{af}},t)$ always grows by one.
Because the new phantom loop for $C_{\mathrm{out}}$ is $t$-positive 
if and only if its corresponding outgoing phantom edge pair for $C_{\mathrm{in}}$
is $i$-negative, we see that either $\dist(C_{\mathrm{out}},t)$ or 
$\dist(C_{\mathrm{in}},i)$ grow by one.
In total, \eqref{eq:cases-comb-iso-merge-next-sign} stays true.

Last, the remaining case where $\gamma$ is affected can be, mutatis mutandis, 
treated as the preliminary case: Attaching a single phantom edge to $C_{\mathrm{out}}$, 
we have that either $\dist(C_{\mathrm{out}},t)$ 
or $\dist(C_{\mathrm{in}},i)$ gets one bigger, while $\stype$ always 
gets one bigger. The difference to the preliminary case is that 
$C_{\mathrm{af}}$ now only gets one new internal phantom loop which 
contributes to $\dist(C_{\mathrm{af}},t)$ if and only if 
it does not contribute to $\sdist$. 
Again, the equation \eqref{eq:cases-comb-iso-merge-next-sign} stays true.

Thus, the claim is proven.\closeqed
\newline

Now, by \fullref{lemma:all-the-same-up-to-sign}, we have
\begin{gather*}
(-1)^{\dist(C_{\mathrm{af}},\opoint)}=(-1)^{\distt(t\rightarrow\opoint)}
(-1)^{\dist(C_{\mathrm{af}},t)},
\\
(-1)^{\dist(C_{\mathrm{out}},\opoint)}=
(-1)^{\distt(t\rightarrow\opoint)}(-1)^{\dist(C_{\mathrm{out}},t)}.
\end{gather*}
Hence, by the above claim, we get the same signs on both sides.
Similar for the horizontal mirror of the situation from \eqref{eq:cases-comb-iso-merge}.
\closeqed\newline

\noindent
\textbf{Non-nested split.}
Now the situation looks as follows:
\begin{gather}\label{eq:cases-comb-iso-non-split}
\scalebox{.95}{$
\xy
\xymatrix{
\xy
(0,0)*{
\includegraphics[scale=1.2]{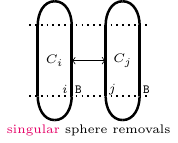}
};
\endxy
&
\xy
(0,0)*{
\includegraphics[scale=1.2]{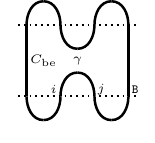}
};
\endxy
\ar[l]_{\;\;\text{neck}}
\ar[r]^{\text{surgery}}
&
\xy
(0,0)*{
\includegraphics[scale=1.2]{figs/6-72.pdf}
};
\endxy
}
\endxy$}
\end{gather}
(Again, \eqref{eq:cases-comb-iso-non-split} should be seen 
as a dummy for the general case.)
In contrast to the merges, 
we can not flatten the picture since there 
is a singular neck appearing around $\gamma$, and the 
singular sphere has to be removed in the leftmost 
situation in \eqref{eq:cases-comb-iso-non-split} 
by taking the singular neck into account using \fullref{lemma:sing-neck}.

\makeautorefname{lemma}{Lemmas}

Moreover, we perform two singular cylinder cuts of which 
precisely two summands survive. Namely one with the 
dot on the $i$-facet, one with the 
dot on the $j$-facet. 
Moreover, the singular sphere in these two cases has 
its dot on the $j$-facet respectively on the $i$-facet. 
By \fullref{lemma:sing-cylinder} and \ref{lemma:s-sphere}
we obtain the two signs: \makeautorefname{lemma}{Lemma}
\begin{gather}\label{eq:comb-iso-sign-non-split}
\begin{aligned}
&(-1)^{\dist(C_{\mathrm{be}},j)}(-1)^{\dist(C_i,i)}(-1)^{\dist(C_j,j)},\\
&(-1)^{\dist(C_{\mathrm{be}},i)}(-1)^{\dist(C_i,i)}(-1)^{\dist(C_j,j)}.
\end{aligned}
\end{gather}
In fact, by cutting the singular neck around $\gamma$ we can rewrite
\begin{gather*}
\begin{aligned}
&(-1)^{\dist(C_{\mathrm{be}},j)}=(-1)^{\sdist(i)}(-1)^{\dist(C_i,i)}(-1)^{\dist(C_j,j)},\\
&(-1)^{\dist(C_{\mathrm{be}},i)}=(-1)^{\stype}(-1)^{\sdist(i)}(-1)^{\dist(C_i,i)}(-1)^{\dist(C_j,j)}.
\end{aligned}
\end{gather*}
Now, rewriting 
this in terms of the basis points (using \fullref{lemma:all-the-same-up-to-sign}), 
and putting it together with \eqref{eq:comb-iso-sign-non-split}, shows 
that this case works as claimed.\closeqed\newline

\noindent
\textbf{Nested split.}
The $\caseC$ shape is (with flatten as for the non-nested merge):
\begin{gather}\label{eq:cases-comb-iso}
\scalebox{.95}{$
\xy
\xymatrix{
\xy
(0,0)*{
\includegraphics[scale=1.2]{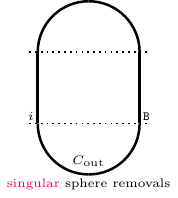}
};
\endxy
&
\xy
(0,0)*{
\includegraphics[scale=1.2]{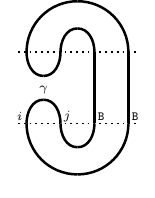}
};
\endxy
\ar[l]_{\text{flatten}}
\ar[r]^{\text{surgery}}
&
\xy
(0,0)*{
\includegraphics[scale=1.2]{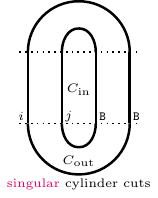}
};
\endxy
}
\endxy$}
\end{gather}
Here we use a notation close to the one from the nested merge and the non-nested split. 
The difference to the nested merge is that our task is easier now. \makeautorefname{lemma}{Lemmas}
In fact, by \fullref{lemma:sing-cylinder} and \ref{lemma:s-sphere} we are basically \makeautorefname{lemma}{Lemma}
done since the contributions of the $C_{\mathrm{out}}$ related 
simplifications almost cancel. 
(We also use \fullref{lemma:all-the-same-up-to-sign} to rewrite everything
in terms of the point $j$. Note also that $(-1)^{\distt(i\rightarrow j)}=(-1)^{\stype}$ 
with $\stype$ taken at $\gamma$.)
The only thing which we can 
not see in the leftmost picture in \eqref{eq:cases-comb-iso} 
are phantom circles which contribute the factor $(-1)^{\apc(\closure{W})}$. 
The case of a $\Ccase$ shape works similar and 
is omitted.\closeqed\newline

It remains to check cases \ref{enum:casev}
and \ref{enum:casevi}. 
Case \ref{enum:casev} is clear (the one special situation comes 
because we need to apply \eqref{eq:phantom3}).
In the remaining case \ref{enum:casevi} one would a priori expect some 
signs turning up, but these are already build into the four main cases.

The rest then works verbatim as in \cite[Proof of Theorem 4.18]{EST1}.
\end{proof}

\subsection{Proofs of \texorpdfstring{\fullref{proposition:sign-adjust}}{\ref{proposition:sign-adjust}} and \texorpdfstring{\fullref{theorem:top-model}}{\ref{theorem:top-model}}}\label{subsec:typeD-foams-proof}

Last, we prove our foamy realization of the type $\typeD$ arc algebra $\arcalg$.

\subsubsection*{Adjusting signs}
We need the following simple observations:

\begin{lemma}\label{lemma:simple-observations}
Let $\pathd{i}{j}$ be a cup or cap connecting $i$ and $j$, then 
\begin{gather}\label{eq:point_difference}
i \equiv (j +1) \; \mathrm{mod} \; 2.
\end{gather}
Further, we also have
\begin{gather}\label{eq:point_difference-2}
\length(\pathd{k}{l}) + 
\lengthf(\pathd{k}{l}) \equiv (k + l) \; \mathrm{mod} \; 2,
\end{gather}
where $k$ and $l$ are two points 
connected by 
a sequence $\pathd{k}{l}$ of cups and caps.
\end{lemma}

\begin{proof}
The equation \eqref{eq:point_difference} 
is evident, while \eqref{eq:point_difference-2}
follows by 
noting that summing up the length of all cups 
and caps in the sequence $\length_\Lambda$  
only contribute to the unmarked ones, while $\lengthf$ 
only contributes to the marked ones.
\end{proof}

\begin{proof}[Proof of \fullref{proposition:sign-adjust}]
The maps $\coeff_{D}$ from \fullref{definition:the-coefficient-map} are, 
by birth, homogeneous and $\K$-linear for all diagrams $D$.

Hence, as in the 
proof for \cite[Proposition 4.15]{EST2}, it remains to 
show that the maps $\coeff_{D}$ successively intertwine the 
two multiplication rules for $\arcalg$ and 
$\arcalgs$. 
Consequently, we compare two 
intermediate multiplications 
steps in the following fashion:
\begin{gather}\label{eq:the-usual-diagram}
\raisebox{.75cm}{\begin{xy}
  \xymatrix{
      D_{m} \ar[rr]^{ \boldsymbol{\mathrm{Mult}}^{\arcalg}_{D_m,D_{m+1}} } \ar[d]_{\coeff_{D_{m}}} & &  D_{m+1} \ar[d]^{\coeff_{D_{m+1}}}  \\
      D_{m} \ar[rr]_{\boldsymbol{\mathrm{Mult}}^{\arcalgs}_{D_m,D_{m+1}} }        &     &   D_{m+1},
  }
\end{xy}}
\end{gather}
where we denote by 
$\boldsymbol{\mathrm{Mult}}^{\arcalg}_{D_m,D_{m+1}}$ and 
$\boldsymbol{\mathrm{Mult}}^{\arcalgs}_{D_m,D_{m+1}}$ the 
surgery procedure rules as indicated 
in the two multiplications. Thus, the goal is 
to show that each such diagrams, i.e. 
for each appearing $D_{m}$ and $D_{m+1}$, commutes.

As usual, this is done by checking the four 
possible cases that appear in the surgery 
procedure. But before we start, note 
that \eqref{eq:point_difference-2} immediately 
implies
\begin{gather}\label{eq:point_difference-3}
(-1)^{\length(\pathd{k}{l})}=(-1)^{\lengthf(\pathd{k}{l})}\cdot
(-1)^{k}\cdot(-1)^{l},
\end{gather}
which we use throughout below.
\newline

\noindent
\textbf{Non-nested merge.} 
Assume that circles $C_b$ and $C_t$ are 
merged into a circle $C_{\mathrm{af}}$.
If both circles are oriented anticlockwise, 
then both multiplication rules yield a factor 
of $+1$ and $\coeff_{D_m}^{C_b} = 
\coeff_{D_m}^{C_t} = \coeff_{D_{m+1}}^{C_{\mathrm{af}}} = +1$ 
as well. The claim follows.

Assume now that the circle $C_{\placeholder}$ for $\placeholderout\in \{ b,t\}$ is 
oriented clockwise and the other circle is oriented 
anticlockwise. The multiplication in $\arcalg$ 
gives the factor $(-1)^{\length(\pathd{\rpoint_{\placeholder}}{\rpoint_{\mathrm{af}}})}$, 
while the multiplication in $\arcalgs$ yields 
$(-1)^{\lengthf(\pathd{\rpoint_{\placeholder}}{\rpoint_{\mathrm{af}}})}$. We check that
\begin{gather*}
\begin{aligned}
& \coeff_{D_{m+1}}^{C_{\mathrm{af}}} \cdot (-1)^{\length(\pathd{\rpoint_{\placeholder}}{\rpoint_{\mathrm{af}}})}
= -(-1)^{\rpoint_{\mathrm{af}}} \cdot (-1)^{\length(\pathd{\rpoint_{\placeholder}}{\rpoint_{\mathrm{af}}})}\\
\overset{\eqref{eq:point_difference-3}}{=}& -(-1)^{\rpoint_{\placeholder}} \cdot (-1)^{\lengthf(\pathd{\rpoint_{\placeholder}}{\rpoint_{\mathrm{af}}})} = \coeff_{D_m}^{C_b} \cdot \coeff_{D_m}^{C_t} \cdot (-1)^{\lengthf(\pathd{\rpoint_{\placeholder}}{\rpoint_{\mathrm{af}}})},
\end{aligned}
\end{gather*}
which proves the claim in this case.

If both circles are oriented clockwise, then both multiplications are zero.\closeqed\newline

\noindent
\textbf{Nested merge.} 
Due to the definition of $\arcalgs$, the signs in the nested 
merge are exactly as in the non-nested case. 
Thus, it is verbatim as the non-nested merge.\closeqed\newline

\noindent
\textbf{Non-nested split.} 
Assume that a circle $C_{\mathrm{be}}$ is split into 
circles $C_i$ and $C_j$ at a cup-cap pair connecting $i$ and $j$.

Assume first that $C_{\mathrm{be}}$ is oriented anticlockwise. 
Note that, by admissibility, it must hold that $\unmarked = 1$.
Hence, the summand where $C_i$ is oriented clockwise and $C_j$ 
is oriented anticlockwise obtains a factor 
$(-1)^{\length(\pathd{i}{\rpoint_i})}(-1)(-1)^{i}$ for $\arcalg$. 
In contrast, the summand only gains the factor $(-1)^{\lengthf(\pathd{i}{\rpoint_i})}$ 
in $\arcalgs$. Thus, we check
\begin{gather*}
\begin{aligned}
& \coeff_{D_{m+1}}^{C_{i}} \cdot \coeff_{D_{m+1}}^{C_{j}} \cdot 
(-1)^{\length(\pathd{i}{\rpoint_i})} \cdot(-1)\cdot (-1)^{i}\\
=& -(-1)^{\rpoint_{i}} \cdot (-1)^{\length(\pathd{i}{\rpoint_i})}\cdot (-1)\cdot(-1)^{i}\\
\overset{\eqref{eq:point_difference-3}}{=}& (-1)^{\lengthf(\pathd{i}{\rpoint_i})} = \coeff_{D_{m}}^{C_{\mathrm{be}}} \cdot (-1)^{\lengthf(\pathd{i}{\rpoint_i})},
\end{aligned}
\end{gather*}
which proves the claim. The second summand is done completely 
analogous by using \eqref{eq:point_difference} to see that 
the factor in $\arcalg$ is equal to 
$(-1)(-1)^{\length(\pathd{j}{\rpoint_j})}(-1)^{j}$.

%Assume now that $C_{\mathrm{be}}$ is oriented clockwise. 
%In this case the factor with respect to the 
%multiplication in $\arcalg$ is equal to 
%\begin{gather*}
%(-1)^{\pos(i)} (-1)^{\unmarked} (-1)^{\length(\pathd{i}{\rpoint_i})}(-1)^{\length(\pathd{\rpoint_{\mathrm{be}}}{\rpoint_j})},
%\end{gather*}
%while the factor in $\arcalgs$ is equal to $(-1)^{\lengthf(\pathd{i}{\rpoint_i})}(-1)^{\lengthf(\pathd{\rpoint_{\mathrm{be}}}{\rpoint_j})}$. Noting that in this case one has $\coeff_{D_{m+1}}^{C_{j}} = -(-1)^{\pos(\rpoint_{j})}$ and $\coeff_{D_{m}}^{C_{\mathrm{be}}} = -(-1)^{\pos(\rpoint_{\mathrm{be}})}$, the claim immediately follows as in the anticlockwise case above.

The clockwise case follows, mutatis mutandis, as the 
anticlockwise case by incorporating the two additional non-trivial coefficients.
\closeqed\newline

\noindent
\textbf{Nested split.} 
Assume the same setup as in the non-nested split case. Note that, 
due to the definition of the multiplication 
in $\arcalg$, we are always looking at the situation 
of the $\caseC$ shape here. Thus, if we assume that 
the circle $C_{\mathrm{be}}$ is oriented anticlockwise, 
the summand with $C_i$ oriented clockwise and $C_j$ oriented 
anticlockwise gains the factors 
$(-1)^{\length(\pathd{i}{\rpoint_i})}(-1)^{\unmarked}(-1)^{i}$ in $\arcalg$ 
and 
$(-1)^{\lengthf(\pathd{i}{\rpoint_i})}(-1)^{\unmarkedf}$ 
in $\arcalgs$. Since 
\begin{gather*}
\unmarked \equiv (\unmarkedf + 1) \; \mathrm{mod} \; 2
\end{gather*}
the claim follows by the same calculation as 
in the non-nested case. 

For the other summand there is no difference 
to the non-nested split, and the case of $C_{\mathrm{be}}$ being 
oriented clockwise is also derived analogously.\closeqed\newline

This in total proves the proposition.
\end{proof}

\subsubsection*{The embedding of the \texorpdfstring{$\typeD$}{D} arc algebras into the web algebras}

Recall that for the type $\typeD$ 
arc algebra the multiplication is zero in 
case the result is non-orientable, i.e. has an odd 
number of markers on some component, see \fullref{remark:non-orientable}. 
Hence, the first thing to make sure is that the isomorphism 
$\isotops$ preserves this. This is the purpose of the following definition 
and two lemmas.

\begin{definition}\label{definition:cup-orientation}
To a cup diagram $c$ we associate a web $\mathtt{u}(c)$
using the rule
\begin{gather*}
\,
\xy
(0,0)*{
\includegraphics[scale=1.2]{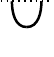}
};
\endxy
\mapsto
\xy
(0,0)*{
\includegraphics[scale=1.2]{figs/6-77.pdf}
};
\endxy
\\
\text{even case:}
\xy
(0,0)*{
\includegraphics[scale=1.2]{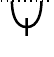}
};
\endxy
\mapsto
\xy
(0,0)*{
\includegraphics[scale=1.2]{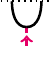}
};
\endxy
\quad,\quad
\text{odd case:}
\xy
(0,0)*{
\includegraphics[scale=1.2]{figs/6-78.pdf}
};
\endxy
\mapsto
\xy
(0,0)*{
\includegraphics[scale=1.2]{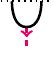}
};
\endxy
\end{gather*}
where we say a marked cup is even respectively odd
if it has an even respectively odd number of 
marked cups to its right.
\end{definition}

\begin{example}\label{example:cup-orientation}
Here is one blueprint example:
\begin{gather*}
\xy
(0,0)*{
\includegraphics[scale=1.2]{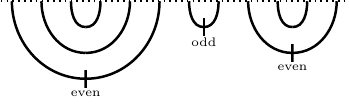}
};
\endxy
\\
\mapsto\,
\xy
(0,0)*{
\includegraphics[scale=1.2]{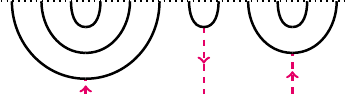}
};
\endxy
\end{gather*}
(Recall that we do not orient ordinary web components.)
\end{example}

\begin{lemma}\label{lemma:cup-orientation}
Given two cups diagrams $c,d$. Then
$cd^{\ast}$ is orientable if and only if 
$\mathtt{u}(c)\mathtt{u}(d)^{\ast}$ is a (well-oriented) web 
(cf. \fullref{convention:no-stupid-webs}).
\end{lemma}

\begin{proof}
If $\mathtt{u}(c)\mathtt{u}(d)^{\ast}$ is a web, 
then every face has an even number 
of adjacent trivalent vertices which translate to 
say that $cd^{\ast}$ has an even number of markers 
per connected component, and we are done.

Conversely, assume (without loss of generality) that 
$cd^{\ast}$ has only one circle $C$, and that $C$ is orientable. 
If $C$ is not marked, then 
we are done since the associated circle in $\mathtt{u}(c)\mathtt{u}(d)^{\ast}$ is 
an ordinary circle. Otherwise, follow $C$ from $\rpoint$ onwards in the 
anticlockwise fashion. 
By admissibility, going around $C$ in this way always passes the 
marked caps in $d^{\ast}$ and then the marked cups in $c$. This can be best seen via example 
(we leave it to the reader to make this rigorous):
\begin{gather}\label{eq:proof-by-example}
\xy
(0,0)*{
\includegraphics[scale=1.2]{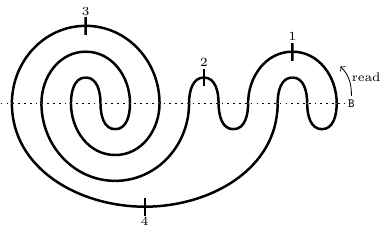}
};
\endxy
\end{gather}
Next, the number of markers on $C$ is even since 
$cd^{\ast}$ is orientable.  
This together with the above observation 
(and recalling that taking ${}^{\ast}$ on webs 
reverses the orientation of phantom edges) 
ensures that 
all 
neighboring phantom edge pairs of $\mathtt{u}(c)\mathtt{u}(d)^{\ast}$ 
are well-attached, and that $\mathtt{u}(c)\mathtt{u}(d)^{\ast}$ 
has an even number of trivalent vertices, i.e. 
$\mathtt{u}(c)\mathtt{u}(d)^{\ast}$ 
is well-oriented.
\end{proof}

\begin{lemma}\label{lemma:cup-orientation-match}
Given two cups diagrams $c,d$. Then, 
up to closing of the phantom edges, we have:
$\mathtt{u}(c)=u(c)$ and $\mathtt{u}(d)=u(d).$
(With $u(\placeholder)$ as in \eqref{eq:web-and-D}.)
\end{lemma}

\begin{proof}
By comparing \eqref{fig:cup-to-foams} and \eqref{eq:proof-by-example}.
\end{proof}

\begin{proof}[Proof of \fullref{theorem:top-model}]
By \fullref{theorem:comb-model}, 
\fullref{proposition:sign-adjust},
and summation, it remains to show 
that $\isotops_{\Lambda}$ is an embedding of graded algebras.

To this end, 
fix $\Lambda\in\X$ of rank $\brank$.
There are four things to be checked, where 
$c_b,d,d^{\prime},c_t$ are always 
cup diagrams of rank $\brank$ and $\lambda,\mu$ are in $\Lambda$:
\smallskip
\begin{enumerate}

\setlength\itemsep{.15cm}

\renewcommand{\theenumi}{(1)}
\renewcommand{\labelenumi}{\theenumi}

\item \label{enum:caseone} 
The $\K$-linear maps $\isotops_{c_b}^{c_t}$ are homogeneous embeddings of $\K$-vector spaces.

\renewcommand{\theenumi}{(2)}
\renewcommand{\labelenumi}{\theenumi}

\item \label{enum:casetwo} We 
have 
$\boldsymbol{\mathrm{Mult}}^{\arcalgs}(c_b\lambda d^{\ast},d^{\prime}\mu c_t^{\ast})=0$ 
because one has $d\neq d^{\prime}$ if and only if the multiplication
$\boldsymbol{\mathrm{Mult}}^{\cwebalg}(\isotops_{\Lambda}(c_b\lambda d^{\ast}),
\isotops_{\Lambda}(d\mu c_t^{\ast}))=0$ 
because $u(d)\neq u(d^{\prime})$.

\renewcommand{\theenumi}{(3)}
\renewcommand{\labelenumi}{\theenumi}

\item \label{enum:casethree} We 
have 
$\boldsymbol{\mathrm{Mult}}^{\arcalgs}(c_b\lambda d^{\ast},d\mu c_t^{\ast})=0$ 
because $c_bc_t^{\ast}$ is not orientable if and only if
$\boldsymbol{\mathrm{Mult}}^{\cwebalg}(\isotops_{\Lambda}(c_b\lambda d^{\ast}),
\isotops_{\Lambda}(d\mu c_t^{\ast}))=0$ 
because $u(c_b)u(c_t)^{\ast}$ is not a 
web.

\renewcommand{\theenumi}{(4)}
\renewcommand{\labelenumi}{\theenumi}

\item \label{enum:casefour} In case $d=d^{\prime}$ and $c_bc_t^{\ast}$ is orientable, 
the usual diagram 
(the one very similar to \eqref{eq:the-usual-diagram}, 
but with exchanged notation) commutes. 

\end{enumerate}
\smallskip
\noindent
\textbf{\ref{enum:caseone}.}
Note that \eqref{eq:degrees_cups} sums up to $0$ 
respectively $2$ 
for anticlockwise respectively clockwise circles.
Thus, $\isotops_{c_b}^{c_t}$ is homogeneous 
by comparing \eqref{eq:dotted-degree} and \eqref{eq:degrees_cups}, 
while keeping the shift $d(\word{k})$ in mind. 
That $\isotops_{c_b}^{c_t}$ is injective follows by definition.
\closeqed\newline

\makeautorefname{lemma}{Lemmas}

\noindent
\textbf{\ref{enum:casetwo}+\ref{enum:casethree}.} 
Directly from \fullref{lemma:cup-orientation} and \ref{lemma:cup-orientation-match}. \makeautorefname{lemma}{Lemma}
\closeqed\newline

\noindent
\textbf{\ref{enum:casefour}.} 
The signs for the multiplication for $\arcalgs$ 
from \eqref{eq:sign-adjust-one} and \eqref{eq:sign-adjust-two} 
are specializations of the ones for $\cwebalg$
for dotted basis webs of shape 
$u(c)u(d)^{\ast}$ 
(with $c,d$ standing for cup diagrams) - up to the 
phantom circle sign. 
Thus, it remains to show that the scaling factor $(-1)^{\upguys}$
accounts for this.
To this end,
one directly observes that only the phantom circle removal 
can change the number $\upguys$. Moreover, $\upguys$ is defined 
to count anticlockwise phantom circles, which 
is what $\apc(\closure{W})$ counts.
\closeqed\newline

Thus, the theorem is proven.
\end{proof}
%%%%%%%%%%%%%%%%%%%%%%%%%%%

%%%%%%%%%%%%%%%%%%%%%%%%%%%%%%%%%%%%%%%%
%%%                                  %%%
%%%            Bibliography          %%%
%%%                                  %%%
%%%%%%%%%%%%%%%%%%%%%%%%%%%%%%%%%%%%%%%%

%%%%%%%%%%%%%%%%%%%%%%%%%%%
\end{document}